\documentclass[11pt,reqno]{amsart}
\usepackage{amsmath,amssymb, amsfonts, amstext,amsthm, latexsym}
\pdfoutput=1
\allowdisplaybreaks
\DeclareMathOperator*{\esssup}{ess\,sup}

\newcommand{\cF}{{\cal F}}
\newcommand{\RR}{{\mathbb R}}

\newcommand{\EX}{{\mathbb E}}
\newcommand{\EE}{{\mathbb E}}

  \newcommand{\PX}{{\mathbb P}}
\newcommand{\PP}{{\mathbb P}}

\newcommand{\s}{\sigma}
\newcommand{\e}{\varepsilon}

\newcommand{\om}{\omega}
\newcommand{\Om}{\Omega}

\newcommand{\de}{\delta}

\renewcommand{\cF}{\mathcal F}

\numberwithin{equation}{section}
\newtheorem{theorem}{Theorem}[section]
\newtheorem{defn}[theorem]{Definition}

\newtheorem{lemma}[theorem]{Lemma}

\newtheorem{prop}[theorem]{Proposition}

\begin{document}
\title[ Stochastic anisotropic 3D Navier-Stokes]
{On stochastic  modified 3D Navier-Stokes %Brinkman-Forchheimer 
 \\ equations with anisotropic  viscosity}%: \\ 
%Well posedness  and large deviations}

%%%%%%%%%%%%%%%%%%%
\author[H. Bessaih]{Hakima Bessaih}
\address{University of Wyoming, Department of Mathematics, Dept. 3036, 1000
East University Avenue, Laramie WY 82071, United States}
\email{ bessaih@uwyo.edu}

\author[A. Millet]{ Annie Millet}
\address{SAMM, EA 4543,
Universit\'e Paris 1 Panth\'eon Sorbonne, 90 Rue de
Tolbiac, 75634 Paris Cedex France {\it and} Laboratoire de
Probabilit\'es et Mod\`eles Al\'eatoires,
  Universit\'es Paris~6-Paris~7} 
\email{amillet@univ-paris1.fr} 

%\author[H. Bessaih and A. Millet]
%{Hakima Bessaih  and Annie Millet}

%\address[H.~Bessaih]
%\address{University of Wyoming, Department of Mathematics, Dept. 3036, 1000
%East University Avenue, Laramie WY 82071, United States}
%\email{ bessaih@uwyo.edu}
%\\
%\\
%} \email[H.~Bessaih]{bessaih@uwyo.edu}

%\address[A.~Millet]
%{ SAMM , Universit\'{e} Paris 1,
% Centre Pierre Mend\`{e}s France,
%90 rue de Tolbiac, F- 75634 Paris Cedex 13, France {\it and}
%Laboratoire de Probabilit\'es et Mod\`eles Al\'eatoires (UMR 7599)}
% \email[A.~Millet]{amillet@univ-paris1.fr {\it and} annie.millet@upmc.fr}

\thanks{  Hakima Bessaih was partially supported by NSF grant DMS 1418838. }  



\begin{abstract}
Navier-Stokes equations in the whole space $\RR^3$ subject to an anisotropic viscosity and a random perturbation of multiplicative type is described.
 By adding a term of Brinkman-Forchheimer type to the model, existence and uniqueness of global weak solutions in the PDE sense are proved. 
 These are strong solutions in the probability sense. The Brinkman-Forchheirmer term
 provides some extra regularity in the space $L^{2\alpha+2}(\RR^3)$, with $\alpha>1$. 
 As a consequence, the nonlinear term has better properties which allow to prove uniqueness.  
 The proof of existence is performed through a control method. 
 A Large Deviations Principle is  given and proven at the end of the paper.
\end{abstract}

\maketitle

%%%%%%%%%%%%%%%%%%%

%%%%%%%%%%%%%%%%%%%%%%%%%%%%%%%%%%%%%%%%%%%%%%%%%%%%%%%%
\section{Introduction}\label{s1} The Navier-Stokes equations describe the time evolution of the velocity $u$ of an incompressible
fluid  in a bounded or unbounded domain of $\mathbb{R}^n, \ n=2, 3$ and are described by:
\begin{align*}
&\partial_t u - \nu \, \Delta u +  u\, \cdot \nabla u  + \nabla p=0,\\
&{\rm div } u=0, \quad u|_{t=0}=u_0, 
\end{align*}
where $\nu>0$ is the viscosity of the fluid and $p$ denotes the pressure. If existence and uniqueness is known to hold in
dimension 2, the case of dimension 3 is still only partially solved. Indeed, there exists a solution in some homogeneous Sobolev
space $\dot{H}^{1/2}$ either on a small time interval or on an arbitrary time interval if the norm of the initial condition is small enough.
 The difficulty in dimension 3 comes from the nonlinear term $(u\, \cdot\, \nabla) u$ that  requires more regularity. 
 However, this regularity is not satisfied by the energy estimates while it is in dimension 2. 
 In particular, the lack of this regularity is essentially the reason the uniqueness  
 cannot  be proved for weak solutions. Many regularizations have been introduced to overcome this difficulty.  
 Here, we will discuss only two of them; a regularization by a rotating term $u\times e_{3}$ and a  regularization 
 by a Brickman-Forchheimer term $\big| u\big|^{2\alpha} \, u$. Of course these two different regularizations give rise to different models. 
 One is related to some rotating flows while the other is related to some porous media models.   
   We refer to \cite{MTT} and the references therein, where the following system has been investigated  (in an even more general
   formulation) 
\begin{align*}
&\partial_t u - \nu \, \Delta u + u\, \cdot \nabla u  + \nabla p+ a \big| u\big|^{2\alpha} \, u=f,\\
&{\rm div } u=0, \quad u|_{t=0}=u_0, 
\end{align*}
where $a>0$ and $\alpha>0$ and $f$ is an external force. 
Under some assumptions on the coefficient $\alpha$, the authors in \cite{MTT} prove the existence and uniqueness of global strong solutions.

A slightly different regularization has been investigated by Kalantarov and Zelik in \cite{KZ}; more precisely they studied 
 some versions of the following model:
 \begin{align*}
&\partial_t u - \nu \, \Delta u + u\, \cdot \nabla u  + g(u)+\nabla p=f , \\
&\mbox{\rm div }\; u=0, \quad u|_{t=0}=u_0,
\end{align*}
where $g\in C^2(\RR^3, \RR^3)$  satisfies the following properties:

\begin{eqnarray}\label{KZ}
\left\{ \begin{array}{ll}
g'(u)v\cdot v &\geq (-K+\kappa |u|^{r-1})|v|^2, \quad \forall u, v\in \RR^3, \\
|g'(u)| &\leq C(1+|u|^{r-1}), \quad \forall u\in \RR^3, 
\end{array}
\right. 
\end{eqnarray}
where $K, C, \kappa$ are some positive constants, $r\in [1, \infty)$  and $u\cdot v$ stands for the inner product in $\RR^3$.  
When the forcing is of random type, that is $f=\sigma(t,u) dW(t)$,  M.~R\"ockner,  T.~Zhang and X.~Zhang 
 tackled a stochastic version 
% associated to 
of a modification of the previous model \eqref{KZ}, 
that they called the  tamed stochastic Navier-Stokes equations,  in several  papers  such as  \cite{R_TZ_XZ}, and \cite{R_TZ}. 
%\begin{align*}
%&\partial_t u + \big[ - \nu \, \Delta u + (u\, \cdot\, \nabla) u + g_N(|u|^2) u +\nabla p\big] dt =\sigma(t,u) dW_t,\\
%&\mbox{\rm div }\; u=0, \quad u|_{t=0}=u_0,
%\end{align*}
%where $g_N(r)=(r-N)/\nu$ for $r>N+1$, $g_N(r)=0$ if $r\in [0,N]$ and $0\leq g_N'(r)\leq 2/(\nu \wedge 1)$.
 Let us mention that in both the deterministic and the stochastic versions of \eqref{KZ}, 
 the solutions are investigated when the regularity of initial condition is   at least $H^1$  and  the viscosity acts in all three directions. 

 \medskip
 
In this paper, we are interested in the 3D Navier-Stokes equations with anisotropic viscosity 
that is acting only in the horizontal directions.  
These models %are interesting and 
have some applications in  atmospheric dynamics where some informations are missing. 
The relevance of the anisotropic viscosity is explained through the Ekeman law (see e.g. \cite{Ped} or the
introduction of \cite{CDGG_2006}). 
 The aim of this paper is to study an anisotropic Navier-Stokes equation
in dimension 3  that is subject to some multiplicative  random forcing.
% with intensity that may depend on  the solution. 
More precisely,  we consider the following model of a modified 3D anisotropic Navier-Stokes system  on a fixed time interval $[0,T]$ which can
be written formally as follows: 
%for $(t,x)\in [0,T]\times \RR^3$:
\begin{align}
& \partial_t u  - \nu\,  \Delta_h  u + u\, \cdot \nabla u + a\,  \big| u\big|^{2\alpha} \, u + \nabla p  =  \sigma(t,u)\, \dot{W} 
 \quad \mbox{\rm for } \; (t,x)\in [0,T]\times \RR^3,  \label{saNS_p}\\
&\nabla\, \cdot \, u =0  \quad \mbox{\rm for } \; (t,x)\in [0,T]\times \RR^3, \nonumber 
\end{align}

with the initial condition $u_0$ independent of the driving noise $W$. Here the viscosity $\nu$ and the coefficient $a$ 
of the nonlinear convective term are  strictly positive, $\alpha >1$,
$\partial_t$ denotes the time partial derivative, $\Delta_h:= \partial_1^2 + \partial_2^2$ and $\partial_i$ denotes the partial derivative
in the direction $x_i$, $i=1,2,3$. Thus the viscosity is only smoothing in the horizontal directions.
 As usual the fluid is incompressible, $p$ denotes the pressure;  the forcing term 
 $\sigma(t,u)\, \dot{W}$ is a multiplicative  noise driven by an infinite dimensional Brownian motion $W$ which is white in
time with spatial correlation.   The convective term $a |u|^{2\alpha} u$ is of Brinkman-Forchheimer type and
has a regularizing effect which can balance on one hand the vertical partial derivative of the bilinear term to prove existence, and on the other hand
provide some control to obtain uniqueness. Note that the space $L^{2\alpha +2}(\RR^3)$ appears naturally in the analysis of \eqref{saNS_p}; it is equal
to $L^4(\RR^3)$ if $\alpha=1$. Furthermore, the homogeneous critical Sobolev space $\dot{H}^{1/2}$ for the Navier-Stokes equation is included in
$L^4$. Hence it is natural to impose $\alpha >1$.

The deterministic counterpart of \eqref{saNS_p}, that is equation \eqref{saNS_p} with $\sigma=0$, has been studied by H.~Bessaih, S.~Trabelsi and
H.~Zorgati in \cite{BTZ}. The authors have proved that if the initial condition $u_0\in \tilde{H}^{0,1}$, for any $T>0$ there exists a unique solution in
$L^\infty(0,T; \tilde{H}^{0,1}) \cap L^2(0,T;\tilde{H}^{1,1})$ which belongs to $C([0,T],L^2)$, for some anisotropic Sobolev spaces which will
be defined in the next section (see \eqref{Hss'}). We generalize this result by allowing the system to be subject to some random external force
whose intensity may depend on the solution $u$ and  on   its horizontal gradient $\nabla_h u$. Note that since no smoothing is provided by
a viscosity in the vertical direction, in the anisotropic case, one requires that the initial condition $u_0$ is square
integrable as well as its vertical partial derivative. 
\smallskip

In  the deterministic setting (that is $\sigma=0$), replacing the Brinkman-Forchheimer  term $a|u|^{2\alpha} u$ by 
the rotating term $\frac{1}{\epsilon} u\times e_3$, J.Y.~Chemin, B.~Desjardin, I.~Gallagher and E.~Grenier \cite{CDGG_2000} have  studied an anisotropic
modified Navier Stokes equation on $\RR^3$ with a vertical viscosity $\nu_v \geq 0$, which is allowed to vanish. Using some homogeneous
anisotropic spaces, they have proved that if $u_0\in H^{0,s}$ with $s>\frac{3}{2}$, there exists $\epsilon_0$ depending only on $\nu$ and $u_0$
such that for $\epsilon\in (0, \epsilon_0]$,
\begin{align*}
& \partial_t u  - \nu\,  \Delta_h  u + u\, \cdot \nabla u + \frac{1}{\epsilon}  u \times e_3 + \nabla p  =  0, 
 \quad \mbox{\rm for } \; (t,x)\in [0,T]\times \RR^3,  \\
&\nabla\, \cdot \, u =0  \quad \mbox{\rm for } \; (t,x)\in [0,T]\times \RR^3, \quad u|_{t=0}=u_0
\end{align*}
has a unique global solution in $L^\infty(0,T ; H^{0,s}) \cap L^2(0,T;H^{1,s})$. The dispersive Brickman-Forchheimer term is "larger" than the rotating term used in \cite{CDGG_2000} but the regularity required on the initial condition is weaker and we allow a stochastic forcing term.
\smallskip

 The paper is organized as follows. In section \ref{functional} we describe the functional setting of our anisotropic model and prove some  
 technical properties of the deterministic terms. Several results were already proved in \cite{BTZ} and we sketch the arguments for the sake
 of completeness.  We also describe the random forcing term and the growth and Lipschitz assumptions on the diffusion coefficient  $\sigma$. 
 In section 3 we prove that if $u_0\in L^4(\Omega,\tilde{H}^{0,1})$ is independent of $W$ and $\sigma$ satisfies some general assumptions 
 (in particular cases $\sigma$ may contain some "small multiple" of the horizontal gradient
 $\nabla_h u$), equation \eqref{saNS_p} has a unique solution in $L^4(\Omega ; L^\infty(0,T; \tilde{H}^{0,1})) \cap L^2(\Omega;L^2(0,T:\tilde{H}^{1,1}))
 \cap L^{2\alpha+2}(\Omega\times (0,T)\times \RR^3)$, which is almost surely continuous from $[0,T]$ to $H$,
 where $H$ denotes the set of square integrable divergence free functions. Examples of such
 coefficients $\sigma$ are provided. Since we are working on the whole space $\RR^3$, and not on a bounded domain, the martingale approach
 used in  \cite{BM},  
  which depends on tightness properties, does not seem appropriate. We use instead the control method
 introduced in \cite{MS02} for the 2D Navier-Stokes equation; see also \cite{Sundar}, \cite{DM}, \cite{CM} and
  \cite{R_TZ}, where this method has been used for 
 the stochastic 2D Navier-Stokes equations, stochastic 2D  general hydrodynamical B\'enard models
  and the stochastic 3D tamed Navier-Stokes equations. 
In section 4,  under stronger assumptions on $\sigma$ (which may no longer depend on the horizontal gradient $\nabla_h u$), 
we also prove a large deviations
 result in $C([0,T];H)\cap L^2(0,T;\tilde{H}^{1,0})$ when the noise intensity is multiplied by a small parameter $\sqrt{\epsilon}$ converging to 0.
 The proof uses the weak-convergence approach introduced by A.~Budhiraja , P.~Dupuis and R.S.~Ellis in  \cite{DupEl97} and \cite{BD00}; 
 see also the references
 \cite{Sundar}, \cite{DM}, \cite{CM} and \cite{R_TZ} where this approach, based on the equivalence of the Large Deviations and Laplace principles,
  is used
 for various stochastic 2D Hydrodynamical models and the stochastic 3D tamed Navier-Stokes equation. For the sake of completeness, 
 some technical well-posedness  result for a stochastic controlled equation and estimates which only depend on the norm 
 stochastic control,  whose proofs are similar to that of
  the original equation in section 3, are given in the  appendix. The proof of the weak convergence and compactness arguments, which have
  also been used in some papers on Large Deviations Principles of stochastic hydrodynamical models, are also described in the appendix
\smallskip

\section{The functional setting}\label{functional} 
\subsection{Some notations} \label{notations}
Let us describe some further notations and the functional framework we will use throughout the paper.
Given a vector $x=(x_1,x_2,x_3)$ let $x_h:=(x_1,x_2)$ denote the horizontal variable, which does not play the same role as the
vertical variable $x_3$. 
Due to the anisotropic feature of the model,  we use anisotropic Sobolev spaces defined  as follows:
given $s,s'\in \RR$  let $H^{s,s'}$ denote the set of tempered distributions $\psi\in {\mathcal S}^{'}(\RR^3)$ such that
\begin{equation} \label{Hss'}
 \|\psi\|_{s,s'}^2:= \int_{\RR^3} \big( 1+\big|(\xi_1,\xi_2)\big|^{2s}\big)\, \big(1+\big|\xi_3\big|^{2s'}\big) \, \big| {\mathcal F} \psi(\xi)\big|^2\, d\xi <\infty ,
 \end{equation}
where ${\mathcal F}$ denotes the Fourier transform. The set $H^{s,s'}$ endowed with the norm $\|\, \cdot\, \|_{s,s'}$ is a Hilbert space. 

Set ${\rm div}_h u = \partial_1 u_1 + \partial_2 u_2$. Note that for $u\in \big( H^{1,0}\cap H^{0,1}\big)^3$ 
\begin{equation}  \label{divergence}
\nabla\, \cdot u =0 \quad \mbox{\rm implies} \quad {\rm div}_h u = - \partial_3 u_3.
\end{equation} 

For exponents $p,q\in [1,\infty)$ let $\|\, \cdot\, \|_p$ denote the $L^p(\RR^3)$ norm while $L^p_h(L^q_v)$  denotes
the space $L^p(\RR_{x_1}\, \RR_{x_2} , L^q(\RR_{x_3}))$   endowed with the norm
\[ \| \phi\|_{L^p_h(L^q_v)} := \Big\{ \int_{\RR^2} \Big( \int_{\RR} \big| \phi (x_h,x_3) \big|^q dx_3\Big)^{\frac{p}{q}} dx_h\Big\}^{\frac{1}{p}}.\]
The space  $L^q_v(L^p_h) = L^q(\RR_{x_3} ; L^p(\RR_{x_1}\, \RR_{x_2} ) )$ is defined in a similar way and endowed with the norm
$ \| \phi\|_{L^q_v(L^p_h)} := \big\{ \int_{\RR} \big( \int_{\RR^2} \big| \phi (x_h,x_3) \big|^p dx_h\big)^{\frac{q}{p}} dx_3\big\}^{\frac{1}{q}}$.
Note that in the above definitions we may assume that $p$ or $q$ is $\infty$  changing the norm accordingly.

Let $\mathcal{V}$ be the space of infinitely differentiable vector
fields $u$ on $\RR^3$ with compact support and
satisfying $\nabla\cdot u=0$.  Let us denote by $H$ the closure of $\mathcal{V}$
in $L^{2}(\RR^3; \mathbb{R}^3)$,   that is
$$H = \left\{u\in L^{2}({\mathbb R}^3;{\mathbb R}^3)\, : \,  \nabla\cdot u=0\ {\rm in}\
\RR^3 \right\}.$$
 The space $H$ is a
separable Hilbert space with the inner product inherited from
$L^{2}$, denoted in the sequel by $(.,.)$ with corresponding norm 
$|\, .\,|_{L^2}$. 

To ease notations, when no confusion arises let $L^p$ (resp. $L^q_v(L^p_h)$)  also denote the set of triples of functions $u=(u_1,u_2,u_3)$
such that each component $u_j$ belongs to $L^p$ (resp. to $L^q_v(L^p_h)$), $j=1,2,3$,
 that is $u\in L^p(\RR^3;\RR^3)$ (resp. $u\in L^q_v(L^p_h)(\RR^3;\RR^3)$).
For non negative indices $s,s'$ we  set 
\[ \tilde{H}^{s,s'}:= \big( H^{s,s'}\big)^3 \cap H\quad \mbox{\rm and again}\quad \|\,\cdot\,\|_{s,s'} \quad\mbox{\rm for the corresponding norm}.\] 
We denote by $(\cdot,\cdot)_{0,1}$ the scalar product in the Hilbert space $\tilde{H}^{0,1}$, that is for $u,v\in \tilde{H}^{0,1}$:
\[ \big( u\, ,\, v\big)_{0,1}=\sum_{j=1}^3 \int_{\RR^3} u_j(x) \, v_j(x)\, dx + \sum_{j=1}^3 \int_{\RR^3} \partial_3 u_j(x) \, \partial_3 v_j(x)\,  dx.\]

As defined previously, we set $\Delta_h:=\partial_1^2 + \partial_2^2$; 
integration by parts implies that given $u\in (H^{2,0})^3$ we have
\[ \big( \Delta_h u \, , u\big) = - \sum_{j=1}^3 \int_{\RR^3}
| \nabla_h u_j |_{L^2}^2 \, dx, \quad \mbox{\rm where} \quad \nabla_h u_j = (\partial_1 u_j, \, \partial_2 u_j, \,  u_j).\]
To ease notation, we write $\nabla_h u$ to denote the triple of functions $(\nabla_h u_j, j=1,2,3)$ so that 
$\big\langle \Delta_h u \, , u\big\rangle =- |\nabla_h u|_{L^2}^2$ for $u\in \tilde{H}^{1,0}$.
%and let $V_h:=H\cap [(-\Delta_h)^{\frac{1}{2}}(L^2)] = $ be endowed
%with the norm 
%\[ \|u\|_{V_h} := | (-\Delta_h )^{\frac{1}{2}} u|_{L^2(\RR^3;\RR^3)} = | \nabla_h u|_{L^2(\RR^3;\RR^3)} .\] 

\par Note that as usual, starting with an initial condition
$u_0\in \tilde{H}^{0,1}$ and projecting equation \eqref{saNS_p}
on the space of divergence-free fields, we get rid of the pressure and rewrite the evolution equation as follows: 
\begin{equation} %\label{saNS}
 \partial_t u  - \nu \, A_h  u + 
 B(u,u) %\big( u\, \cdot \, \nabla\big) u 
 + a\,  \big| u\big|^{2\alpha} \, u   =  \sigma(t,u)\, \dot{W} 
 \quad \mbox{\rm for } \; (t,x)\in [0,T]\times \RR^3,  \label{saNS}
 \end{equation}
 where 
 \begin{equation} \label{defB}
  A_h u =  P_{\rm div} \Delta_h u, \quad B(u,v)= P_{\rm div} \big( u\, \cdot \, \nabla  v\big), \quad  |u|^\alpha u =  P_{\rm div}(|u|^\alpha u ), 
   \end{equation} 
  and $P_{\rm div}$
 denotes the projection on divergence free functions.  For $u\in H^1$ such that $\nabla\cdot u =0$, set 
 \[ B(u):=B(u,u).\]

%&\nabla\, \cdot \, u =0,  \quad \mbox{\rm for } \; (t,x)\in [0,T]\times \RR^3, \nonumber 
%\end{align}
\subsection{Some properties of the non linear terms} 
In this section, we describe some properties of the non linear terms $B(u)=u\cdot \nabla u$ and $|u|^{2\alpha} u$ in equation \eqref{saNS}. 
They will be crucial to obtain apriori estimates and prove  global well posedness. 

First, for $u,v,w$ in the classical (non isotropic) Sobolev space $H^1$ such that $\nabla\cdot u = \nabla\cdot v = \nabla\cdot w=0$,
% set
% \[begin{equation} \label{defB}
%B(u,v):= \big(  u\, \cdot  \nabla \big)  v,\quad \mbox{\rm and} \quad 
%B(u):=B(u, u); 
%\end{equation} 
%then 
the classical antisymmetry property is satisfied:
\begin{equation}  \label{antisymetry}
\big\langle  \big( B(u,v) \, ,  w \big\rangle  = - \big\langle  \big( B(u,w) \, ,  v \big\rangle, \quad \mbox{\rm and} \quad 
\big\langle  \big( B(u,v) \, ,  v \big\rangle=0.
\end{equation}

We will prove that under proper assumptions on the initial condition $u_0$ and on the stochastic forcing term, the solution $u$
to the SPDE  \eqref{saNS}  belongs a.s. to the set $X$ defined by 
\begin{equation} \label{defX}
X:= L^\infty(0,T;\tilde{H}^{0,1})\cap L^2(0,T;\tilde{H}^{1,1})\cap  L^{2(\alpha +1)} ((0,T)\times \RR^3 ;\RR^3) %\cap Y
\end{equation}
and endowed with the norm 
\[ \|u\|_X:=   \sum_{j=1}^3\Big[ \esssup_{t\in [0,T]} \|u_j(t)\|_{0,1} + \Big( \int_0^T  \! \|u_j(t,.)\|_{1,1}^2 dt \Big)^{\frac{1}{2} }
+ \|u_j\|_{L^{2(\alpha +1)}([0,T]\times \RR^3)} \Big]. \]
For random processes, we set  $\Omega_T:= \Omega \times (0,T)$   endowed with the product measure
$ d\PX \otimes ds $ on $  \mathcal{F} \otimes  \mathcal{B} (0, T)$,  and 
\begin{equation} \label{calX}
{\mathcal X}:=L^4\big(\Omega;L^\infty(0,T;\tilde{H}^{0,1})\big)\cap L^4\big(\Omega;(L^2(0,T;\tilde{H}^{1,1})\big) \cap
L^{2(\alpha+1)}(%\Omega\times [0,T]
\Omega_T \times \RR^3;\RR^3).
\end{equation}

First, let us prove some integral upper estimates of the bilinear term. %$\langle B(u(t)), v(t)\rangle$ for $u,v$ that belong to $X$ and ${\mathcal X}$. 
\begin{lemma}  \label{lemma_B_1}
Let $u\in L^\infty(0,T;H) \cap L^2(0,T;\tilde{H}^{1,0})$ and   $v\in L^\infty(0,T;H) \cap L^2(0,T;\tilde{H}^{1,1})$. Then
\begin{align}  \label{majB_1}
  \int_0^T \!\! \big| \big\langle B(&u(t)),v(t) \big\rangle \big| dt  \leq C \Big( \int_0^T \|v(t)\|_{1,1}^2 dt\Big)^{\frac{1}{2}} \, \esssup_{t\in [0,T]}|u(t)|_{L^2} \,
\Big( \int_0^T |\nabla_h u(t)|_{L^2}^2 dt\Big)^{\frac{1}{2}},\\
\big| \big\langle B(u&(t)) - B(v(t))  , (u-v)(t)\big\rangle \big|  \leq C \, \|v\|_{1,1}\, \big| \nabla_h (u(t)-v(t))\big|_{L^2}\, \big| (u-v)(t)\big|_{L^2}, \label{maj_B-B}\\ 
 \int_0^T \big| \big\langle &B(u(t)) - B(v(t)) , (u-v)(t)\big\rangle \big| dt  \leq  C \, \Big( \int_0^T \|v(t)\|_{1,1}^2 dt\Big)^{\frac{1}{2}} \nonumber \\
&\qquad\qquad  \qquad\qquad \times  \esssup_{t\in [0,T]}|(u-v)(t)|_{L^2} \, \Big( \int_0^T |\nabla_h \big( (u-v)(t)\big) |_{L^2}^2 dt\Big)^{\frac{1}{2}}.
 \label{maj_B-B_1} 
\end{align} 
\end{lemma}
\begin{proof}  
Let us prove some upper estimates of $\big\langle B(\varphi, \psi) , v\big\rangle$ for $\varphi,\psi \in \tilde{H}^{1,0}$ and $v\in \tilde{H}^{1,1}$.
Since  $\nabla \, \cdot \, \varphi = \nabla \, \cdot\, \psi = \nabla\, \cdot \, v=0$, using notations similar to that in \cite{BTZ} and part of the
arguments in this reference used to prove the uniqueness of the solution, the antisymmetry \eqref{antisymetry} of $B$ yields
\begin{equation}   \label{J1-J2}%begin{align*}
-\big\langle B(\varphi, \psi)\, ,\, v \big\rangle = \big\langle \big( B(\varphi, v) \, ,\, \psi \big\rangle= J_1+J_2,
\end{equation}  %end{align*}
where
\begin{align*}
J_1&:= \sum_{k=1}^2 \sum_{l=1}^3 \int_{\RR^3} \varphi_k(x) \, \partial_k v_l(x)\, \psi_l(x)\, dx, \quad 
J_2:= \sum_{l=1}^3 \int_{\RR^3} \varphi_3(x) \, \partial_3 v_l(x)\, \psi_l(x)\, dx. 
\end{align*}
The Fubini theorem and H\"older's inequality applied to the Lebesgue integral with respect to $dx_h$ imply that for almost every $t\in [0,T]$:
\begin{align*}   % \label{maj_J1_1}
|J_1|&\leq \sum_{k=1}^2 \sum_{l=1}^3 \int_{\RR} |\partial_k v_l(.,x_3)|_{L^2_h} \, \|\varphi_k(.,x_3)\|_{L^4_h} \,  \|\psi_l(.,x_3)\|_{L^4_h} \,dx_3
\\ %\nonumber \\
&\leq \sum_{k=1}^2 \sum_{l=1}^3 \Big( \sup_{x_3} |\partial_k v_l(.,x_3)|_{L^2_h}\Big) \int_{\RR} \|\varphi_k(.,x_3)\|_{L^4_h} \, \|\psi_l(.,x_3)\|_{L^4_h}\, dx_3. 
\end{align*}
The Gagliardo-Nirenberg inequality implies that for almost every $x_3\in \RR$ we have for $\phi=\varphi_k(.,x_3)$ and $\phi=\psi_l(.,x_3)$: 
\begin{equation} \label{G-N_2D}
 \| \phi\|_{L^4_h} \leq C\, |\nabla_h \phi|_{L^2_h}^{\frac{1}{2}}\,  |\phi|_{L^2_h}^\frac{1}{2}.  
 \end{equation} 
%\[ \|u(t,.,x_3)\|_{L^4_h}^2 \leq C\, |\nabla_h u(t,.,x_3)|_{L^2_h} \, | u(t,.,x_3)|_{L^2_h}.\]
On the other hand, for almost every  $x_3\in \RR$ the Cauchy-Schwarz inequality for the Lebesgue measure on $\RR^3$
implies for $k=1,2$ and $l=1,2,3$:
\begin{align*}  %\label{sup_x3_partial_v} 
|\partial_k v_l(.,x_3)|_{L^2_h}^2 &= \int_{-\infty}^{x_3} \frac{d}{dz} |\partial_k v_l(.,z)|_{L^2_h}^2\, dz 
=2 \int_{-\infty}^{x_3} \int_{\RR^2}\!  \partial_k v_l(x_h,z)\, \partial_z \partial_k v_l(x_h,z) dx_h dz \\ %  \nonumber \\
&\leq C |\nabla_h v|_{L^2} \, 
|\partial_3 \nabla_h v|_{L^2} \leq C \|v\|_{1,1}^2.
\end{align*}
Therefore, the H\"older inequality with respect to the Lebesgue measure $dx_3$ implies that %for $t\in [0,T]$,
\begin{align} \label{maj_J1}
 |J_1| \leq & C \; \|v\|_{1,1} \Big(\int_\RR |\nabla_h \varphi(.,x_3)|_{L^2_h}^2dx_3\Big)^{\frac{1}{4}}\, 
 \Big(\int_\RR |\nabla_h \psi(.,x_3)|_{L^2_h}^2dx_3\Big)^{\frac{1}{4}}\nonumber \\
 & \qquad \times 
\Big(\int_\RR | \varphi(.,x_3)|_{L^2_h}^2 dx_3\Big)^{\frac{1}{4}} \, \Big(\int_\RR | \psi(.,x_3)|_{L^2_h}^2 dx_3\Big)^{\frac{1}{4}}  \nonumber \\
\leq & C \; \|v\|_{1,1}  \,| \nabla_h \varphi |_{L^2}^{\frac{1}{2}} \,  | \nabla_h \psi |_{L^2}^{\frac{1}{2}} \,  |\varphi |_{L^2}^{\frac{1}{2}} 
\,  |\psi |_{L^2}^{\frac{1}{2}}  .
\end{align}
Using once more the Fubini theorem and H\"older's inequality with respect to $dx_h$ we deduce that% for almost every  $t\in [0,T]$
\begin{align*}  %  \label{maj_J2_1}
|J_2|&\leq \sum_{l=1}^3 \int_\RR  \|\partial_3 v_l(.,x_3)\|_{L^4_h}\, |\varphi_3(.,x_3)|_{L^2_h}\, \|\psi_l(.,x_3)\|_{L^4_h} \, dx_3
\\  %\nonumber \\
&\leq \sum_{l=1}^3 \Big( \sup_{x_3} |\varphi_3(.,x_3)|_{L^2_h} \Big) \int_\RR \|\partial_3 v_l(.,x_3)\|_{L^4_h}\, \|\psi_l(.,x_3)\|_{L^4_h} \, dx_3.
\end{align*}
Furthermore, since $\nabla\, \cdot\, \varphi=0$, we deduce that $\partial_3 \varphi_3(x_h,x_3) = - {\rm div }\varphi_h(x_h,x_3):= 
-\big[ \partial_1 \varphi_1(x_h,x_3)
+ \partial_2 \varphi_2(x_h,x_3)\big] $. Therefore, the Cauchy-Schwarz inequality with respect to the Lebesgue measure on $\RR^3$ yields
for almost every $t\in [0,T]$ and $x_3\in \RR$:
\begin{align*}  %\label{sup_x3_u3}
|\varphi_3(.,x_3)|_{L^2_h}^2 &= 2\int_{-\infty}^{x_3} \int_{\RR^2} \varphi_3(x_h,z)\, \partial_z \varphi_3(x_3,z)\, dx_h\, dz \nonumber \\
& =-2 \int_{-\infty}^{x_3} \int_{\RR^2} \varphi_3(x_h,z)\, {\rm div }\varphi_h(x_h,z)\, dx_h\, dz \leq 2 |\nabla_h \varphi |_{L^2} \, |\varphi|_{L^2}.
\end{align*}
Plugging the above upper estimate, using again the Gagliardo-Nirenberg inequality \eqref{G-N_2D} for
 %we deduce that for almost every $(t,x_3)\in [0,T]\times \RR$ we have for 
 $\phi = \partial_3 v_l(.,x_3)$
and $\phi=\psi_l(.,x_3)$,  using the H\"older inequality with respect to the Lebesgue measure $dx_h$ %and then with respect
%to the Lebesgue measure $dt$, 
we obtain:
\begin{align}   \label{maj_J2}  
|J_2|&\leq C \sum_{l=1}^3  |\nabla_h \varphi|_{L^2}^{\frac{1}{2}} \, | \varphi |_{L^2}^{\frac{1}{2}}  
\int_\RR |\nabla_h \partial_3 v_l(.,x_3)|_{L^2_h}^{\frac{1}{2}}\,  |\partial_3 v_l(.,x_3)|_{L^2_h}^{\frac{1}{2}} \nonumber \\
&\qquad \qquad \times 
|\nabla_h \psi_l(t,.,x_3)|_{L^2_h}^{\frac{1}{2}}\, |\psi_l(t,.,x_3)|_{L^2_h}^{\frac{1}{2}}\, dx_3 \nonumber \\
&\leq C\, |\nabla_h \varphi|_{L^2}^{\frac{1}{2}} \, | \varphi |_{L^2}^{\frac{1}{2}} \, |\nabla_h \partial_3 v|_{L^2}^{\frac{1}{2}}\, | \partial_3 v|_{L^2}^{\frac{1}{2}}\, 
 |\nabla_h \psi|_{L^2}^{\frac{1}{2}} \, | \psi |_{L^2}^{\frac{1}{2}}  \nonumber \\
&\leq C\, \|v\|_{1,1} \,|\nabla_h  \varphi |_{L^2}^{\frac{1}{2}}\, |\nabla_h  \psi|_{L^2}^{\frac{1}{2}} \, |\varphi |_{L^2}^{\frac{1}{2}}\, |\psi|_{L^2}^{\frac{1}{2}} . 
\end{align}
The upper estimates \eqref{J1-J2}, \eqref{maj_J1} and \eqref{maj_J2} imply the existence of a positive constant $C$ such that 
\begin{equation} \label{maj_B}
\big| \big\langle B(\varphi,\psi), v \big\rangle \big| \leq 
 C\, \|v\|_{1,1} \,|\nabla_h  \varphi |_{L^2}^{\frac{1}{2}}\, |\nabla_h  \psi|_{L^2}^{\frac{1}{2}} \, |\varphi |_{L^2}^{\frac{1}{2}}\, |\psi|_{L^2}^{\frac{1}{2}} . 
\end{equation}
Let $u\in L^\infty(0,T;H) \cap L^2(0,T;\tilde{H}^{1,0})$ and   $v\in L^\infty(0,T;H) \cap L^2(0,T;\tilde{H}^{1,1})$. 
Since for almost every $t\in [0,T]$ we have $u(t,.)\in \tilde{H}^{0,1}$ and $v(t,.)\in \tilde{H}^{1,1}$, using \eqref{maj_B} for
$\varphi = \psi = u(t)$ and  H\"older's inequality with respect to the Lebesgue measure on $[0,T]$, we obtain
\begin{align}  \label{maj_Bt}
\int_0^T \big| \big\langle B(u(t)) , v(t) \big\rangle \big|\, dt  
 \leq &\,  C\,  \|v\|_{L^2(0,T;\tilde{H}^{1,1})} \Big( \int_0^T |\nabla_h u(t,.)|_{L^2}^2 \, |u(t,.)|_{L^2}^2\, dt \Big)^{\frac{1}{2}}
 \nonumber \\
 \leq & \, C \, \|v\|_{L^2(0,T;\tilde{H}^{1,1})}\, \esssup_{t\in [0,T]} |u(t)|_{L^2}\, \Big(\int_0^T |\nabla_h u(t)|_{L^2}^2 dt\Big)^{\frac{1}{2}}. 
 \end{align}
This concludes the proof of \eqref{majB_1}. 

 Expanding $B(u(t))-B(v(t))$ and using the antisymmetry property \eqref{antisymetry} we deduce that
\[ \big\langle B(u(t,.)) - B(v(t,.))\, , (u-v)(t,.) \big\rangle = \big\langle B\big((u-v)(t,.), v(t,.)\big)\, , (u-v)(t,.) \big\rangle \, .
\]
Using once more the antisymmetry and the upper estimate \eqref{maj_B}  with %$u-v$ instead of $u$, 
 $\varphi = \psi = (u-v)(t)$, we conclude the proof of \eqref{maj_B-B}.
Integrating \eqref{maj_B-B} on $[0,T]$ and using the Cauchy Schwarz inequality, we deduce \eqref{maj_B-B_1}.
\end{proof}
Using H\"older's inequality with respect to the expected value in the upper estimates of Lemma \ref{lemma_B_1}, we deduce the following  
 analog for stochastic  processes. 
\begin{lemma} \label{lem_B_sto}
Let $u\in L^4\big(\Omega; L^\infty(0,T;H)\big)\cap L^4\big(\Omega;L^2(0,T;\tilde{H}^{1,0})\big) $  and
$v\in L^4\big(\Omega; L^\infty(0,T;H)\big)\cap L^4\big(\Omega;L^2(0,T;\tilde{H}^{1,1}) \big)$. Then
\begin{align}  \label{majB_2} 
\EE\int_0^T & \big| \big\langle B(u(t)),v(t) \big\rangle\big|   dt \leq  C\, \|v\|_{L^4(\Omega;L^2(0,T;\tilde{H}^{1,1})) }
\nonumber \\
&\qquad\qquad \times 
\|u\|_{L^4(\Omega;L^\infty(0,T;H))}\, \Big( \EE\Big[ \int_0^T |\nabla_h u(t)|_{L^2}^2 dt\Big]^2\Big)^{\frac{1}{4}}. \\
%\end{align}
%(ii) Let $u\in {\mathcal X}$ and $v\in L^4\big(\Omega;L^2(0,T;\tilde{\mathcal H}^{1,1}) \big)$; then
%\begin{align} 
 \EE\int_0^T & \big| \big\langle  B(u(t))-B(v(t)) , (u-v)(t) \big\rangle \big| dt  \leq  C\, \|v\|_{L^4(\Omega;L^2(0,T;\tilde{ H}^{1,1})) }
\nonumber \\
&\qquad\qquad \times 
\|u-v\|_{L^4(\Omega;L^\infty(0,T;H))}\, \Big( \EE\Big[ \int_0^T |\nabla_h (u-v)(t)|_{L^2}^2 dt\Big]^2\Big)^{\frac{1}{4}}.
 \label{maj_B-B_2}
\end{align} 
\end{lemma}  
The following lemma proves upper estimates for the third partial derivatives of the bilinear term; it is essentially contained in \cite{BTZ}.
 This results shows the crucial
role of the other non linear term $|u|^{2\alpha} u$ of \eqref{saNS} in the control of the  partial derivative $\partial_3$ of the bilinear term. 
\begin{lemma} \label{lem_delta_3_B}
There exists a  positive constant $C$  such that for any $\alpha\in (1,\infty)$ there exists $C_\alpha>0$,
 $\epsilon_0,\epsilon_1>0$,  $s\in [0,T]$ and  $u\in X$:
\begin{align}  \label{delta_3_B}
\Big|  \big\langle \partial_3 B(u(s)), \partial_3 u(s)\big\rangle \Big| \leq &\; C\, \Big[ \epsilon_0  |\nabla_h\partial_3 u(s)|_{L^2}^2\, 
+ \frac{\epsilon_1}{4\epsilon_0}  \big| |u(s)|^\alpha \, \partial_3 u(s)\big|_{L^2}^2  
\nonumber \\
&\qquad + {C_\alpha}\, {\epsilon_0}^{-1} \, \epsilon_1^{-\frac{1}{\alpha-1}}\,  |\partial_3 u(s)|_{L^2}^2  \Big].
\end{align}
\end{lemma} 
\begin{proof}
We briefly sketch the proof in order to be self contained. Since $\mbox{\rm div}_h \partial_3 u(s)= \partial_3 \mbox{\rm div}_h u(s)$,
the antisymmetry \eqref{antisymetry} yields $\big\langle B\big(u(s), \partial_3 u(s)\big)\,,\,\partial_3 u(s)\big\rangle=0$; hence for $s\in [0,T]$:
\[ %begin{align*}
\big\langle \partial_3 B(u(s)),\partial_3 u(s)\big\rangle = \sum_{k,l=1}^3 \int_{\RR^3 }\partial_3 u_k(s,x) \partial_k u_l(s,x) \partial_3 u_l(s,x) dx 
:= \bar{J}_1(s) +\bar{J}_2(s),
%+ \sum_{k,l=1}^3 \int_{\RR^3} u_k(s,x) \partial_k \partial_3 u_l(s,x) \partial_3 u_l(s,x) dx
\] %end{align*}
where integration by parts with respect to $\partial_k$, $k=1,2$ yields
\begin{align*}
\bar{J}_1(s)=&-\sum_{k=1}^2\sum_{l=1}^3 \int_{\RR^3} \partial_k \partial_3 u_k(s,x)\, u_l(s,x)\, \partial_3 u_l(s,x)\, dx \\
&- \sum_{k=1}^2\sum_{l=1}^3 \int_{\RR^3} \partial_3 u_k(s,x)\, u_l(s,x)\, \partial_k \partial_3 u_l(s,x)\, dx, \\
\bar{J}_2(s)=&\sum_{l=1}^3 \int_{\RR^3}  \partial_3 u_3(s,x)\, \big(\partial_3 u_l(s,x)\big)^2 \, dx =-\sum_{l=1}^3 \int_{\RR^3}   
\mbox{\rm div}_h u_h(s,x)  \big(\partial_3 u_l(s,x)\big)^2 \, dx;
\end{align*}
 the last identity comes from the fact that $\nabla\cdot u(s)=0$. Since $\alpha>1$, the H\"older and Young inequalities imply that for functions
 $f,g,h:\RR^3\to \RR$, $\epsilon_0>0$ and then $\epsilon_1>0$, we have for some $C_\alpha >0$:
 \begin{align} \label{fgh}
\Big| \int_{\RR^3} f(x) g(x) h(x) dx\Big| &\leq \big\| \, |f| \, |g|^{\frac{1}{\alpha}}\, \big\|_{L^{2\alpha}}\,
 \big\||g|^{1-\frac{1}{\alpha}} \big\|_{L^{\frac{2\alpha}{\alpha-1}}}\,
|h|_{L^2}  \nonumber \\
 &\leq  \epsilon_0 |h|_{L^2}^2 + 
 \frac{\epsilon_1}{4\epsilon_0}  \big| \, |f|^\alpha \, g\big|_{L^2}^2 
+ C_\alpha \epsilon_0^{-1}\, \epsilon_1^{-\frac{1}{\alpha-1} } |g|_{L^2}^2 . 
 \end{align}
 Using this inequality for $f=u_l(s)$, $g=\partial_3 u_l(s)$ and $h=\partial_k\partial_3 u_k(s)$ (resp. $g=\partial_3 u_k(s)$, $h=\partial_k \partial_3 u_l(s)$)
 we deduce the existence of $C>0$ such that for any $\alpha>1$, $\epsilon_0, \epsilon_1>0$ and some constant $C_\alpha>0$:
 \[ %begin{equation}  \label{d3_J1}
 |\bar{J}_1(s)| \leq C\Big[  \epsilon_0  \big| \, \nabla_h \partial_3 u(s)|_{L^2}^2 + 
 \frac{\epsilon_1}{4\epsilon_0}  \big| \, |u(s)|^\alpha\,  \partial_3 u(s)\big|_{L^2}^2 
+ C_\alpha \epsilon_0^{-1}\, \epsilon_1^{-\frac{1}{\alpha-1} } |\partial_3 u(s)|_{L^2}^2 \Big]\, . 
 \] %end{equation} 
 Integration by parts implies that $\bar{J}_2(s)=2\sum_{k=1}^2\sum_{l=1}^3 \int_{\RR^3} u_k(s,x) \partial_k\partial_3 u_l(s,x) \partial_3 u_l(s,x)dx$. 
 Using \eqref{fgh} with $f=u_k(s)$, $g=\partial_3 u_l(s)$ and $h=\partial_k\partial_3 u_l(s)$, we deduce the existence of $C>0$ such that
  for any $\alpha>1$, $\epsilon_0, \epsilon_1>0$ and $C_\alpha >0$:
 \[ %begin{equation}  \label{d3_J2}
 |\bar{J}_2(s)| \leq C\Big[  \epsilon_0  \big| \, \nabla_h \partial_3 u(s)|_{L^2}^2 + 
 \frac{\epsilon_1}{4\epsilon_0}  \big| \, |u(s)|^\alpha\,  \partial_3 u(s)\big|_{L^2}^2 
+ C_\alpha \epsilon_0^{-1}\, \epsilon_1^{-\frac{1}{\alpha-1} } |\partial_3 u(s)|_{L^2}^2 \Big]  \, . 
 \] %end{equation} 
%Integrating the  above inequalities on the time interval $[0,t]$ yields \eqref{delta_3_B}. % which 
The upper estimates of $\bar{J}_1(s)$ and $\bar{J}_2(s)$ conclude the proof. 
\end{proof}

For any regular enough  function $\varphi:\RR^3\to \RR^3$,  let $F(\varphi)$ be the  function defined by
\begin{equation}\label{defF}
F(\varphi) = \nu \Delta_h \varphi - B(  \varphi) - a\, |\varphi |^{2\alpha} \varphi.
\end{equation}
The following lemma proves that for $u\in X$ (resp. $u\in {\mathcal X}$), $F(u)$ belongs to the dual space of $L^2(0,T;\tilde{H}^{1,1})\cap
L^{2(\alpha +1)}((0,T)\times \RR^3)$ (resp. to the dual space of 
$L^4(\Omega; L^2(0,T;\tilde{H}^{1,1}))\cap L^{2(\alpha+1)}(\Omega_T%\times [0,T]
\times \RR^3)$). % We let $\Omega_T:=\Omega\times [0,T]$. 
\begin{lemma}  \label{lemma_F}
(i) Let $u\in X$ and $v\in L^2(0,T;\tilde{H}^{1,1})\cap L^{2(\alpha +1)}((0,T)\times \RR^3;\RR^3)$; then  
\begin{align}  \label{majF_v}
  \int_0^T  \big| \big\langle &F(u(t,.)) , v(t,.) \big\rangle \big| dt  \leq C \, \Big[ \|v\|_{L^2(0,T; \tilde{H}^{1,0})}  \| u\|_{L^2(0,T; \tilde{H}^{1,0})}
 + \|v\|_{L^{2(\alpha +1)}( [0,T] \times \RR^3 ) }
 \nonumber \\
 & \times     \|u\|_{L^{2(\alpha +1)}( (0,T)  \times \RR^3)}^{2\alpha +1}  
 + \|v\|_{L^2(0,T;\tilde{H}^{1,1})}  \sup_{t\in [0,T]} |u(t)|_{L^2} \, \Big(\int_0^T \! |\nabla_h u(t)|_{L^2}^2 dt\Big)^{\frac{1}{2}}      \Big]. 
 \end{align}
(ii) Let $u \in {\mathcal X}$ and $v\in L^4(\Omega; L^2(0,T;\tilde{H}^{1,1}))\cap L^{2(\alpha+1)}(\Omega_T \times \RR^3)$. Then
\begin{align} \label{maj_E_F}
  \EE& \int_0^T \big|  \big\langle F(u(t,.)) , v(t,.) \big\rangle \big| dt  \leq C \, \Big[ \|v\|_{L^2(\Omega_T;\tilde{H}^{1,0})}  \| u\|_{L^2(\Omega_T; \tilde{H}^{1,0})}
 + \|v\|_{L^{2(\alpha +1)}( \Omega_T \times \RR^3 ) }
 \nonumber \\
 & \times     \|u\|_{L^{2(\alpha +1)}(\Omega_T \times \RR^3)}^{2\alpha +1}  
 +  \|v\|_{L^4(\Omega; L^2(0,T;\tilde{H}^{1,1}))}  \|u\|_{L^4(\Omega;L^\infty(0,T;H))} \|u\|_{L^4(\Omega;L^2(0,T;\tilde{H}^{1,0}))} \Big].
 % E\big(\sup_{t\in [0,T]} |u(t)|^4\big)  E \big( \|u\|_{L^2([0,T;\tilde{H}^{1,0})}^4\big)^{\frac{1}{4}}  
  \end{align} 
\end{lemma}
\begin{proof}
(i) Integration by parts and the Cauchy-Schwarz inequality with respect to $dt\otimes dx$ yield
\begin{align}
\Big|  \nu \int_0^T \langle \Delta_h u(t,.), v(t,.)\rangle dt \Big| &= \Big| -  \nu \int_0^T \int_{\RR^3} \nabla_h u(t,x)\;  \nabla_h v(t,x)\, dx  dt \Big| 
\nonumber \\
&\leq \nu \; \|u\|_{L^2(0,T; \tilde{H}^{1,0})} \|v\|_{L^2(0,T; \tilde{H}^{1,0})}.  \label{majo_F_Deltah}
\end{align} 
Note that  $2\alpha +2$ and $\frac{2\alpha +2}{2\alpha +1}$ are conjugate H\"older exponents.
Since $u\in L^{2(\alpha +1)}((0,T)\times \RR^3)$, the function $|u|^{2\alpha} u$ belongs to $L^{\frac{2(\alpha+1)}{2\alpha+1}}((0,T)\times \RR^3)$
and 
\begin{align}   \label{majpoly}
\Big| \int_0^T \int_{\RR^3} |u(t,x)|^{2\alpha}\, u(t,x)\, v(t,x)\, dx\, dt &\leq 
\big\| |u|^{2\alpha} u\big\|_{L^{\frac{2(\alpha +1)}{2\alpha +1}}((0,T)\times \RR^3)}
\| v\|_{L^{2(\alpha +1)}((0,T)\times \RR^3)}  \nonumber \\
&\leq \|u\|_{L^{2(\alpha +1)}((0,T)\times \RR^3)}^{2\alpha +1} \, \| v\|_{L^{2(\alpha +1)}((0,T)\times \RR^3)}.
\end{align}
The inequalities \eqref{majo_F_Deltah}, \eqref{majB_1} and \eqref{majpoly} conclude the proof of \eqref{majF_v}.

(ii) Let $u\in {\mathcal X}$ and $v\in L^4(\Omega; L^2(0,T;\tilde{H}^{1,1}))\cap L^{2(\alpha+1)}(\Omega_T \times \RR^3)$. Then a.s. we may apply
part (i) to $u(t)(\omega)$ and $v(t)(\omega)$. The Cauchy Schwarz and H\"older inequalities with  respect to the expectation conclude the proof.
\end{proof}

To prove uniqueness of the solution, we will need the following lemma which provides an upper estimate of 
$ \big\langle  F(u(t,.))-F(v(t,.)) , u(t,.)-v(t,.) \big\rangle  $ for $u,v \in X$ and $t\in [0,T]$. 
 
\begin{lemma}  \label{F-F}
There exists a positive constant $\kappa$ depending on $\alpha$,  and for any $\eta\in (0,\nu)$ a positive constant
$C_\eta$ such that for  $u,v \in \tilde{H}^{1,1} \cap L^{2(\alpha +1)}(\RR^3)$; 
\begin{align}  \label{upper_F-F_t}
  \big\langle  &F(u)-F(v) , u-v \big\rangle  \leq  -\eta \, |\nabla_h (u-v) |_{L^2}^2 
  \nonumber \\
&+ C_\eta\,  \| v\|_{1,1}^2 |u-v|_{L^2}^2 
    - a\, \kappa \, \big| \big( |u| + |v|\big)^{ \alpha}\, (u-v) \big|_{L^2}^2.
\end{align} 
\end{lemma}
\begin{proof} 
Integration by parts implies that
\begin{equation}  \label{F-F_Deltah}
\nu \,  \big\langle \Delta_h(u-v) \, , \, u-v \big\rangle = - \nu \, |\nabla_h (u-v)|_{L^2}^2.
\end{equation}
It is well-known  (see \cite{BL}; see also  \cite{MTT} where it is used)  that there exists 
a constant $\kappa$ depending on $\alpha$ such that
\begin{align*}
 \kappa |u(x)-v(x)|^2 & \big( |u(x)| + |v(x)|\big)^{2\alpha} \\
&  \leq  \big( |u(x)|^{2\alpha} u(x) - |v(x)|^{2\alpha} v(x)\big) \, \cdot \, \big(u(x)-v(x)\big),
\end{align*}
which clearly implies: % that 
\begin{align}   \label{F-F_polynomial}
 a\int_{\RR^3} & \big( |u(x)|^{2\alpha} u(x) - |v(x)|^{2\alpha} v(x) \big) \, .\, \big( u(x) - v(x)\big) dx \nonumber \\
&\qquad \geq a\, \kappa\, 
\big|  \big( |u| + |v|\big)^{\alpha} ( u-v )  \big|_{L^2}^2.
\end{align} 
Using Young's inequality in \eqref{maj_B-B} we deduce that for any $\eta\in (0,\nu)$ there exists $C_\eta>0$ such that
\[ \big| \langle B(u)-B(v)\, , \, u-v\rangle \big|  \leq (\nu-\eta) |\nabla_h (u-v)|_{L^2}^2 + C_\eta \|v\|_{1,1}^2\, |(u-v)|_{L^2}^2.
\]
This upper estimate, \eqref{F-F_Deltah} and  \eqref{F-F_polynomial}  conclude the proof of \eqref{upper_F-F_t}. 
\end{proof}

\subsection{The stochastic perturbation}  \label{stoch_def}
 We will consider an   external random force   
in  equation  \eqref{saNS} driven by a Wiener process $W$ 
and whose intensity may depend on the solution $u$.

More precisely, let $(e_k, k\geq 1)$ be an orthonormal basis of $H$ whose elements belong to $H^2:=W^{2,2}(\RR^3; \RR^3)$ and
are orthogonal in $\tilde{H}^{0,1}$. For integers $k,l\geq 1$ with $k\neq l$, we deduce that
\[ ( \partial_3^2  e_k,  e_l) = - (\partial_3 e_k, \partial_3 e_l) = -(e_k,e_l)_{0,1} - (e_k,e_l)=0.\]
Therefore, $\partial_3^2 e_k$ is a constant multiple of $e_k$. Let ${\mathcal H}_n= \mbox{\rm span }(e_1, \cdots, e_n)$ and let
$P_n$ (resp. $\tilde{P}_n$) denote the orthogonal projection from $H$ (resp. $\tilde{H}^{0,1}$) to ${\mathcal H}_n$. 
We deduce that for $u\in \tilde{H}^{0,1}$ we have $P_n u = \tilde{P}_n u$. Indeed,  for  $v\in {\mathcal H}_n$, 
we have $\partial_3^2 v \in {\mathcal H}_n$ and for any $u\in \tilde{H}^{0,1}$: %  we have $(P_n u, v) = (u,v)$ and 
\[ (P_n u, v) = (u,v), \quad \mbox{\rm and }\;  (\partial_3 P_n u, \partial_3 v) = - (P_n u, \partial_3^2 v) =-(u, \partial_3^2 v) = (\partial_3 u, \partial_3 v).\]
Hence  given $u\in \tilde{H}^{0,1}$, we have $(P_n u, v)_{0,1}=(u,v)_{0,1}$ for any $v\in {\mathcal H}_n$; this
 proves that $P_n$ and $\tilde{P}_n$ coincide on $\tilde{H}^{0,1}$. 

Let $(W(t), t\geq 0)$ be a $\tilde{H}^{0,1}$-valued Wiener process with covariance operator $Q$ on a filtered
probability space $(\Omega, {\mathcal F}, ({\mathcal F}_t),  {\mathbb P})$;  that is $Q$  is a positive operator from $\tilde{H}^{0,1}$
to itself which is trace class, and hence compact. Let $(q_k, k\geq 1)$ be the set of eigenvalues of $Q$ with $\sum_{k\geq 1} q_k<\infty$,
and let $(\psi_k, k\geq 1)$ denote the corresponding eigenfunctions (that is $Q \psi_k = q_k \psi_k)$. 
The process $W$ is Gaussian, has independent
time increments,  and for $s,t\geq 0$, $f,g\in \tilde{H}_{0,1}$,
\[
\EE  \big[ (W(s),f)_{0,1} \big] =0\quad\mbox{and}\quad
\EE \big[  (W(s),f)_{0,1} (W(t),g)_{0,1} \big]  = \big(s\wedge t)\, (Qf,g)_{0,1}.
\]
We also have the following representation
\begin{equation}\label{W-n}
W(t)=\lim_{n\to\infty} W_n(t)\;\mbox{ in }\; L^2(\Om; \tilde{H}^{0,1})\; \mbox{ with }
W_n(t)=\sum_{k=1}^n q^{1/2}_k \beta_k(t) \psi_k,
\end{equation}
where  $\beta_k$ are  standard (scalar) mutually independent Wiener processes and $\psi_k$ are the above eigenfunctions of $Q$. 
For details concerning this Wiener process  we refer  to \cite{PZ92}.

Let $H_0 = Q^{\frac12} \tilde{H}^{0,1}$; then $H_0$ is a
Hilbert space with the scalar product
$$
(\phi, \psi)_0 = (Q^{-\frac12}\phi, Q^{-\frac12}\psi)_{0,1},\; \forall
\phi, \psi \in H_0,
$$
together with the induced norm $|\cdot|_0=\sqrt{(\cdot,
\cdot)_0}$. The embedding $i: H_0 \to  \tilde{H}^{0,1}$ is Hilbert-Schmidt and
hence compact; moreover, $i \; i^* =Q$.
 
Let ${\mathcal L} \equiv L^{(2)}(H_0, H) $  (resp.  $\widetilde{\mathcal L} \equiv L^{(2)}(H_0, \tilde{H}^{0,1}) $  )
 be the space of linear operators $S:H_0\mapsto H$ (resp. $S:H_0\mapsto \tilde{H}^{0,1}$) such that
$SQ^{\frac12}$ is a Hilbert-Schmidt operator  from $\tilde{H}^{0,1}$ to $H $ (resp. from $\tilde{H}^{0,1}$ to itself). 
Clearly, $\widetilde{\mathcal L} \subset {\mathcal L}$. 
Set
\begin{align}\label{LQ-norm}
 |S|_{\mathcal L}^2&=\mbox{\rm trace}_{H}  ([SQ^{1/2}][SQ^{1/2}]^*)=\sum_{k=1}^\infty \|SQ^{1/2}\phi_k\|_{L^2}^2,\\
 |S|_{\widetilde{\mathcal L}}^2&=\mbox{\rm trace}_{\tilde{H}^{0,1}} ([SQ^{1/2}][SQ^{1/2}]^*)
 =\sum_{k=1}^\infty |SQ^{1/2}\phi_k|_{0,1}^2  \label{tilde_LQ-norm}
% =\sum_{k=1}^\infty \|[SQ^{1/2}]^*\psi_k\|_{0,1}^2
\end{align}
for any orthonormal basis  $\{\phi_k\}$ in $\tilde{H}^{0,1}$. Let $(\cdot, \cdot)_{\mathcal L}$ and $(\cdot, \cdot)_{\tilde{\mathcal L}}$
denote the associated scalar products. 
\par

The noise intensity of the stochastic perturbation
 $\s: [0, T]\times \tilde{H}^{1,1} \to \widetilde{\mathcal L}$ which we put in \eqref{saNS} satisfies the following classical 
 growth and Lipschitz conditions (i) and (ii). Note that due to the anisotropic feature of our model, we have to impose growth conditions
 both for the $|\, \cdot \, |_{\mathcal L}$ and  $|\, \cdot \, |_{\tilde{\mathcal L}}$ norms. 
 
\par
\noindent \textbf{Condition (C):} 
The diffusion coefficient  $\s  \in C\big([0, T] \times \tilde{H}^{1,1}; \widetilde{\mathcal L}) \big)$ is a linear operator such  that: \\
\indent {\bf (i) Growth condition}
 There exist non negative  constants $K_i$ and $\tilde{K}_i$ %such that 
such that for every $t\in [0,T]$ and $u \in \tilde{H}^{1,1}$:
%\indent {\bf (a)}   
\begin{align}  |\s(t,u)|^2_{\mathcal L}  %\sum_{j\geq 1}  q_j  \sum_{i=1}^j  \big| \tilde{\sigma}_j^i(t,u)\big|^2
& \leq {K}_0 +{K}_1 |u|_{L^2}^2 + {K}_2 |\nabla_h u|_{L^2}^2 , \label{growth_tilde_LQ}   \\
 |\s(t,u)|^2_{\widetilde{\mathcal L}} &\leq   \tilde{K}_0+  \tilde{K}_1 \|u\|_{0,1}^2+  \tilde{K}_2 \big(  |\nabla_h u |_{L^2}^2 + 
 |\partial_3 \nabla_h u|_{L^2}^2 \big) \label{growth_LQ}  .
 \end{align}
 \indent {\bf (ii) Lipschitz condition}  There exists constants ${L}_1$ and ${L}_2$ such that: % for $t\in [0,T]$ and $u\in \tilde{H}^{1,1}$:
\[ 
|\s(t,u)-\s(t,v)|^2_{\mathcal L}  %  =  \sum_{j\geq 1} \sum_{i=1}^j  q_j \big| \tilde{\sigma}_j^i(t,u) - \tilde{\sigma}_j^i(t,v)\big|^2 
\leq {L}_1 |u-v|_{L^2}^2 + {L}_2 |\nabla_h (u-v)|_{L^2}^2, \quad t\in [0,T] \;  \mbox{\rm and } \;  u,v\in \tilde{H}^{1,1}.
\]

\begin{defn}\label{def-sol}
An   $(\cF_t)$-predictable stochastic process $u(t,\om)$ is called a
{\em  weak solution } in   $C([0,T];H) \cap X$  for the stochastic equation \eqref{saNS}  on $[0, T]$
 with initial condition $u_0$ if
  $u\in  C([0, T]; H) \cap X$  a.s., where $X$ is defined in \eqref{defX},  and $u$ 
satisfies
\begin{align*}
 (u(t), v)&-(u_0, v) + \int_0^t\Big [ - \nu \big\langle u(s),  \Delta_h v \big\rangle
 - \big\langle   B(u(s),v) \, ,\,  u(s)\big\rangle \Big] ds \nonumber \\
 &+a  \int_0^t \! \int_{\RR^3}  \big|u(s,x)\big|^{2\alpha} u(s,x) v(x) dx ds 
 =   \int_0^t \big(\s(s,u(s)) dW(s),v \big),\;\; {\rm a.s.},
\end{align*}
for every test function $v \in H^2(\RR^3)$ and all $t \in [0,T]$. All terms are well defined since 
 $u\in L^{2(\alpha +1)}([0,T]\times \RR^3)$ for almost every $s\in [0,T]$; this implies $|u(s)|^{2\alpha} u(s) \in 
L^{\frac{2(\alpha +1)}{2\alpha +1}}(\RR^3)$ which is the dual space of $L^{2(\alpha +1)}( \RR^3)$. 
\end{defn}

Furthermore the Gagliardo-Nirenberg inequality implies 
${\rm Dom}(-\Delta) \subset L^{p}(\RR^3)$ for any $p\in [2,\infty)$. 
Note that this solution is a strong one in the probabilistic meaning, that is the trajectories of $u$ are written
 in terms of stochastic integrals with respect to  the given Brownian motion $W$.
\smallskip

\section{Existence and uniqueness of global solutions} \label{well_posedeness}

The aim of this section is to prove that equation \eqref{saNS} has a unique solution in ${\mathcal X}$ defined in \eqref{calX}. 
We at first prove local well posedeness of a Galerkin approximation of $u$ and apriori estimates. 

\subsection{ Galerkin approximation and apriori estimates}  \label{sec_apriori}
 Let $(e_n, n\geq 1)$ be the orthonormal basis
of   $H$  defined in section \ref{stoch_def} (that is made of functions in $H^2$ which are also orthogonal in $\tilde{H}^{0,1}$).  
Recall that for every integer $n\geq 1$ we set ${\mathcal H}_n:=\mbox{\rm span} (e_1, \cdots , 
e_n)$ and that the orthogonal projection $P_n$ from $H$ to ${\mathcal H}_n$ restricted to $\tilde{H}^{0,1}$ coincides with the
orthogonal projection from $\tilde{H}^{0,1}$ to ${\mathcal H}_n$. 

Let $\Pi_n$ denote the projection in $H_0$ on $Q^{1/2}({\mathcal H}_n)$.
Let $W_n(t)=\sum_{j=1}^n \sqrt{q_j} \psi_j \beta_j(t) = \Pi_n W(t)$ be defined by \eqref{W-n}.

Fix $n\geq 1$ and consider the following stochastic ordinary differential equation
on the $n$-dimensional space ${\mathcal H}_n$  defined  by  $u_{n}(0)=P_n u_0$,  and for
$t\in [0,T]$ and  $v \in {\mathcal H}_n$:
 \begin{equation} \label{un_H}
d(u_{n}(t), v)= \big\langle F(u_{n}(t)),v\big\rangle  dt  % + \big\langle \partial_3 F(u_{n}(t)),\partial_3 v\big\rangle\big]dt 
+  (P_n\,  \s (t, u_{n}(t)) \, \Pi_n\,  dW(t), v).
\end{equation}
%\begin{equation} \label{un_01}
%d(u_{n}(t), v)_{0,1}=\big[ \big\langle F(u_{n,h}(t)),v\big\rangle + \big\langle \partial_3 F(u_{n}(t)),\partial_3 v\big\rangle\big]dt 
%+  (\s (u_{n}(t)) dW_n(t), v)_{0,1},
%\end{equation}
%where $W_n(t)$ is defined in \eqref{W-n}.
%Note that there exists a constant $C(n)$ such that $\|\varphi \|_{2,3}\leq C(n) \, |\varphi|$ for every $\varphi \in {\mathcal H}_n$.
 Then for $k=1, \, \cdots, \, n$ we have for $t\in [0,T]$:
 \[d(u_{n}(t), e_k)= \big\langle F(u_{n}(t)),e_k \big\rangle\,  dt
+ \sum_{j=1}^n q_j^{\frac{1}{2}} \big( P_n\,  \s (t,u_{n}(t))  \psi_j\, ,\, e_k \big) \,  d\beta_j(t).
\] 
% \[d(u_{n}(t), \varphi_k)_{0,1}=\big[ \big\langle F(u_{n}(t)),\varphi_k \big\rangle 
%+ \big\langle \partial_3 F(u_n(t)), \partial_3 \varphi_k\big\rangle  \big]dt
%+ \sum_{j=1}^n q_j^{\frac{1}{2}}  \big( \s (u_{n}(t))e_j\, ,\, \varphi_k \big)_{0,1}  d\beta_j(t).
%\]
Note that  for $v\in {\mathcal H}_n$   the map
$  u\in  {\mathcal H}_n  \mapsto \langle F(u) \, ,\, v\rangle  $  
%$  u\in  {\mathcal H}_n  \mapsto \langle F(u) \, ,\, v\rangle  + \langle \partial_3 F(u),\partial_3 u\rangle$is
is locally Lipschitz. Indeed, $H^2\subset L^{2\alpha +2}$ and there exists some constant $C(n)$
 such that  $\|v\|_{H^2}  \leq C(n) |v|_{L^2}$ 
 %$\|v\|_{H^2} + \|\partial_3 v \|_{H^2} \leq C(n) |v|$  
 for $v\in {\mathcal H}_n$. 
Let $\varphi, \psi, v\in {\mathcal H}_n$; 
integration by parts implies that 
 \[ %begin{align*}
|\langle \Delta_h \varphi - \Delta_h \psi ,v\rangle|  \leq 
 \|\varphi - \psi\|_{1,0}\, \|v\|_{1,0}   \\
 \leq  C(n)^2 |\varphi - \psi|_{L^2}\, |v|_{L^2}.
 \]  %end{align*}  }
% \begin{align*}
%|\langle \Delta_h \varphi - \Delta_h \psi ,v\rangle| + |\langle \partial_3 \big( \Delta_h \varphi - \Delta_h \psi \big), \partial_3 v\rangle |& \leq 
 %\|\varphi - \psi\|_{1,0}\, \|v\|_{1,0} + \|\varphi - \psi \|_{1,1}\|v\|_{1,1}   \\
 %&\leq 2 C(n)^2 |\varphi - \psi|\, |v|.
 %\end{align*}
 In the polynomial nonlinear term, the H\"older and Gagliardo-Nirenberg inequalities imply:
 \begin{align*}
  \Big| \int_{\RR^3}\!  \big(  |\varphi (x) |^{2\alpha}\varphi (x)- &|\psi (x)|^{2\alpha} \psi (x) \big) v(x) dx  \Big|   \\
 & \leq C(\alpha) 
\big( \|\varphi\|_{L^{2\alpha +2}}^{2\alpha }  + \|\psi\|_{L^{2\alpha +2}}^{2\alpha } \big)
\,  \|\varphi - \psi\|_{L^{2\alpha +2}}\,  \|v\|_{L^{2\alpha +2}} 
 \\
&\leq C(\alpha) \, C(n)^{2(\alpha +1)} \big( |\varphi |_{L^2}^{2\alpha} +|\psi |_{L^2}^{2\alpha}\big) \, |\varphi - \psi|_{L^2}\,  |v|_{L^2} .
\end{align*} 
%  \begin{align*}
% & \Big| \int_{\RR^3}\!  \big(  |\varphi (x) |^{2\alpha}\varphi (x)- |\psi (x)|^{2\alpha} \psi (x) \big) v(x) dx  \Big| + 
% \Big| \int_{\RR^3}  \partial_3 \big( |\varphi |^{2\alpha} \varphi - |\psi|^{2\alpha} \psi \big)(x), \partial_3 v(x) dx \Big|  \\
% & \leq C(\alpha) 
%\big( \|\varphi\|_{L^{2\alpha +2}}^{2\alpha }  + \|\psi\|_{L^{2\alpha +2}}^{2\alpha } \big)
%\big(  \|\varphi - \psi\|_{L^{2\alpha +2}} \|v\|_{L^{2\alpha +2}} +  \|\partial_3 \varphi - \partial_3 \psi \|_{L^{2\alpha +2}} \|\partial_3 v\|_{L^{2\alpha +2}} \big) \\
%&\qquad + C(\alpha) 
%\big( \|\varphi \|_{L^{2\alpha +2}}^{2\alpha-1} + \|\psi \|_{L^{2\alpha +2}}^{2\alpha-1} \big) \| \varphi - \psi \|_{L^{2\alpha +2}} 
%\big( \|\partial_3 \varphi \|_{L^{2\alpha +2}} + \|\partial_3 \psi \|_{L^{2\alpha +2}} \big) \|\partial_3 v\|_{L^{2\alpha +2}} \\
%&\leq C(\alpha) \, C(n)^{2(\alpha +1)} \big( |\varphi |^{2\alpha} +|\psi |^{2\alpha}\big) \big( |\varphi|+|\psi|) |v| .
%\end{align*}

Finally, using  \eqref{maj_B}  and integration by parts we deduce:
\begin{align*}
 |\langle  B(\varphi) - B(\psi), v\rangle|  
 &=\big| -\langle B(\varphi -\psi , v)\, , \, \varphi \rangle - \langle B(\psi, v)\, , \, \varphi-\psi \rangle \big| \\
 & \leq 
C\, \|\varphi - \psi\|_{1,0} \big( \|\varphi \|_{1,0}+ \| \psi\|_{1,0} \big)  \|v\|_{1,1} \\
&\leq C C(n)^3 |\varphi - \psi|_{L^2} \big( |\varphi|_{L^2}+|\psi|_{L^2}\big) |v|_{L^2}.
\end{align*}  
%\begin{align*}
% |\langle  B(&\varphi) - B(\psi), v\rangle| + |\langle \partial_3 \big(  B(\varphi) - B(\psi) \big) , \partial_3 v \rangle| \\
% & \leq 
%C\, \|\varphi - \psi\|_{1,0} \big( \|\varphi \|_{1,0}+ \| \psi\|_{1,0} \big) \big( \|v\|_{1,1} + \|\partial_3 v\|_{1,1}\big) 
%\leq C C(n)^3 |\varphi - \psi| \big( |\varphi|+|\psi|\big) |v|.
%\end{align*}

 Condition ({\bf C})  implies that the
map    $u\in {\mathcal H}_n  \mapsto  \big( \sqrt{q_j}\, \big(\s(t,u) \psi_j\, ,\, e_k\big) : 1\leq
j,k\leq n \big)$ 
%$u\in {\mathcal H}_n  \mapsto \big( (\s_n(u) e_j\, ,\, \varphi_k)_{0,1} : 1\leq
%j,k\leq n \big)$ 
satisfies the classical global linear growth and Lipschitz conditions from ${\mathcal H}_n$ to $n\times n$ matrices uniformly in $t\in [0,T]$; indeed, 
the growth and Lipschitz conditions \eqref{growth_tilde_LQ} and {\bf (C)(ii)} imply:
\begin{align*}
 \big| \big( &\s(t,u) \sqrt{q_j} \psi_j\, ,\, e_k\big)  \big| \leq 
%|  \tilde{\s}_j(t,u) | \sqrt{q_j}   \| e_j\|_{0,1} \, |\varphi_k|_{L^2} 
\big| \s(t,u) \sqrt{q_j} \psi_j \big|_H \, |e_k|_{L^2} 
 \leq \sqrt{K_0} +\sqrt{K_1} |u|_{L^2} + \sqrt{K_2} |\nabla_h u|_{L^2} \\
 \qquad &\leq  C(n)  \big( 1+ |u|_{L^2} \big),\\
 \big| \big( &[\s(t,u) - \s(t,v)]  \sqrt{q_j} \psi_j \, , \,  e_k \big) \big| % \, |\varphi_k|_{L^2} 
 \leq \sqrt{L_1} |u-v|_{L_2} + \sqrt{L_2} |\nabla_h (u-v)|_{L^2}^2% \\
% \qquad &
  \leq C C(n) |u-v|_{L_2}.
 \end{align*}
Hence by a well-known result about existence and
uniqueness of solutions to
stochastic  differential equations  (see e.g. \cite{Kunita}), there
exists a maximal solution  $u_{n}=\sum_{k=1}^n (u_{n}\, ,\, e_k\big)\, e_k\in {\mathcal H}_n$ to \eqref{un_H},
i.e., a stopping time
$\tau_{n}^*\leq T$ such that \eqref{un_H} holds for $t< \tau_{n}^*$ and as $t \uparrow \tau_{n}^*<T$,
$|u_{n}(t)|_{L^2} \to \infty$.
\par
The following proposition shows that  $\tau_{n}^*=T$   a.s., that is provides the (global) existence and uniqueness
of the finite dimensional approximations $u_n$.   It  also gives apriori 
estimates of $u_{n}$ which do not depend on $n$; this will be crucial to prove well posedeness of \eqref{saNS}.
%As in Lemma \ref{timeincrement} we can
%made this step under less restrictive  growth
% conditions concerning $\tilde\s$ than \eqref{tilde-s-b} in ({\bf C3}).
\begin{prop} \label{Galerkin_n}  
Let  $u_0$ be a ${\mathcal F}_0$ measurable random variable such that $\EE \|u_0\|_{0,1}^4 <\infty$, 
$T>0$ and $\sigma$ satisfy condition {\bf (C)} with $\tilde{K}_2<\frac{2\nu}{21}$. Then \eqref{un_H} has a  unique   global solution 
(i.e., \ $\tau_{n}^*=T$ a.s.) with a modification  $u_{n} \in C([0, T], {\mathcal H}_n)$. Furthermore,    there exists a
constant $C>0$ such that:  % for every integer $n\geq 1$:
\begin{align} \label{apriori}
\sup_n \EX \,\Big[ & \sup_{t\in [0, T]}\|u_{n}(t)\|_{0,1}^{4}  +
\Big( \int_0^T \!\! \|u_{n}(s)\|_{1,1}^2 \, ds \, \Big)^2 + \int_0^T \! \! \|u_n(s)\|_{L^{2(\alpha +1)}}^{2(\alpha +1)} ds \Big]
 \leq  C \big( \EX\| u_0\|_{0,1}^{4} +1\big) .
\end{align}
\end{prop}

\begin{proof}
Let $u_{n}(t)$ be the % approximate 
maximal solution to \eqref{un_H} described above.
 For every $N>0$, set
\begin{equation*}% \label{TauN}
\tau_N = \inf\{t:\; \|u_{n}(t)\|_{0,1} \geq N \}\wedge \tau_n^*\, .
\end{equation*}
%Recall that  $\Pi_n : H_0\to H_0$ denote the projection
%on the span of 
%$\{ \sqrt{q_k} \, e_k, k=1, \cdots, n\}$, where
% $\{\sqrt{q_k} e_k, k\geq 1\}$ is the orthonormal basis of $H_0$  made by
% eigen-elements of the covariance operator
 %$Q$ and used to define $W_n(t)$ in  \eqref{W-n}.% Note that $\s_n(t,u)=\s(t,u)\circ \Pi_n$. 
\par
 It\^o's
 formula applied to $\| \, .\|_{0,1}$  and the antisymmetry relation  \eqref{antisymetry} of the bilinear term 
yield that for $t \in [0, T]$:
\begin{align} \label{Ito_2}
& \| u_{n}(t\wedge \tau_N)\|_{0,1}^2 = \|P_n u_0\|_{0,1}^2
- 2\nu \int_0^{t\wedge \tau_N}\!\! | \nabla_h u_{n}(s)|_{L^2}^2 ds
- 2\nu \int_0^{t\wedge \tau_N}\!\! | \nabla_h \partial_3 u_{n}(s)|_{L^2}^2 ds   \\
& - 2 a  \int_0^{t\wedge \tau_N}\!\!  \| u_n(s)\|_{L^{2\alpha +2}}^{2\alpha +2}  ds 
-2a(2\alpha +1)  \int_0^{t\wedge \tau_N} \! \int_{\RR^3} \!\!  | u_n(s,x)|^{2\alpha } |\partial_3 u_n(s,x)|^2  ds +\sum_{j=1}^3 T_j(t),
 \nonumber
\end{align}
where
\begin{align*}
T_1(t)&= -2\int_0^{t\wedge \tau_N} \langle \partial_3 B(u_n(s)),\partial_3 u_n(s) \rangle\,  ds, \\
T_2(t)&=2\int_0^{t\wedge \tau_N} \big( \sigma (s,u_n(s)) dW_n(s) , u_n(s) \big)_{0,1} , \\
T_3(t)&= \int_0^{t\wedge \tau_N}  \big| P_n\, \sigma(s,u_n(s)) \Pi_n \big|_{\widetilde{\mathcal L}}^2 \;  ds .
\end{align*}
The growth condition \eqref{growth_LQ} implies that
\[ T_3(t)\leq \int_0^{t\wedge \tau_N} \big[ \tilde{K}_0  + \tilde{ K}_1 \| u_n(s)\|_{0,1}^2 
+ \tilde{K}_2 \big( |\nabla_h u_n(s)|_{L^2}^2 + |\partial_3 \nabla_h u_n(s)|_{L^2}^2  \big) \big] ds,\]
while  \eqref{delta_3_B}  in Lemma \ref{lem_delta_3_B} yields the existence of positive constants $C, C_\alpha, \epsilon_0$ and $\epsilon_1$
such that
\begin{align*}
|T_1(t)| \leq & \, 2C
\Big[ \epsilon_0 \int_0^{t\wedge \tau_N} |\nabla_h\partial_3 u_n(s)|_{L^2}^2\, ds
+ \frac{\epsilon_1}{4\epsilon_0} \int_0^{t\wedge \tau_N}  \big| |u_n(s)|^\alpha \, \partial_3 u_n(s)\big|_{L^2}^2 ds 
\nonumber \\
&\qquad + C_\alpha \epsilon_0^{-1} \epsilon_1^{-\frac{1}{\alpha-1} }\int_0^{t\wedge \tau_N}  |\partial_3 u_n(s)|_{L^2}^2 ds \Big].
\end{align*}
Finally, the Burkholder-Davies-Gundy and Young inequalities as well as  \eqref{growth_LQ}
%condition {\bf (C)(ii-a)} 
 imply that for $\beta \in (0,1)$:
\begin{align*}
\EE\Big( \!\sup_{s\leq t} & \Big|  2\!\! \int_0^{s\wedge \tau_N} \! \!
\big( \sigma (r, u_n(r)) dW_n(r) , u_n(r) \big)_{0,1} \Big| \Big)    \leq 
6   \EE\Big\{ \! \int_0^{t\wedge \tau_N} \! \! \big|P_n \sigma (r,u_n(r)) \Pi_n \big|_{\widetilde{\mathcal L}}^2  \|u_n(r)\|_{0,1}^2 dr \Big\}^{\frac{1}{2}} \\
 \leq &\beta \; \EE\Big( \sup_{s\leq \inf t\wedge \tau_N} \| u_n(s)\|_{0,1}^2 \Big) \\
&+ \frac{9}{\beta} \; \EE \int_0^{t\wedge \tau_N}\!\! 
 \big[ \tilde{K}_0  + \tilde{K}_1 \| u_n(s)\|_{0,1}^2 
+ \tilde{K}_2 \big( |\nabla_h u_n(s)|_{L^2}^2 + |\partial_3 \nabla_h u_n(s)|_{L^2}^2  \big) \big] ds.
%big[ K_2 |\nabla_h u_n(s)|_{L^2}^2 + K_1 | u_n(s)|_{L^2}^2 + K_0 \big] ds.
\end{align*}
If $\tilde{K}_2<\frac{\nu}{5}$ and $\epsilon \in (0,2\nu -10 \tilde{K}_2)$, %\wedge \big( 2a(2\alpha +1)\big)$,
 we may choose $\beta \in (0,1)$  such that $2\nu - \big( \frac{9}{\beta} +1\big) \tilde{K}_2 > \epsilon$, then $\epsilon_0>0$ such that
 $2C\epsilon_0 < \frac{\epsilon}{2}$, and finally $\epsilon_1>0$ such  that $2a(2\alpha +1) - \frac{\epsilon_1 C}{2\epsilon_0}>\epsilon$.
 For  this choice of constants, the inequality $\| P_n u_0\|_{0,1}\leq \|u_0\|_{0,1}$ and the above upper estimates yield (neglecting some non negative
 terms in the left hand side of \eqref{Ito_2}):
 \begin{align}  \label{maj_Eun_1}
 (1-\beta) \EE\Big( \sup_{s\in [0,t]} & \|u_n(s\wedge \tau_N)\|_{0,1}^2 \Big) 
 \leq  \EE\|u_0\|_{0,1}^2 + T \tilde{K}_0 \big(\frac{9}{\beta} +1\big)  %+ \epsilon E\int_0^{t\wedge \tau_N} \big( \|u_n(s)\|_{0,1}^2 +
 % |\partial_3 \nabla_h u_n(s)|_{L^2}^2 \big) ds 
  \nonumber \\
 &\quad 
 + \Big[ \tilde{K}_1\big( \frac{9}{\beta}+1\big) + \frac{2 C C_\alpha}{\epsilon_0 \epsilon_1^{1/(\alpha -1)} }\Big] 
 \EE\int_0^{t} \! ||u_n(s\wedge \tau_N)\|_{0,1}^2 ds. 
 \end{align}
Gronwall's lemma implies that $\EE\big( \sup_{s\in [0,T]} \|u_n(s\wedge \tau_N)\|_{0,1}^2\big) \leq  C$ for some constant $C$ which does not depend on
$n$ and $N$. 
Note that $\| \phi\|_{1,1}^2 = \|\phi\|_{0,1}^2 + |\nabla_h \phi|_{L^2}^2 + |\partial_3 \nabla_h \phi|_{L^2}^2$. 
We use \eqref{maj_Eun_1} and the upper estimates   of $T_i(t)$ for $i=1,2,3$ for the same choice of constants $\beta, \epsilon_0$ and $\epsilon_1$;
this yields
\begin{equation}  \label{apriori_2}
  \EE\Big( \sup_{s\in [0,T]} \|u_n(s\wedge \tau_N)\|_{0,1}^2 \Big) 
 +  \EE\int_0^{\tau_N} \!\! \big( \|u_n(s)\|_{1,1}^2 + \|u_n(s)\|_{L^{2\alpha +2}}^{2\alpha +2} \big) ds \leq C\big( 1+\EE\|u_0\|_{0,1}^2\big)
\end{equation}
for some positive  constant $C$ depending on $\tilde{K}_i$, $i=0,1,2$,  $ \beta$, $\epsilon_0$ and $\epsilon_1$ but independent of $n$ and $N$. 

Apply once more the It\^o formula to the square of $\|\, .\, \|_{0,1}^2$.  This yields
\begin{align}  \label{Ito_4}
\|u_n&(t\wedge \tau_N)\|_{0,1}^4 =  \|P_n u_0\|_{0,1}^4 -4\nu \int_0^{t\wedge \tau_N} \|u_n(s)\|_{0,1}^2 \big[ |\nabla_h u_n(s)|^2_{L^2} + 
|\partial_3 \nabla_h u_n(s)|^2_{L^2}\big] ds  \nonumber \\
&-4a \int_0^{t\wedge \tau_N} \|u_n(s)\|_{0,1}^2  \|u_n(s)\|_{L^{2\alpha +2}}^{2\alpha +2} ds  \nonumber \\
&-4a(2\alpha +1) \int_0^{t\wedge \tau_N} \|u_n(s)\|_{0,1}^2 
\big| \, |u_n(s)|^\alpha \partial_3 u_n(s)\, \big|_{L^2}^2 ds + \sum_{j=1}^4 \tilde{T}_j(t),
\end{align}
where we let
\begin{align*}
\tilde{T}_1(t)=&\, -4 \int_0^{t\wedge \tau_N} \langle \partial_3 B(u_n(s)), \partial_3 u_n(s)\rangle \, \|u_n(s)\|_{0,1}^2 ds, \\
\tilde{T}_2(t) =&\;  4 \int_0^{t\wedge \tau_N} \big( P_n \sigma(s,u_n(s)) dW_n(s), u_n(s) \big)_{0,1} \, \|u_n(s)\|_{0,1}^2 , \\
\tilde{T}_3(t)=& \; 2 \int_0^{t\wedge \tau_N} \big|P_n \sigma(s, u_n(s)) \Pi_n\big|_{\widetilde{\mathcal L}}^2 \, \|u_n(s)\|_{0,1}^2 ds,\\
\tilde{T}_4(t)=&\;  4 \int_0^{t\wedge \tau_N} \big|\big(  \Pi_n \sigma(s,u_n(s))  P_n \big)^* u_n(s)\big|_0^2 \,  \|u_n(s)\|_{0,1}^2 ds.
\end{align*}
The growth condition \eqref{growth_LQ}  implies that 
\[ \tilde{T}_3(t) + \tilde{T}_4(t) \leq 6\int_0^{t\wedge \tau_N} \!\! \big[ \tilde{ K}_0 + \tilde{K}_1 \|u_n(s)\|_{0,1}^2 + 
\tilde{K}_2 \big( |\nabla_h u_n(s)|_{L^2}^2 + |\partial_3 \nabla_h u_n(s)|_{L^2}^2\big)  \big] \|u_n(s)\|_{0,1}^2 ds, \]
while \eqref{delta_3_B} implies  
\begin{align*} |\tilde{T}_1(t)|\leq\,  & 4C\,  \int_0^{t\wedge \tau_N} \!\!
\Big( \epsilon_0 |\nabla_h\partial_3 u_n(s)|_{L^2}^2\, 
+ \frac{\epsilon_1}{4\epsilon_0}  \big| |u_n(s)|^\alpha \, \partial_3 u_n(s)\big|_{L^2}^2  
+ C_\alpha \epsilon_0^{-1}\, \epsilon_1^{-\frac{1}{\alpha-1} }  |\partial_3 u_n(s)|_{L^2}^2  \Big)%\|u_n(s)\|_{0,1}^2\, ds
\\
& %\|u_n(s)\|_{0,1}^2\, ds 
 %\nonumber \\
 \qquad \qquad \times \|u_n(s)\|_{0,1}^2\, ds .
\end{align*}
The Burkholder-Davies-Gundy inequality, the growth condition \eqref{growth_LQ} and Young's inequality imply that for $\beta\in (0,1)$:
\begin{align*}
& \EE\Big( \sup_{s\leq t}  \tilde{T}_2(s) \Big)  \leq 
12  \EE\Big\{ \int_0^{t\wedge \tau_N} \big| \sigma (r,u_n(r)) \big|_{\widetilde{\mathcal L}}^2 \, \|u_n(r)\|_{0,1}^6 dr \Big\}^{\frac{1}{2}} \\
& \quad   \leq  \beta \EE\Big( \sup_{s\leq t\wedge \tau_N} \| u_n(s)\|_{0,1}^4 \Big) \\
&\qquad  + 
 \frac{36}{\beta}  \EE \int_0^{t\wedge \tau_N} \!\! 
  \big[  \tilde{K}_0 + \tilde{K}_1 \|u_n(s)\|_{0,1}^2 + 
\tilde{K}_2 \big( |\nabla_h u_n(s)|_{L^2}^2 + |\partial_3 \nabla_h u_n(s)|_{L^2}^2\big)  \big]
  \|u_n(s)\|_{0,1}^2 ds. 
\end{align*}
If $\tilde{K}_2<\frac{2\nu}{21}$ we may choose $\beta \in (0,1)$ and $\epsilon>0$ such that  $\epsilon < 4\nu -6 \big( 1+6/\beta) \tilde{K}_2$, 
then $\epsilon_0>0$ such that  %$4\nu - 6\big( 1+6/\beta) \tilde{K}_2 -  4C\epsilon_0>\epsilon $
$4C \epsilon_0 < \frac{\epsilon}{2}$, 
and finally $\epsilon_1>0$ such that $\frac{4C\epsilon_1}{4\epsilon_0} + \epsilon < 4a(2\alpha +1)$. For this choice of constants, neglecting
some non positive integrals in the right hand side of \eqref{Ito_4}, we deduce:
\begin{align*}
&(1-\beta) \EE\Big( \sup_{s\in [0,t]} \|u_n(s\wedge \tau_N)\|_{0,1}^4\Big) +  \frac{\epsilon}{2}\,   \EE\int_0^{t\wedge \tau_N}\! \!
 \|u_n(s)\|_{0,1}^2 \big[ |\nabla_h u_n(s)|_{L^2}^2
+ |\partial_3 \nabla_h u_n(s)|_{L^2}^2 \big] ds\\
& \leq
\EE\|u_0\|_{0,1}^4 +  \big( 6+\frac{36}{\beta} \big) \tilde{K}_1  \EE\int_0^{t} \! \!\|u_n(s\wedge \tau_N)\|_{0,1}^4 ds 
+ \Big[ 6+\frac{36}{\beta}\Big]\tilde{K}_0 \EE\int_0^{t} \! \!\|u_n(s\wedge \tau_N)\|_{0,1}^2 ds.
\end{align*}
%Since $\|\phi\|_{1,1}^2 = \|\phi\|_{0,1}^2 + |\nabla_h \phi|_{L^2}^2 + |\phi|_{L^2}^2$, 
This inequality, \eqref{apriori_2} 
and Gronwall's lemma yield $ \sup_n \EE\big( \sup_{s\in [0,T]} \|u_n(s\wedge \tau_N)\|_{0,1}^4 \big) <\infty$.
We deduce the existence of a constant $C$, which does not depend on $n$ and $N$, such that:
\begin{equation}  \label{apriori_4}
  \EE\Big( \sup_{s\in [0,T]} \|u_n(s\wedge \tau_N)\|_{0,1}^4 \Big) 
 +  \EE\int_0^{\tau_N} \!\! \ \|u_n(s)\|_{1,1}^2  \|u_n(s)\|_{0,1}^{2}  ds \leq C\big( 1+\EE\|u_0\|_{0,1}^4\big). 
\end{equation}
We now prove  that \eqref{apriori} holds. As $N\to \infty$, the sequence of stopping times $\tau_N$ increases to $\tau_n^*$,
and on the set $\{\tau_n^*<T\}$ we have $\sup_{s\in [0,\tau_N]} \|u_n(s)\|_{0,1}\to +\infty$. Hence \eqref{apriori_2} proves that $P(\tau_n^*<T)=0$ and
that for almost every $\omega$, for $N(\omega)$ large enough, $\tau_{N(\omega)}(\omega)=T$. 
The   monotone convergence theorem used  in \eqref{apriori_2} and \eqref{apriori_4},
 we deduce  the following upper estimates for some constant which does not depend on $n$: 
\begin{align}
 \EE\Big( \sup_{s\in [0,T]} \|u_n(s)\|_{0,1}^2 \Big) 
 +  \EE\int_0^T \!\! \big( \|u_n(s)\|_{1,1}^2 + \|u_n(s)\|_{L^{2\alpha +2}}^{2\alpha +2} \big) ds& \leq C\big( 1+\EE\|u_0\|_{0,1}^2\big),
 \label{apriori_2_T}\\
  \EE\Big( \sup_{s\in [0,T]} \|u_n(s)\|_{0,1}^4 \Big) 
 +  \EE\int_0^{T} \!\! \ \|u_n(s)\|_{1,1}^2  \|u_n(s)\|_{0,1}^{2}  ds &\leq C\big( 1+\EE\|u_0\|_{0,1}^4\big). \label{apriori_4_T}
\end{align}
To complete the proof and check \eqref{apriori}, we finally prove that
\begin{equation} \label{apriori_last}
\sup_n \EE\Big(\Big|  \int_0^T \|u_n(s)\|_{1,1}^2  ds \Big|^2 \Big) \leq C\, \big(1+\EE\|u_0\|_{0,1}^4\big). 
\end{equation}
The identity \eqref{Ito_2} and the upper estimates of $T_1(t)$ and $T_3(t)$ imply that for $\tilde{K}_2<2\nu$, $2C\epsilon_0 <\tilde{K}_2$
and  $\epsilon_1$ small enough we have for
every $t\in [0,T]$ and :
\begin{align} \label{Ito_2_Bis}
\|u_n(t\wedge \tau_N)\|_{0,1}^2& + (2\nu -\tilde{K}_2) \int_0^{t\wedge \tau_N} \big( |\nabla_h u_n(s)|_{L^2}^2 + |\partial_3 \nabla_h u_n(s)|_{L^2}^2\big) ds 
\nonumber \\
&\leq \|u_0\|_{0,1}^2 + \sup_{s\leq t} |T_2(s)| + J(t),
\end{align}
where for some positive constant $C$:
\[ J(t) = \int_0^{t\wedge \tau_N} \! \Big[ \tilde{K}_1 \|u_n(s)\|_{0,1}^2 + \tilde{K}_0 + \frac{2CC_\alpha}{\epsilon_0 \epsilon_1^{1/(\alpha -1)}}
 |\partial_3 u_n(s)|_{L^2}^2\Big] ds
\leq C \int_0^{t\wedge \tau_N} \!\! \| u_n(s)\|_{0,1}^2 ds.
\] 
Hence for $\tilde{K}_2<2\nu$, using the  Doob and Cauchy Schwarz  inequalities as well as \eqref{growth_LQ}, we deduce: 
\begin{align*}
\EE\Big[ \Big\{& \sup_{s\leq T} \|u_n(s\wedge \tau_N)\|_{0,1}^2 + (2\nu - \tilde{K}_2) \int_0^{\tau_N} \big( |\nabla_h u_n(s)|_{L^2}^2 +  
 |\partial_3 \nabla_h u_n(s)|_{L^2}^2 \big) ds \Big\}^2 \Big] \\
 &\leq 3 \EE (J(T)^2) + 3 \EE\Big(\sup_{s\leq T} T_2^2(s) \Big) + 3 E(\|u_0\|_{0,1}^4) \\
 &\leq 3CT \EE\int_0^{\tau_N} \|u_n(s)\|_{0,1}^4 ds 
  + 3C \EE \int_0^{\tau_N} \|u_n(s)\|_{0,1}^2 \, |\sigma(u_n(s)) \Pi_n|_{\tilde{\mathcal L}}^2 ds + 3 \EE(\|u_0\|_{0,1}^4)\\
  &\leq 3C\,  \EE\int_0^{\tau_N} \big[2 \tilde{K}_2 \|u_n(s)\|_{0,1}^2 \|u_n(s)\|_{1,1}^2 + (\tilde{K}_1+T) \|u_n(s)\|_{0,1}^4 
  + \tilde{K}_0 \|u_n(s)\|_{0,1}^2 \big] ds \\
  &\qquad 
  + 3 \EE(\|u_0\|_{0,1}^4) .
\end{align*} 
Let $N\to \infty$ in this equation.  Since $\tau_{N(\omega)}(\omega)=T$ for $N(\omega)$ large enough, the above inequality 
 where $\tau_N$ is replaced by $T$ (which is deduced by means of the monotone convergence theorem)
 coupled with \eqref{apriori_2_T} and \eqref{apriori_4_T} yield \eqref{apriori_last}. This completes the proof.% of \eqref{apriori}. 
\end{proof}

\subsection{Well posedeness of equation \eqref{saNS}}  \label{sec_WP}
The aim of this section is to prove that if the initial condition $u_0\in L^4(\Omega;\tilde{H}^{0,1})$, equation \eqref{saNS} has a unique (weak) solution
in the space ${\mathcal X}$ which belongs a.s. to $ C([0,T];H)$, where   ${\mathcal X}$  has been defined in \eqref{calX}.   
%\[ %begin{equation}  \label{def_calX}
%{\mathcal X}:= L^4\big(\Omega ; L^\infty(0,T;\tilde{H}^{0,1})\big) \cap L^4\big(\Omega; L^2(0,T; \tilde{H}^{1,1})\big)
%\cap L^{2(\alpha +1)}(\Omega\times [0,T]\times \RR^3).
% \] %end{equation}
\begin{theorem}  \label{th_wp}
Let $\sigma$ satisfy condition {\bf (C)} with $\tilde{K}_2<\frac{2\nu}{21}$  and $u_0$ be independent of $(W(t), t\geq 0)$ 
such that $E(\|u_0\|_{0,1}^4)<\infty$.
Then there exists a weak solution $u\in {\mathcal X}$ to \eqref{saNS} with initial condition $u_0$.  This solution belongs to $C([0,T]; H)$
a.s.   

Furthermore, there exists a constant $C>0$
such that this solution satisfies the following upper estimate:
\begin{equation} \label{moments_u}
\EE\Big( \sup_{0\leq t\leq T} \|u(t)\|_{0,1}^4 + \Big(\int_0^T\!\! \|u(t)\|_{1,1}^2 dt \Big)^2 + 
\int_0^T \!\int_{\RR^3}\!  |u(t,x)|^{2(\alpha +1)} dx dt\Big) 
\leq C\big( 1+\EE\|u_0\|_{0,1}^4\big).
\end{equation}
 If $L_2<2\nu $,   then \eqref{saNS} has a pathwise unique weak solution in ${\mathcal X}$ which belongs a.s. to $ C([0,T];H)$.
\end{theorem}
\begin{proof}
The proof is decomposed in several steps. 

Recall that ${\mathcal L}$ is defined by \eqref{tilde_LQ-norm}
and that $\s$ satisfies \eqref{growth_tilde_LQ}. 
\medskip

%Let 
%\[  \widetilde{L_Q}:= \{ S: H_0\to H \, :\, SQ^{\frac{1}{2}} \; \mbox{\rm is Hilbert-Schmidt from \; } \tilde{H}^{0,1}\; \; \mbox{\rm to\;} H \} .
%\]
\noindent \textbf{Step 1:  Weak convergence of the solution   }
The inequalities \eqref{apriori}
and \eqref{maj_E_F} imply the existence of a subsequence of
$(u_{n}, {n\geq 1})$ (resp. of $\big( P_n \s(.,u_n)\circ \Pi_n, n\geq 1\big)$ and of $\big( F(u_n), n\geq 1\big)$), still denoted by the same notation,
  of  processes
 $u \in {\mathcal X}$ (resp. $ \tilde{S}\in L^2(\Om_T ; {\mathcal L})$ and 
 $\tilde{F} \in \big[ L^4\big(\Omega; L^2(0,T ; \tilde{H}^{1,1})\big)\cap L^{2(\alpha +1)}(\Om_T\times \RR^3)\big]^* $), 
  and finally  of a random variable
 $\tilde{u}(T) \in L^2(\Om ;  \tilde{H}^{0,1})$,
for which the following  properties hold:
\\
(i) $u_{n}  \to u$ weakly in $L^4\big(\Omega;L^2(0,T;\tilde{H}^{1,1})\big) \cap L^{2(\alpha+1)}(\Om_T\times \RR^3)$, \\
(ii)    $u_{n} $ is weak star converging to $ u$  in $L^4\big(\Omega; L^\infty\big(0, T ; \tilde{H}^{0,1})\big)$, \\
(iii) $u_{n}(T)  \to
 \tilde{u}(T)
$ weakly in $L^2(\Om ;  \tilde{H}^{0,1})$, \\
(iv) $F(u_{n}) \to \tilde{F}$ weakly in $\big[ L^4\big(\Omega;L^2(0,T; \tilde{H}^{1,1})\big)\cap L^{2(\alpha +1)}(\Om_T\times \RR^3)\big]^* $ ,   \\
(v)   $P_n \s (.,u_{n}(.)) \Pi_n \to \tilde{S}$   weakly in $L^2(\Om_T; {\mathcal L})$.
 %in  the  $\sigma(L^1(\Om_T,H), L^\infty(\Om_T,H))$ topology
\par

Indeed, (i) and (ii) are straightforward consequences of
 Proposition \ref{Galerkin_n}, of  \eqref{apriori}, and of
 uniqueness of the limit
of $\EX \int_0^T (u_{n}(t), v(t))dt$ for appropriate $v$.
The upper estimate \eqref{maj_E_F} proves (iv). 
The definition of  $P_n$, %$\s_n$, %in \eqref{def_sigma_n}, 
$\Pi_n$,  the growth condition \eqref{growth_tilde_LQ} and \eqref{apriori} imply: % that 
\[ %begin{align*}
\sup_n  \,  \EE\int_0^T   |P_n \s(s, u_n(s)) \Pi_n|_{\mathcal L}^2\, ds 
%=  \sup_n  \EE\int_0^T \sum_{j=1}^n q_j |\sigma(s, u_n(s)) e_j|_H^2 ds \\
% \leq\sup_n \EE \int_0^T \!\! |\s(s,u_n(s)|_{\widetilde{L_Q}}^2 \\  %&\leq \sup_n
 \leq\sup_n  \EE\int_0^T \!\! \big[ {K}_0 + {K}_1 |u_n(s)|_{L^2}^2  + {K}_2 |\nabla_h u_n(s)|_{L^2}^2 \big] ds <\infty.
\] %end{align*}
This proves (v). Finally, % since $u_n\in C([0,T];\tilde{H}^{0,1})$ a.s. and the sequence $(u_n)$
% is bounded in $L^4\big(\Omega;L^\infty(0,T;\tilde{H}^{0,1})\big)$,
\eqref{apriori_4} and the equality $\tau_N=T$ a.s. imply  that $\sup_n \EE\|u_n(T)\|_{0,1}^4 <\infty$, which proves (iii). 

Furthermore, properties (i) and (ii) and \eqref{apriori} imply that
\begin{align*}
\EX\Big[ & \Big(  \int_0^T \|u(s)\|_{1,1}^2  ds \Big)^2 + \int_0^T \|u(s)\|_{L^{2\alpha+2}}^{2\alpha+2} ds \Big] 
\leq C(1+\EX \|u_0\|_{0,1}^4), \\
\EX\Big( & \sup_{s\in [0,T]} \|u(s)\|_{0,1}^4 \Big) \leq C(1+\EX \|u_0\|_{0,1}^4).
\end{align*}
\medskip

\noindent \textbf{Step 2:   An equation for the weak limits  }
The approach is that used in \cite{Pardoux}. 
We prove that $\tilde{u}(T)=u(T)$ a.s. and that for $t\in [0,T]$:
\begin{equation}   \label{eq_u_1}
u(t) = u_0 + \int_0^t \tilde{F}(s) ds + \int_0^t \tilde{S}(s) dW(s). 
\end{equation}
For $\delta >0$, let $f \in H^1(-\delta, T+\delta)$ be such that
 $\|f\|_\infty = 1$,
 $f(0)=1$ and for any integer $j \geq 1$ set
 $g_j(t)=f(t) e_j$,
 where $\{e_j \}_{j\geq 1}$
 is the previous orthonormal basis of $H$ made of elements of $H^2$  which are also orthogonal in $\tilde{H}^{0,1}$, 
 such that for every $n\geq 1$, ${\mathcal H}_n=\mbox{\rm span }\, (e_1, \cdots, e_n)$. 
 \smallskip

\par 
The It\^o formula
implies that for any $j\geq 1$, and for $0 \leq t \leq T$: 
\begin{equation} \label{equality1}
  \big( u_{n}(T)\, ,\, g_j(T)\big) = \big( u_n(0)\, ,\,  g_j(0)\big)
 +\sum_{i=1}^3  I_{n ,j}^i,
\end{equation}
where
\begin{eqnarray*}
I_{n ,j}^1  &= & \int_0^T (u_{n}(s), e_j) \,  f'(s) ds,\qquad 
I_{n ,j}^2  =  \int_0^T \langle F(u_{n}(s)), g_j(s)\rangle  ds ,\\
I_{n ,j}^3  &= &  \int_0^T \big( P_n \s(s,u_{n}(s)) \Pi_n dW(s), g_j(s)\big) . 
\end{eqnarray*}
 We study the convergence of all terms in \eqref{equality1}.  Since $f' \in L^2(0, T)$ and $\alpha >1$, 
 for every $Z \in L^{2(\alpha +1)}(\Om) \subset  L^4(\Om)$, the map 
  $(t, \om) \mapsto e_j  Z(\omega)\, f'(t)$
 belongs to $ L^4\big(\Om; L^2(0,T; \tilde{H}^{0,1})\big) % $ L^2\big(\Om; L^2(0,T; \tilde{H}^{0,1})\big)%\linebreak 
 \subset L^{\frac{4}{3}}(\Omega ; L^1(0,T;\tilde{H}^{0,1}))$. Hence,  the weak-star convergence 
 (ii) above implies that
as $n \to \infty$, $I_{n ,j}^1 \to \int_0^T \big(u(s), e_j\big)
f'(s) ds$ weakly in    $\big( L^{2(\alpha +1)}(\Om) \big)$.     %$L^2(\Om)$. 

 \noindent Furthermore, the Gagliardo-Nirenberg inequality implies that $H^1(-\delta, T+\delta) \subset L^{2(\alpha +1)}(0,T)$ and
$H^2 \subset L^{2(\alpha +1)}(\RR^3)$. Hence $(t,\omega) \mapsto g_j(t) Z(\omega) \in L^{2(\alpha +1)}(\Omega_T \times \RR^3) \cap
L^4(\Omega ; L^2(0,T ; \tilde{H}^{1,1})$; therefore,   (iv) implies that as
$n\to \infty$, $I_{n ,j}^2 \to \int_0^T \langle \tilde{F}(s), g_j(s)\rangle ds$
weakly in  $\big( L^{2(\alpha +1)}(\Om) \big)$.  
\par
To prove the convergence of $I_{n ,j}^3$, as in \cite{Sundar} (see also \cite{CM}), let
$ \mathcal{P}_T$ denote the class of predictable processes in
$L^2(\Om_T, {\mathcal L})$ with the inner product
\begin{eqnarray*}
 (G, J)_{\mathcal{P}_T} =\EX \int_0^T \big(G(s), J(s)\big)_{\mathcal L} \; ds=
 \EX \int_0^T \mbox{\rm trace}_H (G(s)QJ(s)^*) \; ds.
\end{eqnarray*}
The map  $\mathcal{T}: \mathcal{P}_T \to  L^2(\Omega)$
 defined by $ \mathcal{T} (G)(t) =
 \int_0^T \big( G(s)  dW(s) , g_j(s)\big) $
 is linear and continuous  because of the It\^o isometry.
 Furthermore, (v) shows that for every  $G \in \mathcal{P}_T$, as $n\to \infty$,
  $\big(P_n \s(.,u_{n}(.)) \Pi_n, G\big)_{\mathcal{P}_T} \to (\tilde{S}(.),  G)_{\mathcal{P}_T}$
weakly in $L^2(\Omega)$.  Hence, as $n\to \infty$,  $\int_0^T \big( P_n \sigma(s, u_n(s)) \, \Pi_n dW(s)\, , \, g_j(s) \big)$
converges to $\int_0^T \big( \tilde{S}_s \, dW(s)\, , \, g_j(s) \big)$  . 
\par
Finally, as $n\to \infty$, $P_n u_0 =u_{n}(0) \to
u_0 $  in $H$.  By (iii), $(u_{n}(T), g_j(T))$ converges to
$(\tilde{u}(T), g_j(T))$ weakly in $L^2(\Om)$. Therefore,
as $n\to \infty$,   \eqref{equality1} %\eqref{un_H} 
leads to:  

\begin{align} \label{equality2}
  (\tilde{u} (T), e_j)\, f(T)
 &= \; \big( u_0,e_j\big)
 + \int_0^T \big(u(s), e_j\big)
f'(s) ds   +   \int_0^T \langle \tilde{F}(s) , g_j(s)\rangle ds  \nonumber \\
 &\qquad  + \int_0^T \big( \tilde{S}(s)dW(s), g_j(s) \big) \; a.s.
\end{align}

 Choosing $f$ in an appropriate way, we next prove a similar identity for any fixed $t\in [0,T]$. 
For $\delta >0$, $k>\frac{1}{\delta}$, $t \in [0, T]$, let $f_k \in H^1(-\delta, T+\delta)$ be such
that
 $\|f_k\|_\infty =1$,
  $f_k=1$ on $(-\delta, t-\frac{1}{k})$ and $f_k=0$ on
$\big[t, T+\delta\big)$.
 Then $f_k \to 1_{(-\delta, t)}$ in
$L^2$, and $f'_k \to -\delta_t$ in the sense of distributions.
Hence as $k\to \infty$, \eqref{equality2} written with $f:=f_k$
yields
 \[
0= \big(u_0 - {u}(t) ,e_j\big)  +  \int_0^t  \langle \tilde{F}(s), e_j\rangle ds
  +   \int_0^t \big(\tilde{S}(s)dW(s), e_j\big)
\]
for almost all  $(\omega,t)\in \Omega_T$.  Here, the weak continuity (after some modification) of ${u}(t)$ in $H$
for almost all $\om\in\Om$ is deduced by using Lemma 1.4 in Chapter III in Temam \cite{Temam}. 
Indeed, it is  easy to see that \eqref{equality2} provides weak continuity with values in  $H^{-1}$. Using the fact that the solution
 is also a.s $L^\infty(0,T ; H)$,   Lemma 1.4 from \cite{Temam} provides that the solution is a.s. in $ C_{w}([0,T]; H) $. 
   
Note that $j$ is arbitrary and $\EX \int_0^T |\tilde{S}(s)|^2_{\mathcal L} ds < \infty$; hence   for $0\leq t \leq T$
and almost every $\omega$, we deduce \eqref{eq_u_1}.
%\begin{eqnarray} \label{equality3}
% u (t)  =   u_0 + \int_0^t \tilde{F}(s)  ds 
% +  \int_0^t  \tilde{S}(s)dW(s)  \in H.
%\end{eqnarray}
Moreover  $\int_0^t \tilde{F}(s)  ds\in H $ a.s.
 Let $f=1_{(-\delta,
T+\delta)}$; using again  \eqref{equality2}  we obtain
\begin{eqnarray*}
 \tilde{u} (T)   =   u_0 + \int_0^T \tilde{F}(s)  ds
 +  \int_0^T  \tilde{S}(s)dW(s) .
\end{eqnarray*}
This equation and \eqref{eq_u_1} yield that   $\tilde{u}(T) = u(T)$ a.s.
 %$\tilde{u}_h(T) = u_h(T)$ a.s.

\medskip
\par
\noindent \textbf{Step 3:   Identification of the limits }  
In \eqref{eq_u_1} we still have to prove that   $d\PX  \otimes  ds $ a.e. on $\Om_T$,    we have:
\[ %begin{eqnarray*}
 \tilde{S}(s)=\s(s,u(s)) \;   \mbox{\rm and }\;
  \tilde{F}(s)=F(u(s))\, . 
 \]  %end{eqnarray*}
To establish these relations we use the same idea  as in \cite{MS02} (see also \cite{Sundar}).
 More precisely, we introduce a discounting factor which enables us to cancel out terms where both elements in the scalar products
depend on $t$. 
 Let $v\in {\mathcal X}$, where ${\mathcal X}$ has been defined in \eqref{calX}.  
%\begin{eqnarray*}
%v\in   {\mathcal X}=  L^4\big(\Omega ; L^\infty(0,T;\tilde{H}^{0,1})\big) \cap L^4\big(\Omega; L^2(0,T; \tilde{ H}^{1,1})\big)
%\cap L^{2(\alpha +1)}(\Omega\times [0,T]\times \RR^3).
%\end{eqnarray*}
Since $\s$ satisfies the Lipschitz condition {\bf (C)(ii)} with a constant  ${L}_2<2\nu$, we may choose $\eta\in (0,\nu)$  such that ${L}_2<2\eta$.
For this choice of $\eta$, let $C_\eta>0$ be defined by \eqref{upper_F-F_t} and   for every $t \in [0, T]$,
set
\begin{equation} \label{def_r}
 r (t)  = \int_0^t  \Big[ C_\eta \|v(s)\|_{1,1}^2 +{L}_1 \Big]\, ds. 
\end{equation}
Then almost surely,  $0\le r(t) < \infty$ for all $t \in [0, T]$.
Moreover, we also have that
\begin{equation}\label{r-proper}
r\in L^1\big(\Omega ; L^\infty(0,T)\big), \; e^{-r}\in L^\infty(\Omega_T),\;
 r'\in L^1(\Omega_T),\;
r'e^{-r}\in L^2\big(\Om;  L^1(0,T)\big).
\end{equation}
The weak  convergence in (iii) and the property $P_n u_0\to u_0$ in $H$  imply that
\begin{equation}\label{l-lim}
\EX\big( |u(T)|_{L^2}^2\, e^{-r(T)}\big)-\EX |u_0|_{L^2}^2
\leq \liminf_n \left[\EX\big(|u_{n}(T)|_{L_2}^2\, e^{-r(T)}\big)-\EX |P_nu_0|_{L^2}^2\right].
\end{equation}
We now apply  It\^o's formula  to
$ |\phi(t)|_{L^2}^2\, e^{-r(t)}$ for $\phi=u$ and $\phi=u_{n}$.
This gives the relation
\[
\EX\big( |\phi(T)|_{L^2}^2\, e^{-r(T)}\big)-\EX |\phi(0)|_{L^2}^2=\EX\int_0^Te^{-r(s)}d\left\{|\phi(s)|_{L^2}^2\right\}
-\EX\int_0^T r'(s)e^{-r(s)}|\phi(s)|_{L^2}^2 ds,
\]
which can be justified due to \eqref{r-proper} and the property
$|\phi|^2\in L^1(\Om, L^\infty((0,T))$ for both choices of $\phi$.
Using \eqref{eq_u_1}, \eqref{un_H}   and  letting $u = v  + (u-v)$
after simplification, from \eqref{l-lim} we obtain
\begin{align} \label{r1}
  \EX & \int_0^T \!\!  e^{-r(s)}\,  \big[ -r'(s)\big\{ \big|u(s)-v(s)\big|_{L^2}^2
+   2\big( u(s)-v(s)\, ,\, v(s)\big)\}
 + 2\langle \tilde{F}(s),u(s)\rangle + |\tilde{S}(s)|^2_{\mathcal L} \; \big] ds \nonumber \\
&
\leq \liminf_n  X_n,
\end{align}
where
\begin{align*}
 X_n=\EX\int_0^T e^{-r(s)}\big[ & -r'(s) \big\{ \big|u_{n}(s)-v(s)\big|_{L^2}^2
+ 2\big( u_{n}(s)- v(s)\, ,\, v(s)\big) \big\}  \\
 & + 2\langle F(u_{n}(s)),u_{n}(s)\rangle
+ | P_n \sigma (s,u_{n}(s))\Pi_n|^2_{\mathcal L} \big]\; ds .
\end{align*}
We write  $X_n=Y_n + \sum_{i=1}^3 Z_n^i$, where $Y_n$ need not converge but is non positive, 
while the sequences  $Z_n^i$, $i=1,2,3$ converge as $n\to \infty$. 
 The upper estimate  in \eqref{upper_F-F_t} and  the Lipschitz   condition {\bf (C)(ii)} % and the definition of $\sigma_n$ 
 imply that for $s\in [0,T]$ and ${L}_2<2\eta<2\nu$:
 \begin{align*}
 &2\langle F(u_n(s)) - F(v(s))\, ,\, u_n(s)-v(s)\rangle +\big|  P_n \sigma (s,u_{n}(s))\Pi_n - P_n \sigma(s,v(s)) \Pi_n\big|^2_{\mathcal L}  \\
%  \sum_{j=1}^n q_j \big| \big[ \sigma_n(s, u_n(s)) - \sigma_n(s,v(s))\big] e_j\big|_{L^2}^2 \\
%&\quad  \leq -2\eta \big|\nabla_h \big( u_n(s)-v(s)\big)\big|_{L^2}^2 + C_\eta \|v(t)\|_{1,1}^2 \, |u_n(s)-v(s)|_{L^2}^2 \\
%&\qquad +
%\sum_{j=1}^n q_j \big| \big[ \sigma(s, u_n(s)) - \sigma(s,v(s))\big] e_j\big|_{L^2}^2 \\
 %&|\sigma(s,u_n(s)) - \sigma(s,v(s))|_{L_Q}^2 \\
 &\quad \leq -2\eta \big|\nabla_h \big( u_n(s)-v(s)\big)\big|_{L^2}^2 + C_\eta \|v(t)\|_{1,1}^2 \, |u_n(s)-v(s)|_{L^2}^2 
  + |\s(s,u_n(s))-\s(s,v(s))|_{\mathcal L}^2\\
% & \qquad + L_2|\nabla_h (u_n(s)-v(s)|_{L^2}^2 
 %+ L_1 |u_n(s)-v(s)|_{L^2}^2 \\
 &\quad \leq -(2\eta - {L}_2) |\nabla_h (u_n(s)-v(s)|_{L^2}^2 + \big( C_\eta \|v(s)\|_{1,1}^2 + {L}_1\big) |u_n(s)-v(s)|_{L^2}^2.
 \end{align*}
 Hence the definition of $r$ in \eqref{def_r} implies that 
 \begin{align} \label{Y_n}
Y_n \; : = \; & \EX \int_0^T e^{-r(s)}\big [-r'(s)
|u_{n}(s)- v(s)|_{L^2}^2   % \nonumber \\ &
 +   2 \langle F(u_{n}(s))-F(v(s)), u_{n}(s)- v(s)\rangle
  \nonumber \\
&\;  + \big| P_n\, \big[ \s(s,u_n(s)) - \s(s,v(s)) \big] \Pi_n \big|_{\mathcal L}^2 \; \Big] ds \leq 0.
%\sum_{j=1}^n q_j \big| \big( \s(s,u_n(s)) - \s (s,v(s)) \big) e_j\big|_{L^2}^2  \Big] ds  \leq 0.
\end{align}
Furthermore, $X_n=Y_n+\sum_{j=1}^3 Z_n^j$, where 
\begin{align*}
 Z_n^1 \, = \, &  \EX  \int_0^T \!\!  e^{-r(s)} \Big[ -2  r'(s) \big(u_{n}(s))- v(s), v(s) \big)
 + 2\langle F(u_{n}(s)), v(s) \rangle  + 2 \langle F(v(s)),u_{n}(s)\rangle
\\
\; & \qquad \qquad  -  2\langle F(v(s)), v(s)\rangle
+  2 \big( P_n \s(s, u_n(s)) \Pi_n\, , \, \s(s,v(s)) \big)_{\mathcal L}  \Big]\,  ds ,
 \\
 Z_n^2  \, = \, &2 \, \EX  \int_0^T \!\! e^{-r(s)}  \big( P_n \s(s, u_n(s)) \Pi_n \, , \, P_n \s(s,v(s)) \Pi_n - \s(s, v(s)) \big)_{\mathcal L}\;   ds, 
\\
Z_n^3 \, =\, & -  \,  \EX \int_0^T e^{-r(s)}  \big| P_n \s(s,v(s)) \Pi_n \big|_{\mathcal L}^2 \; ds.
\end{align*}
 We next study the convergence of $Z_n^j, j=1,2,3$,  and first prove that $\tilde{S}(s) = \sigma(s, u(s)) $  a.e. on $\Omega_T$. 
The definition of ${\mathcal X}$ and \eqref{r-proper} imply that $r' e^{-r} v \in L^2\big(\Omega; L^1(0,T; \tilde{H}^{0,1})\big)$.
Hence the weak star convergence (ii) implies that as $n\to \infty$:
\[ \EE\int_0^T e^{-r(s)} r'(s) \big( u_n(s)-v(s) \, , \, v(s) \big) ds \to \EE\int_0^T 
e^{-r(s)} r'(s) \big( u(s)-v(s) \, , \, v(s) \big) ds.\]
Since \eqref{maj_E_F} implies that $F(v)\in \big( L^4(\Omega;L^2(0,T;\tilde{H}^{1,1})) \cap L^{2(\alpha +1)}(\Omega_T\times \RR^3)\big)^*$,
the weak convergence (i) implies that $\EE\int_0^T e^{-r(s)} \langle F(v(s)), u_n(s)\rangle ds \to \EE\int_0^T e^{-r(s)} \langle F(v(s)), u(s)\rangle ds$. 
Since $v\in L^4(\Omega;L^2(0,T;\tilde{H}^{1,1})) \cap L^{2(\alpha +1)}(\Omega_T\times \RR^3)$, the weak convergence (iv) implies that
$\EE\int_0^T e^{-r(s)} \langle F(u_n(s)), v(s)\rangle ds \to \EE\int_0^T e^{-r(s)} \langle \tilde{F}(s), v(s)\rangle ds$.  Finally, the weak convergence
(v) implies that  as $n\to \infty$:
\[ %begin{align*}
 \EE\int_0^T \!\! e^{-r(s)} \big( P_n \s(s,u_n(s))\Pi_n\, , \, \s(s,v(s))\, \big)_{\mathcal L}\; ds 
%\sum_{j\geq 1} q_j \big( \s(s,u_n(s))\Pi_n e_j \,  , \, \s(s,v(s)) e_j\big) ds \\
\;  \to \; 
\EE\int_0^T\!\! e^{-r(s)}
\big( \tilde{S}(s)\, , \, \s(s,v(s)) \big)_{\mathcal L}\; ds.  %\sum_{j\geq 1} q_j \big( \tilde{S}(s) e_j \,  , \, \s(s,v(s)) e_j\big) ds.
\] % end{align*}
Hence as $n\to \infty$, 
\begin{align}   \label{lim_Z1}
Z^1_n \to  \EE\int_0^T e^{-r(s)}& \Big[ -2 r'(s) \big( u(s)-v(s), v(s)\big) + 2 \langle \tilde{F}(s),v(s)\rangle 
+ 2\langle F(v(s)), u(s)\rangle  \nonumber \\
& - 2 \langle F(v(s)),v(s)\rangle + 2
\big( \tilde{S}(s)\, , \, \s(s,v(s)) \big)_{\mathcal L}\; \Big] \, ds.
%\sum_{j\geq 1} q_j \big( \tilde{S}(s) e_j\, ,\, \s(s,v(s)) e_j \big) \Big] ds.
\end{align}
For almost every $(\omega,t)\in \Omega_T$ and any orthonormal basis $\psi_j$ of $H_0$, 
$ \sum_{j\geq n+1} q_j | {\s}(s,v(s)) \psi_j|_{L^2}^2$ converges to $0$  as $n\to \infty$. 
%\leq \sum_{j\geq n+1} q_j \|\tilde{\s}_j(s,v(s)) e_j\|_{0,1}^2  
%=  \sum_{j\geq n+1} q_j \tilde{\s}_j(s,v(s))^2\to 0,\]
This sequence is dominated by $|\s(s,v(s))|_{\mathcal L}^2$ which belongs to $L^1(P)$ by means of the growth condition \eqref{growth_tilde_LQ}
 and the definition of ${\mathcal X}$. 
Furthermore, the inequality $|P_n \s(s,u_n(s))\circ \Pi_n |_{\mathcal L}\leq |\s(s,u_n(s))|_{\mathcal L}$,  the growth condition \eqref{growth_tilde_LQ}, \eqref{apriori_4} and 
the Cauchy-Schwarz inequality yield
\[ Z_n^2 \leq \Big( \EE\int_0^T e^{-r(s)} |P_n \s(s,u_n(s))\circ \Pi_n|_{\mathcal L}^2 ds \Big)^{\frac{1}{2}} \Big( \EE\int_0^T e^{-r(s)}
\sum_{j\geq n+1} q_j |\s(s,v(s)) \psi_j|_{L^2}^2 ds \Big)^{\frac{1}{2}}.
\] 
In the above right hand side, the first factor remains bounded, while as $n\to \infty$ 
 the second one converges to 0 by  the dominated convergence theorem. This yields
\begin{equation} \label{lim_Z2}
Z_n^2 \to 0 \; \mbox{\rm as } \; n\to \infty.
\end{equation}
Finally, the definition of $P_n$,  $\Pi_n$ and the growth condition \eqref{growth_tilde_LQ} imply that for a.e. $(\omega,s)\in \Omega_T$,
  %triangular inequality yields:
\begin{align*}
 \Big| \sum_{j\geq 1} q_j |P_n  \s(s,v(s)) \Pi_n \psi_j|_{L^2}^2 -  |\s(s,v(s))|_{\mathcal L}^2 \Big| 
%\sum_{j\geq 1} q_j  | \s(s,v(s))  e_j|_{L^2}^2 \Big| 
 \leq  &\;  2 \sum_{j\geq n+1} q_j | \s(s,v(s))  \psi_j|_{L^2}^2  \\
 & \; + 2 |(P_n-{\rm Id}) \s(s,v(s))|_{\mathcal L}^2 \to 0 
 \end{align*}  %\\
%& \quad \leq \sum_{j\geq 1} q_j | \s(s,v(s)) \Pi_n e_j - \s(s,v(s)) e_j |_{L^2}^2 
%\leq   \sum_{j\geq 1} q_j \|\s(s,v(s)) \Pi_n e_j - \s(s,v(s)) e_j \|_{0,1}^2 \\
%&\quad 
%\leq 2  \sum_{j\geq n+1} q_j \| \s(s,v(s)) \Pi_n e_j\|_{0,1}^2 
%+ 2 \sum_{j\geq 1} q_j \|\big(P_n-\mbox{\rm Id}\big)  \s(s,v(s)) \Pi_n e_j\|_{0,1}^2
%\to 0
% \end{align*}
as $n\to \infty$. Furthermore,  the growth conditon \eqref{growth_tilde_LQ} implies that for every $n$:
\begin{align*}
 \Big| \sum_{j\geq 1} &q_j |P_n \s(s,v(s)) \Pi_n \psi_j|_{L^2}^2 -   | \s(s,v(s)) |_{\mathcal L}^2 \Big|
 \leq 2|\s(s,v(s))|_{\mathcal L}^2 \\ %2 \sum_{j\geq 1} q_j |\s(s,v(s))e_j|_{L^2}^2\\
&\qquad \leq    2\big[ {K}_0 + {K}_1|v(s)|_{L^2}^2  +  {K}_2 |\nabla_h v(s)|_{L^2}^2 \big] \in L^1(\Omega_T).
\end{align*}
Hence the dominated convergence theorem implies that
\begin{equation} \label{lim_Z3}
Z_n^3 \to -\EE\int_0^T e^{-r(s)}  \, |\s(s,v(s))|_{\mathcal L}^2 ds.
% \sum_{j\geq 1} q_j |\s(s,v(s))|_{L^2}^2\, ds. 
\end{equation}
Using the inequalities \eqref{r1}--\eqref{lim_Z3} we obtain: 
\begin{align}  \label{maj_FS_uv}
\EE\int_0^T e^{-r(s)} \Big[ &-r'(s) |u(s)-v(s)|_{L^2}^2 +2 \langle \tilde{F}(s) - F(v(s))\, , \, u(s)-v(s)\rangle   
\nonumber \\
& + |\tilde{S}(s) - \s(s,v(s))|_{\mathcal L}^2 \Big] ds \leq 0.
%& + \sum_{j\geq 1} q_j|\tilde{S}(s)e_j - \s(s,v(s))e_j|_{L^2}^2 \Big] ds \leq 0.
\end{align}
Let $v=u\in {\mathcal X}$; then we deduce that for almost every $(\omega,s)\in \Omega_T$ we have 
%$\sum_{j\geq 1} |\tilde{S}(s)e_j - \sigma(s,u(s))e_j|_{L^2}^2 =0$, which yields 
$\tilde{S}(s)=\s(s,u(s))$. 
\par 
 Using another choice of $v$, we next trove that $\tilde{F}(s) = F(u(s))$ a.e. on $\Omega_T$.
Let $\lambda\in  \RR$ and $\tilde{v}\in {\mathcal X}$ and set $v_\lambda = u+\lambda \tilde{v}\in {\mathcal X}$. Then if $r_\lambda$ is defined
in terms on $v_\lambda$ using \eqref{def_r}, the inequality \eqref{maj_FS_uv} yields
\begin{align} \label{maj_F_lambda}
\lambda^2 \EE\int_0^T e^{-r_\lambda(s)} r'_\lambda(s) |\tilde{v}(s)|_{L^2}^2 ds +2\lambda \EE\int_0^T e^{-r_\lambda(s)} \langle \tilde{F}(s) - F(u(s)) \, ,\, 
\tilde{v}(s)\rangle ds  \nonumber \\
+ 2\lambda \EE\int_0^T e^{-r_\lambda(s)} \langle F(u(s))- F(v_\lambda(s)) \, , \, \tilde{v}(s)\rangle ds \leq 0.
\end{align}
The upper estimate \eqref{upper_F-F_t}  and H\"older's inequality imply that for $\eta \in (0,\nu)$ and $\lambda\in (0,1]$, 
\begin{align*}
\big|& \langle F(v_\lambda(s)) - F(u(s))\, , \, \tilde{v}(s)\rangle \big| =
 \frac{1}{|\lambda|} \big| \langle F(v_\lambda(s)) - F(u(s))\, , \, v_\lambda(s) - u(s) \rangle \big| \leq |\lambda | \phi(t),
 %\\
 %&  \leq \nu |\lambda| \|\tilde{v}(s)\|_{1,1}^2 +
% C_\eta |\lambda| \|v_\lambda(s)\|_{1,1}^2 |\tilde{v}(s)|_{L^2}^2  + \frac{a\kappa}{|\lambda|} \big| \big(|u(t,.)| + |v_\lambda(t,.)|\big)^\alpha 
 % \lambda \tilde{v}(s)\big|_{L^2}^2 \leq |\lambda | \phi(t), 
\end{align*}
where by H\"older's inequality we have 
\begin{align*}
 \phi(t) = &\eta \|\tilde{v}(t)\|_{1,1}^2 + 2 C_\eta \big( \|u(s)\|_{1,1}^2 + \|\tilde{v}\|_{1,1}^2\big)   |\tilde{v}(s)|_{L^2}^2  \\
&+ C a\kappa \big( \|u(s)\|_{L^{2(\alpha +1)}}^{2\alpha} + \|\tilde{v}(s)\|_{L^{2(\alpha +1)}}^{2\alpha}\big) \|\tilde{v}(s)\|_{L^{2(\alpha +1)}}^{2}.
\end{align*}
Using once more H\"older's inequality, we deduce that 
\begin{align*}
& \EE\int_0^T \!\! \phi(t) dt \leq \;  C \EE \Big[  \int_0^T\!\! \|\tilde{v}(s)\|_{1,1}^2 ds + 
\Big[ \Big( \Big| \int_0^T\!\! \|u(s)\|_{1,1}^2 ds \Big|^2 \Big)^{\frac{1}{2}} + \Big( \Big| \int_0^T\!\! \|\tilde{v}(s)\|_{1,1}^2 ds \Big|^2 \Big)^{\frac{1}{2}} \Big]
\\
&\times \Big(  \sup_{s\in [0,T]} |\tilde{v}(s)|_{L^2}^4 \Big)^{\frac{1}{2}}  
+ C\Big[ \|u\|_{L^{2(\alpha +1)}(\Omega_T\times \RR^3)}^{2\alpha} + \|\tilde{v}(s)\|_{L^{2(\alpha +1)}(\Omega_T\times \RR^3)}^{2\alpha}\Big]
 \|\tilde{v}\|_{L^{2(\alpha +1)}(\Omega_T\times \RR^3)}^2 <\infty.
\end{align*}
Since $r_\lambda(s)\geq 0$, the dominated convergence theorem implies that 
\[ \EE\int_0^T e^{-r_\lambda(s)} \langle F(u(s))-F(v_\lambda(s)), \tilde{v}(s)\rangle ds \to 0 \; \mbox{\rm as } \; \lambda \to 0.\]
Furthermore, since $\tilde{F}(s) - F(u(s))\in \big( L^4(\Omega;L^2(0,T;\tilde{H}^{1,1})) \cap L^{2(\alpha +1)}(\Omega_T\times \RR^3)\big)^*$, using once
more the dominated convergence theorem we deduce that as $\lambda \to 0$:
\[ \EE\int_0^T e^{-r_\lambda(s)}\langle \tilde{F}(s) - F(u(s))\, ,\, \tilde{v}(s)\rangle ds \to \EE\int_0^T e^{-r_0(s)} \langle \tilde{F}(s) - F(u(s))\, , \, 
\tilde{v}(s)\rangle ds.\]
Dividing \eqref{maj_F_lambda} by $\lambda$ and letting $\lambda \to 0^+$ and $\lambda \to 0^-$, we deduce that for every $\tilde{v}\in {\mathcal X}$,
\[ \EE\int_0^T e^{-r_0(s)} \langle \tilde{F}(s) - F(u(s))\, , \, \tilde{v}(s)\rangle ds =0.\]
This implies that $\tilde{F}(s) = F(u(s))$ a.e. on $\Omega_T$. 
\par

\noindent {\bf Step 4:   Continuity of the solution}  We next prove that $u\in C([0,T];H)$ a.s.  
The proof, based on some regularization of the solution, is similar to that in \cite{CM}; 
however, the functional setting  is different which requires some changes. 
%%%%%%
Set $A= P_{\rm div} \Delta$;   then  
 $e^{-\de A}$ maps $(\tilde{H}^{1,1})^*\subset H^{-2}$ to H   for any $\de>0$. 
Furthermore, the Gagliardo Nirenberg inequality implies that $H^2(\RR^3) \subset L^{2(\alpha +1)}(\RR^3)$, so that the semi-group $e^{-\de A}$  
also maps  $L^{\frac{2 (\alpha +1)}{2\alpha +1}} (\RR^3)=( L^{2(\alpha +1)}(\RR^3))^* \subset H^{-2}$ to $H$.  
Since $|u(s)|^{2\alpha} u(s) \in L^{\frac{2 (\alpha +1)}{2\alpha +1}} (\RR^3)$ for almost every $(\omega ,t)$, 
 we deduce that  $e^{-\delta A}\int_0^. F(u(s))\, ds $ belongs to $C([0,T],H)$.
Finally, \eqref{growth_tilde_LQ} in the growth condition  {\bf C(i)}  implies
$\EE \int_0^T |e^{-\delta A} \sigma(s,u)(s)|^2_{L_Q}\, ds <+\infty$. Thus
$\int_0^. e^{-\delta A}  \s(s,u(s))\, dW(s)$ belongs to $C([0,T],H)$ a.s. (see e.g. \cite{PZ92}, Theorem 4.12).
 Therefore, it  is sufficient to prove that a.s.  $e^{-\delta A} u$ converges to $u$ uniformly on the time interval $[0,T]$, that is
\begin{equation}\label{conv-cont}
\lim_{\delta\to0} \EX\left\{ \sup_{0\leq t\leq T} |u(t)-e^{-\de A}u(t)|_{L^2}^2\right\}=0.
\end{equation}
Let $G_\de= Id-e^{-\de A}$ and apply It\^o's   formula to
$|G_\de u(t)|_{L^2}^2$. This yields
\begin{eqnarray} \label{ito-g-de}
 |G_\de u(t)|_{L^2}^2 & = &  |G_\de u_0|_{L^2}^2  -2\nu \int_0^{t}\!\! \| G_\de u(s)\|_{\tilde{H}^{1,0}}^2 ds
+ 2  I(t)   +    \int_0^{t} \!\! | G_\de \sigma(u(s))|_{\mathcal L}^2\, ds
 \nonumber \\
&& \; - 2\int_0^{t}\!\! \big\langle  B(u(s)) +2a \, |u(s)|^{2\alpha} u(s) \, , \, 
    G^2_\de u(s)\big\rangle \, ds,
\end{eqnarray}
where $I(t)=  \int_0^{t} \big( G_\de\sigma(u(s))  dW(s),  G_\de u(s)\big)$.
By the Burkholder-Davies-Gundy and Schwarz inequalities we have
\begin{eqnarray*}
\EX\sup_{0\leq t\leq T} |I(t)| &\le&
C\EX\left( \int_0^{T}\!\! |G_\de u(s)|^2
 |G_\delta \, \sigma(s,u(s))|^2_{\mathcal L} ds\right)^{1/2}  \\
&\leq  &  \frac{1}{2}\;  \EX\sup_{0\leq t\leq T}  |G_\de u(t)|_{L^2}^2 + \frac{C^2}{2} \EX
\int_0^{T} \!\! | G_\de \sigma(s,u(s))|_{\mathcal L}^2\, ds .
 \end{eqnarray*}
Hence  for some constant $C$, \eqref{ito-g-de} yields
\begin{align*}
\EX\sup_{0\le t\le T } | & G_\de u(t)|_{L^2}^2 \le  2 \, |G_\de u_0|_{L^2}^2
+ C\, \EX \int_0^{T}\!\! | G_\de \sigma(s,u(s))|_{\mathcal L}^2 ds \\
& \; +\, 4\,  \EX \int_0^{T}\!\!  \left|\big\langle  B(u(s))+  
  2a |u(s)|^{2\alpha} u(s) ,  G^2_\de u(s)\big\rangle \right|  \, ds .
 % + | G_\de \sigma(u_{h}(s))|_{L_Q}^2\right]\, ds.
\end{align*}
Since for every $u\in H$, $|G_\de u|_{L^2} \to 0$   as $\delta\to 0$
and $\sup_{\delta>0} |G_\delta|_{L(H,H)}\leq 2$,   we deduce that if $\{\varphi_k\}$
denotes an orthonormal basis in $H$, then
$ | G_\de \sigma(s,u_{h}(s))Q^{1/2} \varphi_k|_{L^2}^2 \to 0$
for every $k$  and almost every $(\omega,t)\in \Omega\times [0,T]$.
% Therefore
%\[
%|G_\delta \sigma(u_h(s))|_{L_Q}^2 \equiv
% \sum_k  | G_\de \sigma(u_{h}(s))Q^{1/2} \varphi_k|^2\to 0~~\mbox{ as }~~ \delta\to 0
%\quad
%\]
%for almost all $(\om,s)\in \Omega\times [0,T]$.
Since
\[
{\displaystyle \sup_{\delta>0 }
|G_\delta \sigma(s,u(s))|_{\mathcal L}^2 \leq
  \sum_k  \sup_{\delta >0}  | G_\de \sigma(s, u(s) )Q^{1/2} \varphi_k|_{L^2}^2
\leq C |\sigma(s,u(s))|_{\mathcal L}^2\in L^1(\Omega\times [0,T])},
\]
the Lebesgue  dominated convergence theorem
 implies  $  \EX \int_0^{T}\!\! | G_\de \sigma(s,u(s))|_{\mathcal L}^2 ds \to 0$.
Given  $u\in \tilde{H}^{1,1}\subset H^1$ we have  $\|G_\de^2 u\|_{H^1} \to 0$ as $\delta\to 0$; furthermore,
 $\sup_{\delta >0} |G_\delta|_{L(\tilde{H}^{1,1},\tilde{H}^{1,1})}\leq 2 $.  Hence
 $ \big\langle  B(u(s))\, , \,  G^2_\de u(s)\big\rangle \to 0$ for almost every
  $(\omega,s)\in \Omega_T$.
   Furthermore, $e^{-\delta A}$ is a bounded operator of $L^{2(\alpha +1)}$ (see e.g. the Appendix of \cite{PZ92}).
Hence
 $ \big\langle  |u(s)|^{2\alpha} u(s) \, , \,  G^2_\de u(s)\big\rangle \to 0$ for almost every $(\omega,s)\in \Omega_T$.
Therefore, as above, the Lebesgue
  dominated convergence theorem concludes the proof of \eqref{conv-cont}.

%%%%%%%%

\par
\medskip

\noindent {\bf Step 5: Pathwise uniqueness of the solution     }
 We finally prove that if ${L}_2$ is small enough, there exists a unique process   in 
 $ {\mathcal X}$   and a.s. in $C([0,T];H)$  which is a weak solution to \eqref{saNS}. Let $u,v\in {\mathcal X}$ be solutions
 to \eqref{saNS}  and 
belong a.s. to $C([0,T];H)$.  For every $N$ set
\[ \tau_N=\inf \{ s\geq 0 : |u(s)|_{L^2} \vee |v(s)|_{L_2} \geq N \} \wedge T  .\]
Since $|u(.)|_{L^2}$ and $|v(.)|_{L^2}$ are a.s. bounded on $[0,T]$ by the definition of ${\mathcal X}$, we deduce that 
a.s. $\tau_N\to T$ as $N\to \infty$. 
Set $U=u-v$; since $L_2<2\nu$, we may choose   $\eta \in (0,\nu)$  such that $ L_2 <2\eta  < 2\nu$. 
Let  $C_\eta$ be a constant defined in \eqref{upper_F-F_t};
  as in the argument of Step 4, despite of the lack of regularity of $u$, we may apply It\^os formula to the square of the $H$ norm
 and deduce
\[    e^{-2\, C_\eta \int_0^{t\wedge \tau_N} \|v(r)\|_{1,1}^2 dr}   \, |U(t\wedge \tau_N)|_{L^2}^2 = 2 M(t\wedge \tau_N) 
+ \int_0^{t\wedge \tau_N} \psi(s) ds,\]
where
\begin{align*}
M(\tau)&= \int_0^\tau e^{ -2C_\eta \int_0^s \|v(r)\|_{1,1}^2 dr }   \big( U(s)\, , \, \big[ \s(s,u(s)) - \s(s,v(s))\big] dW(s) \big),\\
\psi(s) &= e^{ -2C_\eta \int_0^s \|v(r)\|_{1,1}^2 dr } \big[ -2C_\eta \|v(s)\|_{1,1}^2 |U(s)|_{L_2}^2 +2 \langle F(u(s))-F(v(s))\, ,\, U(s)\rangle \\
&\quad + |\s(s,u(s))-\s(s,v(s))|_{\mathcal L}^2 \; \big] .
\end{align*}
We at first check that the process $M$ is a square integrable martingale. Indeed, the Cauchy-Schwarz and the Young inequalities, 
the Lipschitz condition {\bf (C)(ii)} and the definition of ${\mathcal X}$
imply that
\begin{align*}
 \EE \int_0^T  & e^{-4C_\eta \int_0^s  \|v(r)\|_{1,1}^2 dr }   |U(s)|_{L^2}^2 \,  \big| \s(s,u(s)) - \s(s,v(s))\big|_{\mathcal L}^2 ds \\
 &\qquad \leq \;  \EE \int_0^T  |U(s)|_{L_2}^2 \big[ {L}_1 |U(s)|_{L_2}^2 +{L}_2 |\nabla_h U(s)|_{L^2}^2 \big] ds \\
 & \qquad \leq \; C\;  \EE \Big( \sup_{t\in [0,T]} |U(s)|_{L^2}^4 \Big) + C\;  \EE \Big( \Big| \int_0^T |\nabla_h U(s)|_{L^2}^2 ds \Big|^2 \Big) <\infty.
\end{align*}
Furthermore, the upper estimate \eqref{upper_F-F_t} and the Lipschitz condition {\bf (C)(ii)} imply that for ${L}_2<2\eta<2\nu$, we have
\[ |\psi(s)|\leq  \big( {L}_2 -2\eta\big) |\nabla_h U(s)|_{L^2}^2 + {L}_1 |U(s)|_{L^2}^2 \leq {L}_1 |U(s)|_{L^2}^2.
\]
Hence taking expected values, we deduce that for any $t\in [0,T]$:
\[ \EE\Big(  e^{-2\, C_\eta \int_0^{t\wedge \tau_N} \|v(r)\|_{1,1}^2 dr}   \, |U(t\wedge \tau_N)|_{L^2}^2\Big) \leq \int_0^t 
\EE\Big(  e^{-2\, C_\eta \int_0^{s\wedge \tau_N} \|v(r)\|_{1,1}^2 dr}   \, |U(s\wedge \tau_N)|_{L^2}^2\Big)ds.
\]
The Gronwall lemma implies that for every $t\in [0,T]$, we have  $U(t\wedge \tau_N)=0$ a.s. Since
$U$ a.s. belongs to $C([0,T];H)$,  this completes the proof as $N\to \infty$.
\end{proof}

\subsection{Examples}  \label{examples}

 Here, we provide two examples of coefficients $\s$ which satisfy condition {\bf (C)} 

Let $\{ \psi_k, k\geq 1\}$ denote an orthonormal basis of $H_0=Q^{\frac{1}{2}} \tilde{H}^{0,1}$ and for $t\in [0,T]$,  $u\in \tilde{H}^{1,1}$ 
and $\psi \in H_0$; set
\[ \s(t,u) \psi (x):= \sum_{k=1}^{\infty}  \big( \psi, \psi_{k} \big)_0\,  \sigma_k(t,x,u(x),\nabla_h u(x)), \]
where $\sigma_k:[0,T]\times \RR^3 \times \RR^3 \times \RR^6 \to \RR^3$ are measurable functions with appropriate
regularity and $\nabla_h = (\partial_1 u, \partial_2 u)$. 

\smallskip
{\bf Example 1:} 
For $t\in [0,T]$, $x\in \RR^3$, $y\in \RR^3$ and $z=(\zeta,\tilde{\zeta}) $ for $\zeta, \tilde{\zeta} \in \RR^3$  set
\[ \s_k(t,x,y,z) = \sigma_{k,0}(t,x) + \s_{k,1}(t,x) y + \s_{k,2}(t,x) \zeta + \tilde{\s}_{k,2}(t,x) \tilde{\zeta}, \]
where    $\sigma_{k,0}(t,.)\in \tilde{H}^{0,1}$, $\s_{k,1}(t,.)$, $\s_{k,2}(t,.)$,  $\tilde{\s}_{k,2}(t,.)$, $\partial_3\s_{k,0}(t,.)$; 
$\partial_3 \s_{k,2}(t,.)$ and $\partial_3 \tilde{\s}_{k,2}(t,.)$
belong to $ L^\infty(\RR^3) $. Suppose furthermore that:
\begin{align*}
\sup_{t\in [0,T]} &\sum_{k\geq 1} \Big[  \|\s_{k,0}(t,.)\|_{0,1}^2   +   \|\s_{k,1}(t,.)\|_{L^\infty}^2 + \|\s_{k,2}(t,.)\|_{L^\infty}^2 
+  \|  \tilde{\s}_{k,2}(t,.)\|_{L^\infty}^2  \Big] <\infty,      \\
\sup_{t\in [0,T]} & \sum_{k\geq 1} \Big[ % |\partial_3 \s_{k,0}(t,.)|_{L^2}^2 +  
\|\partial_3 s_{k,1}(t,x)\|_{L^\infty}^2 + \|\partial_3 \s_{k,2}(t,x)\|_{L^\infty}^2
+  \|\partial_3 \tilde{\s}_{k,2}(t,x)\|_{L^\infty}^2 \Big] <\infty. 
\end{align*}
Then condition \eqref{growth_tilde_LQ}  holds with $K_0=3 \sup_t \sum_k  |\s_{k,0}(t,.)|_{L^2}^2$, 
$K_1=3\sup_t \sum_k \|\s_{k,1}(t,.)\|_{L^\infty}^2 $
and $K_2 = 3\sup_t \sum_k \big( \|\s_{k,2}(t,.)\|_{L^\infty}^2 + \| \tilde{\s}_{k,2}(t,.)\|_{L^\infty}^2 \big)$.  
The Lipschitz condition {\bf (C)(ii)} holds with $L_1=\frac{2}{3}K_1$ 
and $L_2=\frac{2}{3}K_2$. 

Taking the partial derivative with respect to $x_3$, we deduce that \eqref{growth_LQ} holds with 
$$\tilde{K}_0=5\sup_t \sum_k \|\s(t,.)\|_{0,1}^2,$$
$$\tilde{K}_1=K_1+ 5 \sup_t \sum_k \Big( \|\s_{k,1}(t,.)\|_{L^\infty}^2 + \|\partial_3 \s_{k,1}(t,.)\|_{L^\infty}^2 \Big)$$
and finally
$$\tilde{K}_2=K_2+ 5 \sup_t \sum_k \Big( \|\s_{k,2}(t,.)\|_{L^\infty}^2 + \|\tilde{\s}_{k,2}(t,.)\|_{L^\infty}^2 +  \|\partial_3 \s_{k,2}(t,.)\|_{L^\infty}^2
+\|\partial_3 \tilde{\s}_{k,2}(t,.)\|_{L^\infty}^2 \Big).$$

\smallskip

\par

\noindent {\bf Example 2} The following example has some more general  Lipschitz structure.\\
For $t\in [0,T]$, $x\in \RR^3$, $y,y'\in \RR^3$  and $z,z'\in \RR^6$ set
\begin{align*} 
& |\s_k(t,x,y,z) - \s_k(t,x,y',z')| \leq C_{k,1}(t,x) |y-y'| + C_{k,2}(t,x) |z-z'|,\\
&|\partial_{x_3} \sigma_k(t,x,y,z)| \leq \tilde{C}_{k,0}(t,x) + \tilde{C}_{k,1}(t,x) |y| + \tilde{C}_{k,2}(t,x) |z|,
\end{align*}
where $\s_k(t,.,0,0)$ and $ \tilde{C}_{k,0} $ belong to $ L^2(\RR^3)$, while $C_{k,1}(t,.)$, $C_{k,2}(t,.)$, $\tilde{C}_{k,1}(t,.)$ and
$\tilde{C}_{k,2}(t,.)$ belong to $\big[ L^\infty(\RR^3)\big]^3$.  Moreover, we suppose that 
\begin{align*}
&\sup_{t\in [0,T]} \sum_{k\geq 1} \sup_{(x,y,z)\in \RR^{12}} |\nabla_y \s_k(t,x,y,z)|^2 = \tilde{C}_3<\infty,\\
&\sup_{t\in [0,T]} \sum_{k\geq 1} \sup_{(x,y,z)\in \RR^{12}} |\nabla_z \s_k(t,x,y,z)|^2 = \tilde{C}_4<\infty, 
\end{align*}
and 
\begin{align*}
\sup_{t\in [0,T]} &\sum_{k\geq 1} \Big( |\s_k(t,.,0,0)|_{L^2}^2 +  |\tilde{C}_{k,0}(t,.)|_{L^2}^2 \big) < \infty\\
\sup_{t\in [0,T]} &\sum_{k\geq 1} \Big(
\|C_{k,1}(t,.)\|_{L^\infty}^2 + \|C_{k,2}(t,.)\|_{L^\infty}^2  + \|\tilde{C}_{k,1}(t,.)\|_{L^\infty}^2 + \|\tilde{C}_{k,2}(t,.)\|_{L^\infty}^2  \big) <\infty.
%\\
%\sup_{t\in [0,T]} &\sum_{k\geq 1} \Big( |\partial_3 \s_k(t,.,0,0)|_{L^2}^2 + \|\partial_3 C_{k,1}(t,.)\|_{L^\infty}^2 + \|\partial_3 C_{k,2}(t,.)\|_{L^\infty}^2 
%\Big) <\infty. \\
\end{align*}
The  growth condition \eqref{growth_tilde_LQ} holds with:
\[ K_0=3\sup_t \sum_k |\s_k(t,.,0,0)|_{L^2}^2 , \; K_1=3\sup_t\sum_k \|C_{k,1}(t,.)\|_{L^\infty}^2, \; 
K_2=3\sup_t \sum_k \|C_{k,2}(t,.)\|_{L^\infty}^2.
\]
 The Lipschitz condition {\bf (C)(ii)} holds with $L_1=\frac{2}{3}K_1$ and $L_2=\frac{2}{3}K_2$. 
Taking partial derivatives with respect to $x_3$ yields that the growth condition \eqref{growth_LQ} is satisfied with: 
\begin{align*}
\tilde{K}_0=& K_0 + 5 \sup_t \sum_k |\tilde{C}_{k,0}(t,.)|_{L^2}^2, \\
\tilde{K}_1=& K_1+ 5 \tilde{C}_3 + \sup_t\sum_k \big( 3 \|C_{k,1}(t,.)\|_{L^\infty}^2 +  5   \|\tilde{C}_{k,1}(t,.)\|_{L^\infty}^2\big)  ,\\
\tilde{K}_2= & K_2+ 5 \tilde{C}_4 +  \sup_t \sum_k \big( 3 \|C_{k,2}(t,.)\|_{L^\infty}^2 + 5  \|\tilde{C} _{k,2}(t,.)\|_{L^\infty}^2 \big).
\end{align*}

\section{Large deviations} 
 Recall that the set of processes ${\mathcal X}$ has been defined in \eqref{calX} .
For $\epsilon >0$, let $u^\epsilon \in {\mathcal X}$ such that $u^\epsilon \in C([0,T]; H) $ a.s. 
 denote the solution of \eqref{saNS} where the noise intensity is multiplied by a small parameter $\epsilon >0$, that is 
% the solution $u^\epsilon\in X$:  %$u^\epsilon(0)=u_0$ and
\begin{equation}   \label{u_e}
u^\epsilon(t)=u_0 + \int_0^t \big[ \nu A_h u^\epsilon(s) - B(u^\epsilon(s)) - a |u^\epsilon(s)|^{2\alpha} u^\epsilon(s) \big] ds + \sqrt{\e} \int_0^t
\s(s,u^\epsilon(s)) dW(s).
\end{equation} 
For any constants $K_i$, $\tilde{K}_i$ and $\tilde{L}_i$ in Condition {\bf (C)}, for $\epsilon$ small enough there is a unique solution to \eqref{u_e}
which is denoted $u^\epsilon={\mathcal G}^\epsilon(\sqrt{\epsilon} W)$ for some Borel-measurable function ${\mathcal G}^\epsilon : C([0,T];\tilde{H}^{0,1}) 
\to X$.

In this section we prove that  $u^\epsilon$ satisfies a large deviations principle in the space $Y:= C([0,T];H)\cap L^2(0,T;\tilde{H}^{1,0})$.
%For technical reasons,  in all this section we will suppose that Condition {\bf {(C)}} holds with $K_2=\tilde{K}_2=\tilde{L}_2=0$. 
We use the weak convergence approach introduced in \cite{BD00} and \cite{BD07}.
We  at first prove apriori estimate for stochastic control equations deduced from
\eqref{saNS} by shifting $W$ by some random element.
 To describe a set of admissible random shifts,   we introduce the class
 $\mathcal{A}$ as the  set of $H_0-$valued
$(\cF_t)-$predictable stochastic processes $\phi$ such that
$\int_0^T |\phi(s)|^2_0 ds < \infty, \; $ a.s.
Let
\[S_M=\Big\{\phi \in L^2(0, T; H_0): \int_0^T |\phi(s)|^2_0 ds \leq M\Big\}.\]
The set $S_M$ endowed with the following weak topology is a
  Polish space (complete separable metric space)
\cite{BD07}:
$ d_1(\phi, \psi)=\sum_{i=1}^{\infty} \frac1{2^i} \big|\int_0^T \big(\phi(s)-\psi(s),
\tilde{e}_i(s)\big)_0 ds \big|,$
where $
\{\tilde{e}_i(s)\}_{i=1}^{\infty}$ is an  orthonormal basis
for $L^2(0, T; H_0)$.
Define
\begin{equation} \label{AM}
 \mathcal{A}_M=\{\phi \in \mathcal{A}: \phi(\om) \in
 S_M, \; a.s.\}.
\end{equation}

Let $\mathcal{B}(Y)$ denote the  Borel $\s-$field of the Polish space $Y$ endowed with the metric
 associated with the norm 
 \begin{equation} \label{norm_Y}
  \|u\|_Y=\sup_{t\in [0,T]} |u(t)|_{L^2} + \big( \int_0^T \|u(t)\|_{1,0}^2 ds \Big)^{\frac{1}{2}}.
  \end{equation} 
We recall some classical   definitions;
by convention the infimum over an empty set is  $ +\infty$.
\begin{defn}
   The random family
$(u^\e )$ is said to satisfy a large deviation principle on
$Y$  with the good rate function $I$ if the following conditions hold:\\
\indent \textbf{$I$ is a good rate function.} The function function $I: Y \to [0, \infty]$ is
such that for each $M\in [0,\infty[$ the level set $\{\phi \in Y: I(\phi) \leq M
\}$ is a    compact subset of $Y$. \\
 For $A\in \mathcal{B}(Y)$, set $I(A)=\inf_{u \in A} I(u)$.\\
\indent  \textbf{Large deviation upper bound.} For each closed subset
$F$ of $Y$:
$$
\lim\sup_{\e\to 0}\; \e \log \PX(u^\e \in F) \leq -I(F).
$$
\indent  \textbf{Large deviation lower bound.} For each open subset $G$
of $Y$:
$$
\lim\inf_{\e\to 0}\; \e \log \PX(u^\e \in G) \geq -I(G).
$$
\end{defn}

For all $\phi \in L^2(0, T ;  H_0)$, we will prove that there exists a unique solution let $u^0_\phi\in Y$  of
the deterministic control equation (\ref{dcontrol}) with initial condition
$u^0_\phi(0)=u_0\in L^4(\Omega, \tilde{ H}^{0,1})$:
\begin{eqnarray} \label{dcontrol}
d u^0_\phi(t) + [-\nu A_h u^0_\phi(t) +B(u^0_\phi(t)) + a |u^0_\phi(t)|^{2\alpha} u^0_\phi(t) ]dt =\s(t,u^0_\phi(t)) \phi(t) dt .
\end{eqnarray}
Let ${\mathcal C}_0=\{ \int_0^. \phi(s)ds \, :\, \phi \in L^2(0,T ;  H_0)\}  \subset C([0, T], H_0)$.
Define ${\mathcal G}^0:   C([0, T], H_0)  \to Y$ by
$ {\mathcal G}^0(\Phi)=u_\phi $ for $  \Phi=\int_0^. \phi(s)ds \in {\mathcal C}_0$
and ${\mathcal G}^0(\Phi)=0$ otherwise.

 Since  the argument below requires some information about the difference of the
solution at two different times, we need an additional assumption about the
regularity of the map $\sigma(.,u)$. Furthermore, for technical reasons,  we will suppose that condition {\bf (C)} holds with 
stronger growth and Lipschitz conditions, which forbid any gradient. This is summarized in the following:
\smallskip
\par

 \noindent  {\bf Condition (C')}\\
\indent {\bf (i)} ({\it Stronger growth and Lipschitz conditions}):  
The coefficient $\sigma$ satisfies condition {\bf (C)}  with the constants $K_2=\tilde{K}_2 =
L_2 =0$.\\ 
\indent {\bf (ii)}  ({\it Time H\"older regularity of $\sigma$}):
There exist  constants $\gamma>0$ and $C\geq 0$ such that for $t_1, t_2\in [0,T]$
and $u\in \tilde{H}^{1,0}$:
\[ |\sigma(t_1,u)-\sigma(t_2,u)|_{\mathcal L} \leq C \, \left( 1+ \| u\|_{1,0} \right) |t_1-t_2|^\gamma.\]

%\smallskip
%\par
The following theorem is the main result of this section.

\begin{theorem}\label{PGDue}
Suppose  that condition %s {\bf (C)}  with $K_2=\tilde{K}_2 = {L}_2=0$ and 
 {\bf (C')} is satisfied and that
$u_0\in \tilde{\mathcal H}^{0,1}$.
Then the solution  $(u^\e)$ to \eqref{u_e} satisfies the large deviation principle in
$Y=C([0, T]; H) \cap L^2(0, T;\tilde{  H}^{1,0})$,  with the good rate function
\begin{eqnarray} \label{ratefc}
 I_\xi (u)= \inf_{\{\phi \in L^2(0, T; H_0): \; u ={\mathcal G}^0(\int_0^. \phi(s)ds) \}}
 \Big\{\frac12 \int_0^T |\phi(s)|_0^2\,  ds \Big\}.
\end{eqnarray}
\end{theorem}
 
The proof relies on properties of a stochastic control equation. Let $M>0$, $\phi\in {\mathcal A}_M$ and $u_0\in L^4(\Omega; \tilde{H}^{0,1})$.
 Suppose that $\s$ satisfies condition  {\bf (C')(i)}   %with $K_2= \tilde{K}_2 =L_2=0$ 
 and consider the following non linear
SPDE with initial condition $u_\phi(0)=u_0$:
\begin{align}   \label{u_phi}
d_t u_\phi(t) + \big[ - \nu A_h \Delta_h u_\phi(t) + & B\big(u_\phi (t) \big) + a |u_\phi(t)|^{2\alpha} u_\phi(t)\big] dt  \nonumber \\
& = \s\big(t,u_\phi(t)\big)  dW(t)  + \s\big(t,u_\phi(t) \big) \phi(t)  dt.
\end{align}
The following theorem shows that  Theorem \ref{th_wp} holds in this setting. Its proof, which is similar to that
of Theorem \ref{th_wp} (see also Theorem 2.4 in \cite{CM}), is given in the appendix.
Note that the result would still be valid with "small enough" 
$K_2$, $\tilde{K}_2$ and $L_2$. However, some further arguments needed to prove the Large Deviations Principle require  these coefficients
to vanish.

\begin{theorem}\label{th4.1}
Let $\s$ satisfy condition   {\bf (C')(i)}
%with $K_2=0$ (resp. $\tilde{K}_2=0$) in the growth condition \eqref{growth_LQ} (resp. \eqref{growth_tilde_LQ}),
%and with ${L}_2=0$ in condition {\bf (C)(ii)}.
 Then for every
 $M>0$ and $T>0$  %there exists a constant $C(T,M)$ such that for
 and any ${\mathcal F}_0$-measurable  $u_0$ such
 that   $\EX \|u_0\|_{0,1}^4 < \infty$ and any 
 $\phi \in   \mathcal{A}_M$, %$K_2\in [0, \bar {K}_2[$ and $L_2<2$
there exists a   unique weak solution $u_\phi$  in ${\mathcal X} $ of the
equation  \eqref{u_phi} with
initial data  $u_\phi(0)=u_0 \in L^4(\Omega; \tilde{H}^{0,1})$. Furthermore,   $u_\phi \in C(0,T;H)$ a.s.
 and  there exists a constant
 $C:=C(K_0, K_1, \tilde{K}_0, \tilde{K}_1, T,M)$
 such that
for $\phi\in {\mathcal A}_M$,
\begin{equation} \label{apriori_phi}
 \EX\Big( \sup_{0\leq t\leq T}
 \|u_\phi(t)\|_{0,1}^4
+ \Big(\int_0^T \!\! \|u_\phi(t)\|_{1,1}^2\, dt\Big)^2 +\int_0^T\!\! \|u_\phi(t)\|_{L^{2\alpha +2}}^{2\alpha +2}  dt \Big) \leq C\, \big( 1+E\| u_0\|_{0,1}^4\big).
\end{equation}
  \end{theorem}

We next consider stochastic control evolution equations deduced from \eqref{u_e} by a random shift by a function $\phi \in {\mathcal A}_M$, that
is the solution $u^\epsilon_\phi$ to the evolution equation:
\begin{align}  \label{u_e_phi}
u^\epsilon_\phi(t) = & \, u_0 + \int_0^t \big[ \nu A_h u^\epsilon_\phi(s) - B(u^\epsilon_\phi(s)) - a|u^\epsilon_\phi(s)|^{2\alpha} u^\epsilon_\phi(s) 
+ \s(s,u^\epsilon_\phi(s)) \phi(s) \big] ds  \nonumber \\
& + \sqrt{\epsilon} \int_0^t \s(s,u^\epsilon_\phi(s)) dW(s). 
\end{align} 
Let $\e_0>0$,   $(\phi_\e , 0 < \e \leq \e_0)$  be a family of random elements
taking values in the set ${\mathcal A}_M$ given by \eqref{AM}.
Let  $u^\e_{\phi_\e}$, be
the solution of the corresponding stochastic control equation
 with initial condition $u^\e_{\phi_\e}(0)=u_0 \in \tilde{H}^{0,1}$:
\begin{align} \label{scontrol}
d_t u^\e_{\phi_\e}(t)  + [-\nu A_h u^\e_{h_\e} (t) + B(u^\e_{\phi_\e}(t)) & + a |u^\e_{\phi_\e}(t)|^{2\alpha} 
u^\e_{\phi_\e}(t) ]dt \nonumber \\
& 
=\s(t,u^\e_{\phi_\e}(t)) \big[  \phi_\e(t) dt+\sqrt{\e} \;  dW(t) \big] .
\end{align}
Note that for $W^\e_.=  W_. + \frac{1}{\sqrt \e}
 \int_0^. \phi_\e(s)ds$ we have   $u^\epsilon_{\phi_\e}={\mathcal G}^\e\big(\sqrt{\e} W^\e
 \big)$.
 
The following proposition establishes the weak convergence of the family $(u_{\phi_\e})$ as
$\e\to 0$. Its proof, which is similar to that of Proposition 4.3 in
\cite{DM} (see also Proposition 3.4 in \cite{CM}), is given in the appendix.  

\begin{prop}  \label{weakconv}
Suppose that  condition {\bf (C')} is satisfied.   %and {\bf(C')} are satisfied
%with $K_2=\tilde{K}_2= {L}_2=0$. 
Let $u_0 $ be
${\mathcal F}_0$-measurable such that $E\|u_0\|_{0,1}^4<+\infty$, and let
$\phi_\e$ converge to $\phi$ in distribution as random elements taking
values in ${\mathcal A}_M$, where this set is defined by \eqref{AM}
and endowed with the weak topology of the space $L_2(0,T;H_0)$. Then
as $\e \to 0$, the solution $u^\e_{\phi_\e}$ of \eqref{scontrol} converges
in distribution to the solution $u^0_\phi$ of \eqref{dcontrol} in
$Y=C([0, T]; H) \cap L^2(0,T;\tilde{H}^{1,0})$ endowed with the norm
\eqref{norm_Y}. That is, as $\e \to 0$,
 ${\mathcal G}^\e\Big(\sqrt{\e} \big( W_. + \frac{1}{\sqrt{\e}} \int_0^. \phi_\e(s)ds\big) \Big)$  converges in
distribution to $ {\mathcal G}^0\big(\int_0^. \phi(s)ds\big)$ in $Y$.
\end{prop}

The following compactness result
is the second ingredient which allows to transfer the  LDP  from $\sqrt{\e} W$ to $u^\e$.
Its proof is similar to that of Proposition \ref{weakconv}
and easier; it will be sketched in the appendix.  % (see also \cite{DM}, Proposition 4.4).
\begin{prop}  \label{compact}
Suppose that  condition {\bf (C')} holds. 
% and {\bf (C')} hold  with $K_2=\tilde{K}_2={L}_2=0$   hold. 
Fix  $M>0$, $u_0\in \tilde{ H}^{0,1}$ and let
$ K(M)=\{u^0_\phi \in X :  \phi \in S_M \}$,
where $u^0_\phi$ is the unique solution of the deterministic control
equation \eqref{dcontrol}, and let $Y  = C([0, T]; H) \cap L^2(0, T; \tilde{ H}^{1,0})$.
Then $K(M)$ is a compact subset of  $ Y$.
\end{prop}

Using the above results, we can complete the proof of the Large Deviations Principle for 
our stochastic Brinkman-Forchheimer 3D Navier-Stokes
equations. 

\noindent {\bf Proof of Theorem \ref{PGDue}:}
 Propositions \ref{compact} and  \ref{weakconv} imply that the family
$\{u^\e\}$ satisfies the Laplace principle,  which is equivalent
to the large deviation principle, in $Y$ with the good  rate function defined by \eqref{ratefc};
 see Theorem 4.4 in \cite{BD00} or
Theorem 5 in \cite{BD07}. This  concludes the proof of Theorem \ref{PGDue}.
\hfill $\Box$

\section{Appendix} \label{Appendix} 
The computations in this section are similar to the ones established for the stochastic equation \eqref{saNS}.
Equation \eqref{dcontrol}  is a particular case of equation \eqref{u_phi} and the proof of  the well posedness
of \eqref{u_phi} follows the steps  used to prove that  of \eqref{saNS}.  
However, for the sake of completeness, we show some of the estimates that  are performed for  
\eqref{u_phi} to show how  the extra term $\s\big(t,u_\phi(t) \big) \phi(t)$  with respect to \eqref{saNS} can be dealt with.

\subsection{A priori estimates for the stochastic control equation}
In this section we will only show how to obtain the estimates given in Theorem \ref{th4.1}. 
The argument is similar to that of Theorem \ref{th_wp} (see also Theorem 2.4 in  \cite{CM}).
We briefly sketch it only pointing out the changes to be made to deal with the random shift $\phi$.  

We at first consider an analog of \eqref{un_H}. For $t\in [0,T]$, $\phi\in {\mathcal A}_M$, $v\in {\mathcal H}_n$ and $u_{n,\phi}(0)=P_n u_0$, let
$u_{n,\phi}$ be defined on ${\mathcal H}_n$ as follows:
\begin{align}  \label{u_nphi}
d\big( u_{n,\phi}(t),v\big) = \langle F(u_{n,\phi}(t)\, , \, v\rangle dt + \big( P_n\s(t, u_{n,\phi}(t)) dW_n(t) ,v\big) 
+ \big(P_n  \s(t,u_{n,\phi}(t)) \Pi_n \phi(t)\, ,v) dt.
\end{align}
We check that an analog of \eqref{apriori} can be obtained for these processes with a constant $C$ which only depends on $M$
(but not on $\phi$ and $n$). We let $\tau_N=\inf\{ t : \|u_{n,\phi}(t)\|_{0,1} \geq N\} \wedge T$.

We apply the It\^o formula to $\| .\|_{0,1}^2$ and the process $u_{n,\phi}$.   This yields an equation similar to \eqref{Ito_2} where $u_n$ is
replaced by $u_{n,\phi}$, and where we add the term $T_4(t)$ in the right hand side, with
 \[ T_4(t) = 2\int_0^{t\wedge \tau_N} \big( \s(s,u_{n,\phi}(s)) \phi \, , \, u_{n,\phi}(s)\big) ds.
\] 
The growth condition \eqref{growth_LQ} with $\tilde{K}_2=0$, the Cauchy-Schwarz inequality, 
and the inequality
$|y|\leq 1+y^2$ imply
\begin{align*}
 |T_4(t)| \leq &\; 2\int_0^{t\wedge \tau_N} \Big[ \sqrt{\tilde{K}_0} + \sqrt{\tilde{K}_1} |u_{n,\phi}(s)|_{L^2} \Big]\, | \phi(s)|_0\, \|u_{n,\phi}(s)\|_{0,1} ds \\
\leq &\; 2\sqrt{\tilde{K}_0\, M\, T}  + 2\Big( \sqrt{\tilde{K}_0} + \sqrt{\tilde{K}_1}\Big) \int_0^{t\wedge \tau_N} |\phi(s)|_0\, \|u_{n,\phi}(s)\|_{0,1}^2 ds. 
\end{align*}
Fix $\epsilon>0$; as in the proof of Proposition \ref{Galerkin_n}, choose $\epsilon_0>0$ small enough to ensure 
$2C\epsilon_0 < 2\nu - \epsilon$, where $C$ is the constant
in the right hand side of \eqref{majF_v}, and then $\epsilon_1>0$ small enough to ensure $\frac{\epsilon_1}{4\epsilon_0} <2 a (2\alpha +1)-\epsilon$. 
Set
\begin{align*}
 X(t)=&\sup_{s\leq t\wedge \tau_N} \|u_{n,\phi}(s)\|_{0,1}^2 + 2a\int_0^{t\wedge \tau_N} \|u_{n,\phi}(s)\|_{L^{2\alpha +2}}^{2\alpha +2} ds,\\
Y(t)=&\int_0^{t\wedge \tau_N} \Big[ \big( |\nabla_h u_{n,\phi}(s)|_{L^2}^2 + |\partial_3 \nabla_h u_{n,\phi}(s)|_{L^2}^2 \big) 
+ \| u_{n,\phi}(s)\|_{L^{2\alpha +2}}^{2\alpha +2}\Big] ds.
\end{align*}
 
 For this choice of constants, we deduce that
 \[ X(t) + \epsilon Y(t) \leq Z + \int_0^t \varphi(r) X(r) dr + I(t),\]
 where $\varphi(r)= \tilde{K}_1 + C_\alpha \epsilon_0^{-1} \epsilon_1^{-\frac{1}{\alpha -1}} + 2 \big(\sqrt{\tilde{K}_1} + \sqrt{\tilde{K}_0}) |\phi(r)|_0$ and 
 \[ Z=\|u_0\|_{0,1}^2 + \tilde{K}_0 T + 2 \sqrt{\tilde{K}_0 T M}, \; I(t)=
 2\sup_{s\in [0,T]} \Big| \int_0^{s\wedge \tau_N} \!\!\big( \s(r,u_{\phi}(r)) dW(r)\, , \, u_{\phi}(r)
 \big)_{0,1} \Big|.\]
The Burkholder-Davies-Gundy inequality, the growth condition \eqref{growth_LQ} with $\tilde{K}_2=0$ and arguments similar 
to those in the proof of Proposition
\ref{Galerkin_n} imply that for $\beta\in (0,1)$, $\gamma=\frac{9}{\tilde{K}_1}$, $\tilde{C}= \frac{9}{\beta} \tilde{K}_0 T$ we have
\[ \EX\big( I(t)\big) \leq \beta \EX\big( X(t)\big) + \gamma \int_0^t \EX\big( X(s)\big) ds + \tilde{C} 
\]
Then $\int_0^T \varphi(s) ds \leq \tilde{K}_1 T +  
C_\alpha \epsilon_0^{-1} \epsilon_1^{-\frac{1}{\alpha -1}} T +  2\Big(\sqrt{\tilde{K}_1} + \sqrt{\tilde{K}_0}\Big) \sqrt{MT}:=C(1)$.

Since $\phi$ is random,  we need an extension of Gronwall's lemma (see \cite{DM}, Lemma
3.9 for the proof of a more general result). 

\begin{lemma} \label{lemGronwall}
Let $X$, $Y$,  $I$ and  $\varphi$ be  non-negative processes and
 $Z$ be a non-negative integrable random variable. Assume that
$I$ is non-decreasing and there exist non-negative constants
$C$, $\kappa, \beta, \gamma$ with the following  properties
\begin{equation} \label{Grw-cond}
\int_0^T \varphi(s)\, ds \leq C\; a.s., \quad 2\beta e^C\le 1, %\quad 2\delta e^C\le \alpha,
\end{equation}
  and  such that
 for $0\leq t\leq T$,
\begin{eqnarray*} %\label{XYZ}
X(t)+ \kappa Y(t) & \leq & Z + \int_0^t \varphi(r)\, X(r)\,  dr + I(t),\;
\mbox{\rm a.s.},\\
% \label{borneI}
\EX(I(t)) &\leq & \beta\, \EX (X(t)) + \gamma  \int_0^t \EX (X(s))\, ds
% + \delta\, \EX (Y(t))  
+ \tilde{C},
\end{eqnarray*}
where  $\tilde{C}>0$ is a constant.
If $ X \in L^\infty([0,T] \times \Omega)$, then
we have
\begin{equation} \label{Gronwall}
\EX \big[ X(t) + \kappa Y(t)\big]  \leq 2 \,  \exp\big( {C+2 t  \gamma  e^C }\big)\,
 \big( \EX(Z) +\tilde{C} \big), \quad t\in [0,T].
\end{equation}
\end{lemma}

Lemma \ref{lemGronwall} implies that for all $t\in [0,T]$ we have
 $ \EX\big( X(t) + \epsilon Y(t)\big) \leq  2 \exp( C(1) + 2t\gamma e^{C(1)})\big[ \EX Z + \tilde{C}]$.

Hence there exists a constant $C$, which only depends on $M,T$ and the constants %$K_i$ and 
$\tilde{K}_i$, $i=0,1$ in Condition {\bf(C)}, such that for every
$\phi\in {\mathcal A}_M$ 
\begin{equation}   \label{maj2_unphi}
 \EX \Big(\sup_{s\in [0,T]} \|u_{n,\phi}(s\wedge \tau_N)\|_{0,1}^2 
+ \int_0^{\tau_N} \!\! \big( \|u_{n,\phi}(s)\|_{1,1}^2 + \|u_{n,\phi}(s)\|_{L^{2\alpha +2}}^{2\alpha +2} \big) ds \Big)
 \leq C\big( 1+ \EX \|u_0\|_{0,1}^2\big).
 \end{equation} 
 
We then apply once more the It\^o formula to the square of $\|u_{n,\phi}\|_{0,1}^2$. This yields an upper estimate similar to \eqref{Ito_4} with $u_{n,\phi}$
instead of $u_n$,  and where we add $\tilde{T}_5(t)$ in the right hand side, with
\[ \tilde{T}_5(t)=4\int_0^{t\wedge \tau_N} \big( \s(s,u_{n,\phi}(s)) \phi(s)\, , \, u_{n,\phi}(s) \big)_{0,1} \|u_{n,\phi}(s)\|_{0,1}^2 ds.\]

Using the Cauchy-Schwarz inequality and the growth condition \eqref{growth_LQ} with $\tilde{K}_2=0$, we deduce that
\[ |\tilde{T}_5(t)| \leq 4 \int_0^{t\wedge \tau_N} \Big(\sqrt{\tilde{K}_1}+ \sqrt{\tilde{K}_0}\Big)  \|u_{n,\phi}(s)\|_{0,1}^4 |\phi(s)|_0 ds 
+ 4 \sqrt{\tilde{K}_0TM}.\]
Let 
\[ \bar{X}(t)=\sup_{s\in [0,t]} \|u_{n,\phi}(s\wedge \tau_N)\|_{0,1}^4, \; \bar{Y}(t)=\int_0^{t\wedge \tau_N} \!\!\|u_{n,\phi}(s)\|_{0,1}^2
\big( |\nabla_h u_{n,\phi}(s)|_{L^2}^2 + \partial_3 \nabla_h u_{n,\phi}(s)|_{L^2}^2 \big) ds.\]
Then choosing again $\epsilon_0$ and $\epsilon_1$ small enough, we deduce that for some $\epsilon >0$, 
\[ \bar{X}(t) + \epsilon \bar{Y}(t) \leq \bar{Z}+ \bar{I}(t) + \int_0^t \bar{\varphi}(s) \bar{X}(s) ds,\]
where $\bar{\varphi}(s)=6\tilde{K}_1 +4\big(\sqrt{\tilde{K}_0}+\sqrt{\tilde{K}_1}) |\phi(s) |_0$, $I(t)=\sup_{s\in [0,t]} \tilde{T}_2(s)$ for $\tilde{T}_2(s)$
 defined in \eqref{Ito_4}  and 
$ \bar{Z}=\sqrt{4\tilde{K}_0TM} + 6\tilde{K}_0 \int_0^{\tau_N} \|u_{n,\phi}(s)\|_{0,1}^2 ds $. 
Then $ \int_0^T \bar{\varphi}(s) ds \leq C(2):= 6 \tilde{K}_1 T + 4\big(\sqrt{\tilde{K}_0} + \sqrt{\tilde{K}_1}) \sqrt{TM}$. 
For $\bar{\beta}\in (0,1)$ and $\bar{\gamma}
= \frac{36}{\bar{\beta}} \tilde{K}_1$, we have $\EX \bar{I}(t) \leq \bar{\beta} \EX \bar{X}(t) + \bar{\gamma} \int_0^t \EX \bar{X}(s) ds + \bar{C}'$ where 
$\bar{C}'=\frac{36}{\beta} \EX\int_0^{\tau_N} \|u_{n,\phi}(s)\|_{0,1}^2 ds<\infty$ by \eqref{maj2_unphi}. 
Using once more Lemma \ref{lemGronwall} we deduce the existence of a constant $C$ depending  on $M$, $T$ and 
the constants $\tilde{K}_i$ in \eqref{growth_LQ}
such that
\begin{equation}  \label{maj4_unphi}
\EX \Big[ \sup_{t\in [0,T]} \|u_{n,\phi}(s\wedge \tau_N)\|_{0,1}^4 + \int_0^{\tau_N} \|u_{n,\phi}(s)|_{0,1}^2 \|u_{n,\phi}(s)_{1,1}^2 ds\Big]
\leq C\big( 1+\EX \|u_0\|_{0,1}^4\big). 
\end{equation}
holds for any $\phi\in {\mathcal A}_M$. 

This estimate being established, we follow the steps in the proof of Theorem \ref{th_wp} and prove that the weak limit $u_\phi$ of a proper subsequence
of the sequence $(u_{n,\phi}, n\geq 1)$ is a solution   of the evolution equation \eqref{u_phi}. 
 In order to   conclude the proof of  Theorem \ref{th4.1}, it remains only to prove the almost sure continuity of the process  $u_{\phi}$.

Let $W^\phi(t) = W(t) +\int_0^t \phi(s) ds$; the Girsanov theorem implies that $W^\phi$ is a Brownian motion 
under the probability  $\tilde{\mathbb P}$ with density
$\exp\big(-\int_0^t \phi(s) dW(s) - \frac{1}{2} \int_0^t |\phi(s)|_0^2 ds\big)$ with respect to ${\mathbb P}$ on ${\mathcal F}_t$.
 Under $\tilde{\mathbb P}$ the process $u_\phi$ is the unique solution
to the evolution equation \eqref{saNS} in ${\mathcal X}$ and belongs $\tilde{\mathbb P}$ a.s. to $C([0,T]:H)$. Since the probabilities $\tilde{\mathbb P}$ and 
${\mathbb P}$  are equivalent and this completes the proof of  Theorem \ref{th4.1}.

\subsection{Weak convergence of the stochastic control equations (Proposition \ref{weakconv}). }   \label{WC}
%To establish the large deviation principle, we need to strengthen the hypothesis on the
%growth condition and Lipschitz property of $\sigma$ by requiring that $K_2=L_2=0$.
%This stronger assumption is needed to
We at first prove   the following
technical lemma, which studies time increments of the solution to
the stochastic control problem  \eqref{u_e_phi}. To state the lemma mentioned above,  we need the following notations.
For every integer $n$, let $\psi_n : [0,T]\to [0,T]$ denote a measurable map
such that for every $s\in [0,T]$,
$s\leq \psi_n(s) \leq \big(s+c2^{-n})\wedge T$ for some positive constant $c$.
 Given $N>0$,  $\phi\in {\mathcal A}_M$,
 and for  $t\in [0,T]$,  let
\[ G_N(t)=\Big\{ \omega \, :\, \Big (\sup_{0\leq s\leq t}  |u_\phi^\e(s)(\omega)|_{L^2}^2 \Big)\vee
\Big(  \int_0^t \|u_\phi^\e(s)(\omega)\|_{1,1}^2 ds \Big) \vee \Big(\int_0^t \|u^\e_\phi(s)\|_{L^{2\alpha +2}}^{2\alpha +2} ds \Big) \leq N\Big\}.\]
\begin{lemma} \label{timeincrement}
Let $\e_0, M,N>0$, $\sigma$
satisfy condition {\bf (C')(i)}.  % with $K_2=\tilde{K}_2={L}_2=0$. 
Let  $u_0\in L^4(\Om;\tilde{H}^{0,1})$ be ${\mathcal F}_0$-measurable,
and let $u_\phi^\e(t)$
be solution  of \eqref{u_e_phi}.
 Then there exists a positive  constant
$C$ (depending on
$K_i,  \tilde{K}_i,i=0,1,  {L}_1, T, M, N, \e_0$)
 such that for
 any  $\phi\in {\mathcal A}_M$, $\e\in [0, \e_0]$:
\begin{align}  \label{time}
I_n(\phi,\e):=& \EX\Big[ 1_{G_N(T)}\;  \int_0^T  \Big\{ |u_{\phi}^\e(s)-u_{\phi}^\e(\psi_n(s))|_{L^2}^2  \nonumber  \\
& \qquad + \int_s^{\psi_n(s)} \big( |\nabla_h u^\e_\phi(r)|_{L^2}^2 dr
+  \| u^\e_\phi(r)\|_{L^{2\alpha+2}}^{2\alpha +2} \big) dr \Big\}\, ds\Big]
\leq C\, 2^{-\frac{n}{2}}.
\end{align}
\end{lemma}

\begin{proof} The proof is  close to that of Lemma 3.3 in  \cite{CM}.
%To lighten the notation we skip the time dependence of $\sigma$.
Let $\phi\in {\mathcal A}_M$, $\e\geq 0$;
for any $s\in [0,T]$, It\^o's formula yields
 $|u_\phi^\e(\psi_n(s))-u_\phi^\e(s)|_{L^2}^2 = \sum_{i=1}^6 I_{n,i}$, where
\begin{eqnarray*}
I_{n,1}&=&2\, \sqrt{\e}\;  \EX\Big( 1_{G_N(T)} \int_0^T \!\! ds \int_s^{\psi_n(s)}\!  \big(
\sigma(r, u_\phi^\e(r)) dW(r) \, , \, u_\phi^\e(r)-u_\phi^\e(s) \big)\Big) , \\
I_{n,2}&=&{\e}\;  \EX \Big( 1_{G_N(T)} \int_0^T  \!\!ds \int_s^{\psi_n(s)} \!\!
|\sigma(r, u_\phi^\e(r))|_{\mathcal L}^2 \, dr\Big) , \\
I_{n,3}&=&2 \,  \EX \Big( 1_{G_N(T)} \int_0^T \!\! ds \int_s^{\psi_n(s)} \!\! \big(
{\sigma}(r, u_\phi^\e(r)) \, \phi(r)\,  , \, u_\phi^\e(r)-u_\phi^\e(s) \big)\, dr\Big), \\
I_{n,4}&=&2  \nu\,  \EX \Big( 1_{G_N(T)}   \int_0^T  \!\! ds \int_s^{\psi_n(s)} \!\!
 \big\langle \Delta_h \, u_\phi^\e(r)\, , \, u_\phi^\e(r)-u_\phi^\e(s)\big\rangle  \, dr\Big) ,  \\
I_{n,5}&=&-2 \, \EX \Big( 1_{G_N(T)} \int_0^T  \!\! ds \int_s^{\psi_n(s)} \!\!
  \big\langle B( u_\phi^\e(r))\, , \,   u_\phi^\e(r)-u_\phi^\e(s)\big\rangle \, dr\Big) ,  \\
I_{n,6}&=&- 2  a \, \EX \Big( 1_{G_N(T)} \int_0^T \!\!  ds \int_s^{\psi_n(s)} \!\!
  \int_{\RR^3} \!\! |u^\e_\phi(r,x)|^{2\alpha} u^\e_\phi(r,x) \big( u^\e_\phi(r,x)-u^\e_\phi(s,x)\big)\, dx \, dr\Big) .
\end{eqnarray*}
Clearly  $G_N(T)\subset G_N(r)$ for $r\in [0,T]$.
In particular this means that  $|u_\phi^\e(r)|^2_{L^2}+|u_\phi^\e(s)|^2_{L^2} \le N$
%in each integral $I_{n,j}$.
 on $G_N(r)$ for $0\leq s\leq r\leq T$.
 We use this observation in the considerations
below.
\par
\noindent The
Burkholder-Davis-Gundy inequality and the growth condition \eqref{growth_tilde_LQ} yield for  $\e \in [0, \e_0]$:
\begin{eqnarray*}
|I_{n,1}|&\leq &
 6\sqrt{\e} \int_0^T ds \; \EX \Big( %1_{G_N(T)}
 \int_s^{\psi_n(s)} |\s(r, u_\phi^\e(r))|_{\mathcal L}^2
 1_{G_N(r)}\, | u_\phi^\e(r)- u_\phi^\e(s)|^2 \;  dr \Big)^{\frac{1}{2}} \\ %\nonumber \\
&\leq & 6 \sqrt{2 \e_0 N }  \int_0^T ds \; \EX \Big( %1_{G_N(T)}
 \int_s^{\psi_n(s)}
[{K}_0+{K}_1\,  |u_\phi^\e(r)|_{L_2}^2  ]\;   \, dr \Big)^{\frac{1}{2}}.
\end{eqnarray*}
Schwarz's inequality and Fubini's theorem as well as \eqref{apriori_phi}, which holds
uniformly in  $\e\in ]0,\e_0]$ for fixed $\e_0>0$ (since the constants ${K}_i$ and ${L}_1$
are multiplied by at most  $\e_0$),  imply
\begin{equation} \label{In1}
|I_{n,1}|\leq 
6 \sqrt{2 \e_0 N T }\, \Big[ \EX \int_0^T\!\! %1_{G_N(T)}\,
\big( {K}_0+{K}_1\, |u_\phi^\e(r)|_{L^2}^2 \big)\, \Big( \int_{(r-c2^{-n})\vee 0}^r ds\Big) \, dr
 \Big]^{\frac{1}{2}} %\nonumber \\
%&\leq & 
\leq C_1 2^{-\frac{n}{2}}
\end{equation}
for some constant $C_1$ depending only on  %$\EX\|u_0\|_{0,1}^4$,  $K_i, 
${K}_i, i=0,1$, %$L_j$, $\tilde{L}_j, j=1,2$,
 $M$,
 $\e_0$, $N$ and $T$.
The growth condition \eqref{growth_tilde_LQ}  and Fubini's theorem  imply that for $\e\in [0,\e_0]$:
\begin{eqnarray} \label{In2}
|I_{n,2}|&\leq &\e_0  \,   \EX \Big( 1_{G_N(T)}   \int_0^T \!\! ds \int_s^{\psi_n(s)} \!\!
\big({K}_0+{K}_1|u_\phi^\e(r)|_{L^2}^2  \big) dr\Big)  %\nonumber \\
%&\leq &
%\e_0\,   \EX \int_0^T  1_{G_N(T)} \, ({K}_0 +{K}_1\, N  )\, c 2^{-n}\, dr \;
\leq\;  C_2 2^{-n}
\end{eqnarray}
for some constant $C_2$ depending on the same parameters as $C_1$.  %$K_i$, $i=0,1,2$,  $\e_0$, $N$ and $T$.
The Cauchy-Schwarz  inequality, Fubini's theorem,  the growth condition \eqref{growth_tilde_LQ}   and the definition %\eqref{AM} 
of ${\mathcal A}_M$ yield
\begin{align}  \label{In3}
|I_{n,3}|&\leq 2   \;  \EX \Big( 1_{G_N(T)} \int_0^T \!\! ds  %\nonumber \\
%& {}\quad \times 
\int_s^{\psi_n(s)} \!\!
 \big({K}_0 +{K}_1|u_\phi^\e(r)|_{L^2}^2 \big)^{\frac{1}{2}}\, |\phi(r)|_0 |\, u_\phi^\e(r)-u_\phi^\e(s)|_{L^2}\, dr\Big)
\nonumber \\
& \leq  4 \sqrt{N}  \; \EX  \int_{0}^T 1_{G_N(T)}
 |\phi(r)|_0 ({K}_0 +{K}_1 N )^{\frac{1}{2}}
 \,
  \Big( \int_{(r-c2^{-n})\vee 0}^r ds \Big)\, dr   \leq \; C_3\,  2^{-n},
 % \nonumber \\
%& \leq 4 \sqrt{N}\, c 2^{-n} \sqrt{M}\,
 %%\EX  \Big( 1_{G_N(T)}\int_0^T (\tilde{K}_0 +\tilde{K}_1 N  )  \, dr \Big)^{\frac{1}{2}}
%\sqrt{T} \big( {K}_0 + {K}_1 N\big)^{\frac{1}{2}} \, 
%\leq \; C_3\,  2^{-n},
\end{align}
for some constant $C_3$ depending on the same parameters as $C_1$.
Using the Cauchy-Schwarz  inequality we deduce that
\begin{eqnarray} \label{In4}
|I_{n,4}|  &= &  2  \Big|  \EX \Big( 1_{G_N(T)} \int_0^T  \!\! \! ds \int_s^{\psi_n(s)} \!\!\!
dr  \big[ - |\nabla_h u_\phi^\e(r)|_{L^2}^2   +  |\nabla_h u_\phi^\e(r)|_{L^2} |\nabla_h u_\phi^\e(s)|_{L^2} \big]\Big) \Big|  \nonumber \\
&\leq &
 \frac{1}{2}\;   \EX \Big( 1_{G_N(T)} \int_0^T ds \; |\nabla_h u^\e_\phi(s)|_{L^2}^2 \,
\int_s^{\psi_n(s)} dr \Big) \leq   C \;  N \;    2^{-n}.
\end{eqnarray}
The antisymmetry relation \eqref{antisymetry}, the inequality \eqref{maj_B}, the Cauchy-Scwarz inequality and Fubini's theorem and inequality yield:
\begin{align} \label{In5}
|I_{n,5}| &\leq 2 \EX \Big( 1_{G_N(T)} \int_0^T ds \int_s^{\psi_n(s)} dr \big| \langle B(u^\e_\phi(r)) , u^\e_\phi(s) \rangle \big| \Big)
\nonumber \\
%& \leq \EX \Big( 1_{G_N(T)}\!\! \int_0^T  \!\!\! ds
%\int_s^{\psi_n(s)} \!\!\! dr \|u^\e_\phi(r)\|_{1,1} \, |\nabla_h u^\e_\phi(r)|_{L^2} \, |u^\e_\phi(r)|_{L^2} \Big) \nonumber \\
& \leq CN \EX \Big[ 1_{G_N(T)} \Big( \int_0^T \|u^\e_\phi(s)\|_{1,1}^2 ds \Big)^{\frac{1}{2}} 
\Big( \int_0^T \Big( \int_s^{\psi_n(s)} |\nabla_h u^\e_\phi(r)|_{L_2}
dr \Big) ds \Big)^{\frac{1}{2}} \Big]  \nonumber \\
%&\leq C  N^{\frac{3}{2}} 2^{-\frac{n}{2}}  \EX\Big[  1_{G_N(T)} \int_0^T ds \Big(  \int_s^{\psi_n(s)} |\nabla_h u^\e_\phi(r)|_{L^2}^2 dr \Big)^{\frac{1}{2}} \Big]
%\nonumber \\
& \leq C(T) N^{\frac{3}{2}} 2^{-\frac{n}{2}} \EX \Big[ 1_{G_N(T)} 
\Big\{ \int_0^T dr \Big( \int_{(r-c2^{-n})\vee 0}^r ds \Big) |\nabla_h u^\e_\phi(r)|_{L^2}^2 \Big\}^{\frac{1}{2}} \Big]
\leq C_5 2^{-n}
\end{align}
for some constant $C_5$ which depends on $T$ and $N$.

Finally, Fubini's theorem and H\"older's inequality imply:
\begin{align} \label{In6}
& |I_{n,6}| \leq  2 a \, \EX\Big[ 1_{G_N(T)} \int_0^T\!\!ds \int_s^{\psi_n(s)} \!\! dr  \int_{\RR^3} \big( |u^\e_\phi(r)|^{2\alpha+2} 
+ |u^\e_\phi(r)|^{2\alpha +1} |u^\e_\phi(s)|\big) dx \Big]  \nonumber \\
&\quad \leq 2a \, \EX\Big[ 1_{G_N(T)} \int_0^Tds \int_s^{\psi_n(s)}  \|u^\e_\phi(s)\|_{L^{2\alpha+2}}\, \|u^\e_\phi(r)\|_{L^{2\alpha+2}}^{2\alpha +1} dr \Big] 
\nonumber \\
& \qquad + 
2a \, \EX\Big[ 1_{G_N(T)} \int_0^T \|u^\e_\phi(r)\|_{L^{2\alpha +2}}^{2\alpha +2} \Big( \int_{(r-c2^{-n})\vee 0}^r ds\Big) dr \Big] \nonumber \\
%& \quad \leq2a\, \EX\Big[ 1_{G_N(T)} \int_0^T dr \|u^\e_\phi(r)\|_{L^{2\alpha+2}}^{2\alpha +1} \big(c2^{-n}\big)^{\frac{2\alpha+1}{2\alpha+2}} 
%\Big(\int_{(r-c2^{-n})\vee 0}^r \|u^\e_\phi(s)\|_{L^{2\alpha+2}}^{2\alpha+2} ds \Big)^{\frac{1}{2\alpha+2}} \Big]
%\nonumber \\
&\qquad +  2ac2^{-n} N 
\nonumber \\
&\quad \leq 2a  \big(c2^{-n}\big)^{\frac{2\alpha+1}{2\alpha+2}} \, 
\EX\Big[ 1_{G_N(T)}\Big(  \int_0^T\! \|u^\e_\phi(r)\|_{L^{2\alpha+2}}^{2\alpha+2} \Big)^{\frac{2\alpha+1} {2\alpha+2}}
\Big( \int_0^T \! ds \|u^\e_\phi(s)\|_{L^{2\alpha+2}}^{2\alpha +2} \int_s^{\psi_n(s)}\!\!  dr\Big)^{\frac{1}{2\alpha+2}}  \Big] 
\nonumber \\
&\qquad + 2ac2^{-n} N \; \leq \; C_6 2^{-n}
\end{align}
for some constant $C_6$ depending on $T$ and $N$. 
Collecting the upper estimates from \eqref{In1}-\eqref{In6},
we conclude the proof of \eqref{time}.
\end{proof}

 In the setting of large deviations, we will use  Lemma \ref{timeincrement}
 with the following choice of the function $\psi_n$.
 For any integer $n$  define a step function $s\mapsto \bar{s}_n$
on  $[0,T]$ by the formula
\begin{equation} \label{step-f}
\bar{s}_n= t_{k+1}\equiv(k+1)T 2^{-n}~~\mbox{ for }~~s\in [k T 2^{-n}, (k+1) T 2^{-n}[.
\end{equation}
Then the map $\psi_n(s)=\bar{s}_n$ clearly satisfies the previous requirements with $c=T$.
%\end{remark}
\par

\noindent \textbf { \textit {Proof of Proposition \ref{weakconv}}} 

Now we return to the setting of  this proposition and recall that  for  random elements $(\phi_\e , 0 < \e \leq \e_0)$ 
taking values in the set ${\mathcal A}_M$, we let $u^\e_{\phi_\e}$ denote the solution to \eqref{scontrol} 
with initial condition $u^\e_{\phi_{\e}}  (0)= u_0\in \tilde{H}^{0,1}$.
%\begin{align*} 
%d_t u^\e_{\phi_\e}(t)  + [-\nu A_h u^\e_{h_\e} (t) + B(u^\e_{\phi_\e}(t)) & + a |u^\e_{\phi_\e}(t)|^{2\alpha} 
%u^\e_{\phi_\e}(t) ]dt \\
%& 
%=\s(t,u^\e_{\phi_\e}(t)) \big[  \phi_\e(t) dt+\sqrt{\e} \;  dW(t) \big] .
%\end{align*}

 Since ${\mathcal A}_M$ is a Polish space (complete
separable metric space), by the Skorokhod representation theorem, we
can construct processes $(\tilde{\phi}_\e, \tilde{\phi}, \tilde{W}^\e)$
such that the joint distribution of $(\tilde{\phi}_\e,  \tilde{W}^\e)$
is the same as that of $(\phi_\e, W^\e)$,  the distribution of  
$\tilde{\phi}$ coincides with that of $\phi$, and $ \tilde{\phi}_\e  \to
\tilde{\phi}$,  a.s., in the (weak) topology of $S_M$.   Hence a.s. for
every $t\in [0,T]$, $\int_0^t \tilde{\phi}_\e(s) ds - \int_0^t
\tilde{\phi}(s)ds \to 0$ weakly in $H_0$. To lighten notations, we will
write $(\tilde{\phi}_\e, \tilde{\phi}, \tilde{W}^\e)=(\phi_\e,\phi,W)$.
Let $U_\e=u^\e_{\phi_\e}-u^0_\phi$; then  $U_\e(0)=0$ and
\begin{align}\label{difference2}
d U_\e(t)  = & \big[ F(u^\e_{\phi_\e}(t)) - F(u^0_\phi(t))  %\big[AU_\e & +B(u_{h_\e})-B(u_h)+\tilde R(t,u_{h_\e})-\tilde R(t,u_h)\big]dt
+ \s\big(t,u^\e_{\phi_\e}(t)\big) \phi_\e(t) -\s\big(t,u^0_\phi(t)\big) \phi(t)\big] dt  
\nonumber\\
 & +\sqrt{\e} \s\big(t,u^\e_{\phi_\e}(t)\big) dW(t).
\end{align}
Let $\eta \in (0,\nu)$ and $C_\eta$ be defined in \eqref{upper_F-F_t};
 It\^o's formula, the upper estimate \eqref{upper_F-F_t}, the growth condition \eqref{growth_tilde_LQ}
  and the Lipschitz condition {\bf (C')(i)} imply for $t\in [0,T]$:
\begin{align} \label{total-error}
&  |U_\e(t)|_{L^2}^2 +2\eta \int_0^t |\nabla_h U_\e(s)|_{L^2}^2 ds 
+2a\kappa \int_0^t \big| \big( |u^\e_{\phi_\e}(s)| + |u^0_\phi(s)|\big)^\alpha |U_\e (s)| \big|_{L^2}^2 ds 
\nonumber \\
& \quad   \leq 
  \sum_{i=1}^3 T_i(t,\e) + 2   \int_0^t\!\!
    \big( {C}_{\eta} \, \|u^0_\phi(s)\|_{1,1}^2  + \sqrt{L_1}  |\phi_\e(s)|_0\big) |U_\e(s)|_{L^2}^2   ds,
\end{align}
where
\begin{eqnarray*}
T_1(t,\e)&=& 2\sqrt{\e}\int_0^t \big( U_\e(s), \s(s,u^\e_{\phi_\e}(s))\,  dW(s)  \big), \\
T_2(t,\e)&= & \e   \int_0^t ({K}_0+{K}_1|u^\e_{\phi_\e}(s)|_{L^2}^2) ds , \\
T_3(t,\e)&=& 2\int_0^t \Big( \s(s,u^0_\phi(s))\, \big( \phi_{\e}(s)-\phi(s)\big), \, U_\e(s)\Big)\, ds.
\end{eqnarray*}
We want  to show that as $\e \to 0$,  $\|U_\e\|_Y \to 0$ in
probability, which implies
 that $u^\e_{h_\e} \to u_h$  in distribution in $Y$.
 Fix $N>0$ and for $t\in [0,T]$ let
\begin{align*}  G_N(t)=&\Big\{ \sup_{0\leq s\leq t} |u^0_\phi(s)|_{L^2}^2 \leq N \Big\} \cap
 \Big\{ \int_0^t \big( \|u^0_\phi(s)\|_{1,1}^2  + \|u^0_\phi(s)\|_{L^{2\alpha+2}}^{2\alpha+2}\big) ds \leq N
\Big\}, \\
G_{N,\e}(t)=&  G_N(t)\cap \Big\{ \sup_{0\leq s\leq t} \big( |u^\e_{\phi_\e}(s)|_{L^2}^2 \leq N\Big\} \cap
 \Big\{ \int_0^t \big( \|u^\e_{\phi_\e}(s)\|_{1,1}^2  + \|u^\e_{\phi_\e}(s)\|_{L^{2\alpha +2}}^{2\alpha +2} \big) ds \leq N
\Big\}.
\end{align*}
The proof consists in two steps.\\
\textbf{Step 1:}
For any $\e_0 \in ]0,1]$, we have $ {\displaystyle \sup_{0<\e\leq \e_0}\; \sup_{\phi,\phi_\e \in {\mathcal A}_M}
\PX(G_{N,\e}(T)^c )\to 0 \;
\mbox{\rm as }\; N\to \infty.}$
\\
Indeed, for $\e\in ]0,\e_0]$, $\phi,\phi_\e \in {\mathcal A}_M$, the Markov inequality and
the a priori estimate  \eqref{apriori_phi}, which holds uniformly in $\e\in ]0,\e_0]$, imply
\begin{align}\label{G-c}
&  \PX ( G_{N,\e}(T)^c )
\leq   \PX \Big(\sup_{0\leq s\leq T} |u^0_\phi(s)|_{L^2}^2 > N \Big)
+  \PX\  \Big( \int_0^T\! \big(  \|u^0_\phi(s)\|_{1,1}^2 + \| u^0_\phi(s)\|_{L^{2\alpha+2}}^{2\alpha+2}\big) ds  >N\Big)
\nonumber \\
&\quad +\PX \Big(\sup_{0\leq
s\leq T} |u^\e_{\phi_\e}(s)|_{L_2}^2 > N \Big) %\nonumber \\
%&\qquad \qquad\qquad   
+ \PX\Big(\int_0^T \big( \|u^\e_{\phi_\e}(s)\|_{1,1}^2 + \| u^\e_{\phi_e} (s)\|_{L^{2\alpha+2}}^{2\alpha+2}\big)ds > N \Big)\nonumber \\
%&\leq  \frac{1}{N}
%%  \sup_{ \;h, h_\e \in {\mathcal A}_M}
% \EX
%\Big( \sup_{0\leq s\leq T} |u_h(s)|^2 + \sup_{0\leq s\leq T}
%|u_{h_\e}(s)|^2
%+ \int_0^T(\|u_h(s)\|^2+ \|u_{h_\e}(s)\|^2 )ds \Big) \nonumber \\
&\leq
 {C \, \big(1+ E\|u_0\|_{0,1}^4\big)}{N}^{-1},
\end{align}
for some constant $C$ depending on $T$ and $M$.

\noindent \textbf{Step 2:} Fix $N>0$,
 $\phi, \phi_\e \in {\mathcal A}_M$  such that  as $\e \to 0$, $\phi_\e\to \phi$ a.s. in the weak topology
 of $L^2(0,T;H_0)$; then  one has  as $\e \to 0$:
\begin{equation} \label{cv1}
\EX\Big[ 1_{G_{N,\e}(T)} \Big( \sup_{0\leq t\leq T } |U_\e(t)|_{L^2}^2 + \int_0^T  |\nabla_h U_\e(t)|_{L^2}^2 \, dt\Big)
\Big] \to 0.
\end{equation}
Indeed,  \eqref{total-error}  and Gronwall's lemma imply that on $G_{N,\e}(T)$,
\[ \sup_{0\leq t\leq T} |U_\e(t)|_{L^2}^2 \leq \Big[ \sup_{0\leq t\leq T}
 \big( T_1(t,\e) + T_3(t,\e)\big) +
\e C_*\Big] \, \exp\Big(2{C}_{\eta} N +2\sqrt{L_1 M T}\Big),\]
 where $C_*= T({K}_0+{K}_1N)$.
Using again \eqref{total-error}  we deduce that for some constant
$\tilde{C}=C(T,M,N)$,
 one has for every $\e\in [0,\e_0]$:
\begin{equation} \label{*}
\EX\big( 1_{G_{N,\e}(T)} \, \|U_\e\|_Y^2 \big) \leq \tilde{C} \Big( \e  + \EX \Big[ 1_{G_{N,\e}(T)} \,
\sup_{0\leq t\leq T} \big( T_1(t,\e) + T_3(t,\e)\big) \Big] \Big).
\end{equation}
Since the sets $G_{N,\e}(.)$ decrease, $\EX\big(1_{G_{N,\e}(T)}
 \sup_{0\leq t\leq T} |T_1(t,\e)|\big) \leq
\EX(\lambda_\e)$, where
\[ \lambda_\e := 2\sqrt{\e} \sup_{0\leq t\leq T }
\Big|\int_0^t 1_{G_{N,\e}(s)} \Big( U_\e(s), \s(s,u_{\phi_\e}(s)) dW(s) \Big)\Big|.\]
The scalar-valued random variables $\lambda_\e$ converge to 0 in $L^1$ as $\e\to 0$.
Indeed, by the Burkholder-Davis-Gundy inequality, \eqref{growth_tilde_LQ} and the definition of $G_{N,\e}(s)$,
we have
\begin{align}
\EX (\lambda_\e) & \leq  6\sqrt{\e} \; \EX
\Big\{\int_0^T 1_{G_{N,\e}(s)} \,  |U_\e(s)|_{L^2}^2 \;
 |\s(s, u^\e_{\phi_\e}(s))|^2_{\mathcal L}s
 ds\Big\}^\frac12 \nonumber \\
&\leq  6\sqrt{\e} \; \EX \Big[  \Big\{ 4N \int_0^T 1_{G_{N,\e}(s)}\,
 ({K}_0+ {K}_1| u^\e_{\phi_\e}(s) |_{L^2}^2 )  ds\Big\}^\frac12 \Big]
\leq  C(T,N) \,   \sqrt{\e}. \label{lambda1}
\end{align}
\par
 In further estimates we use Lemma~\ref{timeincrement} with $\psi_n=\bar{s}_n$, where
$\bar{s}_n$ is  defined in \eqref{step-f}.
For any $n,N \geq 1$, if we set $t_k=kT2^{-n}$ for $0\leq k\leq 2^n$,  we obviously have:
\begin{equation}\label{t3}
\EX\Big( 1_{G_{N,\e}(T)}\sup_{0\leq t\leq T} |T_3(t,\e)| \Big)
\leq 2\;  \sum_{i=1}^4  \tilde{T}_i(N,n, \e)+ 2 \; \EX \big( \bar{T}_5(N,n,\e)\big),
\end{equation}
 where
\begin{align*}
\tilde{T}_1&(N,n,\e)= \EX \Big[ 1_{G_{N,\e}(T)} \sup_{0\leq t\leq T}  \Big| \int_0^t  \Big(  \s(s,u^0_\phi(s))
  \big(\phi_\e(s)-\phi(s)\big) \, ,\,
\big[U_\e(s)-U_\e(\bar{s}_n)\big] \Big)  ds\Big| \Big] ,\\
\tilde{T}_2 &(N,n,\e)= \EX\Big[ 1_{G_{N,\e}(T)} \\ & {}\quad
\times\sup_{0\leq t\leq T} \Big| \int_0^t
\Big( [\s(s,u^0_\phi(s)) - \s(\bar{s}_n, u^0_\phi(s))] (\phi_\e(s)-\phi(s))\, ,\, U_\e(\bar{s}_n)\Big)ds \Big|\Big], \\
\tilde{T}_3 &(N,n,\e)= \EX\Big[  1_{G_{N,\e}(T)}
\\ & {}\quad \times
\sup_{0\leq t\leq T} \Big| \int_0^t \Big( \big[ \s(\bar{s}_n,u^0_\phi(s))
- \s(\bar{s}_n, u^0_\phi(\bar{s}_n))\big]
\big(\phi_\e(s) - \phi(s) \big)\, ,\, U_\e(\bar{s}_n)\Big) ds\Big| \Big] ,\\
\tilde{T}_4&(N,n,\e)= \EX \Big[  1_{G_{N,\e}(T)} \sup_{1\leq k\leq 2^n}  \sup_{t_{k-1}\leq t\leq t_k}
\Big| \Big( \sigma(t_k, u^0_\phi(t_k))  \int_{t_{k-1}}^t\!\!  (\phi_\e(s)-\phi(s))\, ds \, ,\, U_\e(t_k)\Big) \Big| \Big], \\
\bar{T}_5&(N,n, \e)=  1_{G_{N,\e}(T)} \sum_{k=1}^{2^n} \Big| \Big( \s(t_k, u^0_\phi(t_k))
\int_{t_{k-1}}^{t_k }  \!\! \big(\phi_\e(s)-\phi(s)\big)\, ds \, ,\,
U_\e(t_k )\Big) \Big| .
\end{align*}
%Here above $\bar{s}_n$ is the step function of $s$ defined in \eqref{step-f}
%and $t_k=kT 2^{-n}$.
Using the Cauchy-Schwarz inequality, the growth condition \eqref{growth_tilde_LQ} and Lemma~\ref{timeincrement}
with $\psi_n=\bar{s}_n$,
we deduce that  for some constant
$\bar{C}_1:= C(T,M,N)$ and any   $\e \in ]0, \e_0]$:
\begin{align} \label{eqT1}
&  \tilde{T}_1(N,n,\e)\leq
 \EX\Big[ 1_{G_{N,\e}(T)}  \int_0^T
 \big( {K}_0+{K}_1|u^0_\phi(s)|_{L^2}^2\big)^{\frac{1}{2}}
|\phi_\e(s)-\phi(s)|_0\, \big| U_\e(s)-U_\e(\bar{s}_n)\big|_{L^2}\, ds\Big]  \nonumber \\
& \quad   \leq
\Big(  \EX \Big[ 1_{G_{N,\e}(T)}  \int_0^T  \big\{ |u^\e_{\phi_\e}(s) - u^\e_{\phi_\e}(\bar{s}_n)|_{L^2}^2 +
 |u^0_{\phi}(s) - u^0_{\phi}(\bar{s}_n)|_{L^2} ^2 \big\}\, ds\Big] \Big)^{\frac{1}{2}}
\nonumber \\
&\qquad \times
\sqrt{2({K}_0 + {K}_1N)}\;  \Big( \EX  \int_0^T |\phi_\e(s)-\phi(s)|_0^2\, ds \Big)^{\frac{1}{2}}
\leq \bar{C}_1 \;  2^{-\frac{n}{4}}.
\end{align}
A similar computation based on the Lipschitz condition {\bf (C)(ii)}  and Lemma \ref{timeincrement} yields  for
some constant $\bar{C}_3:=C(T,M,N)$ and any   $\e \in ]0, \e_0]$
\begin{align} \label{eqT2}
 \tilde{ T}_3 (N,n,\e) &\leq  \sqrt{2N L_1} \Big( \EX \Big[ 1_{G_{N,\e}(T)}  \int_0^T\!\!
 |u^0_{\phi}(s) - u^0_{\phi}(\bar{s}_n)|_{L^2}^2 \, ds\Big]  \Big)^{\frac{1}{2}} \Big(
 \EX \int_0^T \!\! |\phi_\e(s)-\phi(s)|_0^2  ds \Big)^{\frac{1}{2}}
 \nonumber \\
& \leq \bar{C}_3 \;  2^{-\frac{n}{4}}.
\end{align}
The time H\"older regularity {\bf (C') (ii)}  on $\s(.,u)$ and the Cauchy-Schwarz inequality imply:
\begin{equation}  \label{eqHolder}
\tilde{T}_2(N,n,\e)\leq C \, \sqrt{N} \, 2^{-n\gamma}\, \EX\Big(
1_{G_{N,\e}(T)} \int_0^T\left(1+\| u^0_\phi(s)\|_{1,0}\right)  |\phi_\e(s)-\phi(s)|_0\, ds \Big)
\leq \bar{C}_2  2^{-n\gamma}
\end{equation}
for some constant $\bar{C}_2=C(T,M,N)$.
Using the Cauchy-Schwarz  inequality and the growth condition \eqref{growth_tilde_LQ},  we deduce  for $\bar{C}_4=C(T,N,M)$
and any $\e \in ]0, \e_0]$
\begin{align} \label{eqT3}
\tilde{T}_4(N,n,\e)&\leq  \EX \Big[  1_{G_{N,\e}(T)} \sup_{1\leq k\leq 2^n}
\big({K}_0+{K}_1| u^0_\phi(t_k)|_{L^2}^2
\big)^{\frac{1}{2}} \int_{t_{k-1}}^{t_k}\!\! |\phi_\e(s)-\phi(s)|_0
\, ds \, |U_\e(t_k)|_{L^2} \Big]\nonumber \\
&\leq 2 \sqrt{ N(K_0+K_1N)}\;  \EX\Big( \sup_{1\leq k\leq 2^n}
 \int_{t_{k-1}}^{t_k} |\phi_\e(s)-\phi(s)|_0 \, ds\Big)
\leq 4  \bar{C}_4\;  2^{-\frac{n}{2}}.
\end{align}
Finally, note that the weak convergence of $\phi_\e$ to $\phi$ implies that for any $a,b\in [0,T]$, $a<b$,
as $\e \to 0$ the integral $\int_a^b \phi_\e(s) ds $ converges to $ \int_a^b \phi(s) ds$ in the weak topology of $H_0$.
Therefore, since for
the operator $\sigma(t_k, u^0_\phi(t_k))$ is compact from $H_0$ to $H$, we deduce that
for every $k$,
\[
\Big| \sigma(t_k, u^0_\phi(t_k)) \Big(   \int_{t_{k-1}}^{t_k} \phi_\e(s) ds - \int_{t_{k-1}}^{t_k} \phi(s) ds
  \Big) \Big|_{L^2}  \to 0~~\mbox{ as }~~\e \to 0.
\]
Hence a.s.,  for fixed $n$ as $\e \to 0$, $\bar{T}_5 (N,n,\e,\omega) \to 0$. Furthermore,
 $\bar{T}_5(N,n,\e,\omega)
\leq C(K_0,K_1,N, M)$ and hence the dominated convergence theorem proves that for any
fixed $n,N$, $\EX(\bar{T}_5(N,n,\e))\to 0$ as $\e \to 0$.
\par
Thus, \eqref{t3}--\eqref{eqT3} imply that for any fixed $N\geq 1$  and any integer $n\geq 1$
\begin{equation*}
\limsup_{\e\to 0}\EX \Big[ 1_{G_{N,\e}(T)} \sup_{0\leq t\leq T} |T_3(t,\e)|\Big] \leq
C_{N,T,M}\;  2^{-n(\gamma \wedge \frac{1}{4} )}.
\end{equation*}
Since $n$ is arbitrary, this yields for any integer $N\geq 1$:
\begin{equation*}
\lim_{\e\to 0}\EX \Big[ 1_{G_{N,\e}(T)} \sup_{0\leq t\leq T} |T_3(t,\e)|\Big] =0 .
\end{equation*}
Therefore from \eqref{*} and \eqref{lambda1} we obtain \eqref{cv1}.
By the Markov inequality
\[
\PP(\|U_\e\|_Y  > \de )  \leq  \PP(G_{N,\e}(T)^c )+ \frac{1}{\delta^2}
 \EX\Big( 1_{G_{N,\e}(T)} \|U_\e\|^2_Y\Big)
~~\mbox{ for any }~~ \de>0.
\]
Finally,  \eqref{G-c} and \eqref{cv1} yield that for any integer $N\geq 1$,
\[
\limsup_{\e\to 0} \PP(\|U_\e\|_Y  > \de )\le  C(T,M) N^{-1},
%  ~~\mbox{ for any }~~ N\ge 1
\]
for some constant $C(T,M)$ which does not depend on $N$.
This implies  $\lim_{\e\to 0} \PP(\|U_\e\|_Y  > \de )=0$ for any $\delta>0$,
which  concludes the proof of  Proposition \ref{weakconv}.  \hfill $\Box$
%\end{proof}

 \subsection{Proof of the compactness of the set of controlled equations (Proposition \ref{compact})} 

Recall that we want to prove that the set  $K(M)= \{u^0_\phi \in X :  \phi \in S_M \}$ is  a compact subset of $Y$. 
Let $\{u^0_n\}$ be a sequence in $K(M)$, corresponding to solutions of
(\ref{dcontrol}) with controls $\{\phi_n\}$ in $S_M$:
\begin{eqnarray*} %\label{dcontroln}
d u^0_n(t) = F(u^0_n(t) dt + \s(t,u^0_n(t)) \phi_n(t) dt, \;\;
u^0_n(0)=u_0\in {\mathcal H}^{0,1}.
\end{eqnarray*}
 Since $S_M$ is a
bounded closed subset in the Hilbert space $L^2(0, T; H_0)$, it
is weakly compact. So there exists a subsequence of $\{\phi_n\}$, still
denoted as $\{\phi_n\}$, which converges weakly to a limit $\phi$ in
$L^2(0, T; H_0)$. Note that in fact $\phi \in S_M$ as $S_M$ is
closed. We now show that the corresponding subsequence of
solutions, still denoted as  $\{u^0_n\}$, converges in $ Y$
 to $u^0_\phi$ which is the solution of the
following ``limit'' equation
\begin{eqnarray*}% \label{limiteqn}
d u^0_\phi (t) = F(u^0_\phi(t)) dt + \s(t, u^0_\phi(t)) \phi(t) dt, \;\;
u(0)=u_0.
\end{eqnarray*}
This will complete the proof of the compactness of $K(M)$.
  To ease notation
 we will often drop the time parameters $s$, $t$, ... in the equations and integrals.
\par

Let $U_n=u^0_n-u^0_\phi$; using \eqref{upper_F-F_t} with $\eta\in (0,\nu)$, Condition {\bf (C)}  and Young's inequality,
we deduce   for $t\in [0, T]$:
\begin{align}
& |U_n(t)|_{L^2}^2 +2\eta  \int_0^t\!\!  |\nabla_h U_n(s)|_{L^2}^2  ds \leq  2C_\eta \int_0^t \|u^0_\phi(s) \|_{1,1}^2 |U_n(s)|_{L^2}^2 ds \nonumber \\
%\quad -2\int_0^t \!\! \big\langle B(u_n(s))-B(u(s)) , U_n(s)\big\rangle\, ds
% - 2\int_0^t \big( \tilde R(s,u_{n}(s))-\tilde R(s,u_h(s)) , U_n(s)\big) ds \nonumber \\
&\quad  %2C_\eta \int_0^t \|u^0_\phi(s) \|_{1,1}^2 |U_n(s)|_{L^2}^2  
+ 2\int_0^t \Big\{ \Big( \big[ \s(s,u^0_n(s))
-\s(s,u^0_\phi(s))\big] \phi_{n}(s), U_n(s)\Big) \nonumber \\ &{}\qquad
{}\qquad{}\qquad{}\qquad {}\qquad{}\qquad{}\qquad
 + \big( \s(s,u^0_\phi(s)) \big(\phi_n(s)-\phi(s)\big)\, ,\, U_n(s)\big)\Big\} ds  \nonumber \\
& \leq  2 \int_0^t |U_n(s)|^2 \big(
C_{\eta}\|u^0_\phi(s)\|_{1,1}^2 +\sqrt{ L_1} \, |\phi_n(s)|_0\big)\, ds
\nonumber \\
&\quad
 + 2 \int_0^t \Big(  \s(s,u^0_\phi(s))\, [\phi_{n}(s)-\phi(s)]\; ,\; U_n(s)\Big) \, ds  .
\label{error1}
\end{align}
The inequality \eqref{apriori_phi} implies that there exists a finite positive constant $\bar{C}$ such that
\begin{equation}\label{est-uni}
  \sup_n \Big[ \sup_{0\leq t\leq T} \big(|u(t)|_{L^2}^2 + |u_n(t)|_{L^2}^2\big) + \int_0^T
\big( \|u^0_\phi(s)\|_{1,1}^2 +\|u^0_n(s)\|_{1,1}^2\big)ds \Big] =  \bar{C}.
\end{equation}
Thus Gronwall's lemma implies that
\begin{equation} \label{errorbound}
 \sup_{ t\leq T} |U_n(t)|^2 +2\eta \int_0^T |\nabla_h U_n(t)\|_{L^2}^2\, dt  \leq
\exp\Big(2
  \big( {C}_{\eta} \bar{C}  + \sqrt{L_1 M T}  \big)\Big)\,
  \sum_{i=1}^5 I_{n,N}^{i}  ,
\end{equation}
where,  as in the proof of  Proposition~\ref{weakconv}, we have for $t_k=kT2^{-N}$:
\begin{eqnarray*}
I_{n,N}^1&=& \int_0^T \big| \big( \s(s,u^0_\phi(s))\,  [\phi_n(s)- \phi(s)]\, ,\, U_n(s)-U_n(\bar{s}_N)\big)\big|\, ds,
\\
I_{n,N}^2&=& \int_0^T\Big|  \Big( \big[ \s(s,u^0_\phi(s)) - \s(\bar{s}_N, u^0_\phi(s)) \big]  [\phi_n(s)-\phi(s)]\, ,\,
U_n(\bar{s}_N)\Big)\Big| \, ds,  \\
I_{n,N}^3&=& \int_0^T\Big|  \Big( \big[ \s(\bar{s}_N,u^0_\phi(s)) - \s(\bar{s}_N,u^0_\phi(\bar{s}_N))\big]
  [\phi_n(s)-\phi(s)]\, ,\,
U_n(\bar{s}_N)\Big)\Big| \, ds,  \\
I_{n,N}^4&=& \sup_{1\leq k\leq 2^N} \sup_{t_{k-1}\leq t\leq t_k} \Big|\Big( \s(t_k,u^0_\phi(t_k ))
 \int_{t }^{t_k}  (\phi_\e(s)-\phi(s))
ds \; ,\; U_n(t_k) \Big)\Big| , \\
I_{n,N}^5&=& \sum_{k=1}^{2^N} \Big( \s(t_k, u^0_\phi(t_k )) \,  \int_{t_{k-1}}^{t_k }
[ \phi_n(s)-\phi(s)]\, ds   \; ,\; U_n(t_k ) \Big) .
\end{eqnarray*}
The Cauchy-Schwarz  inequality, condition {\bf (C') (i)}  and Lemma \ref{timeincrement} imply that
for some constants   $C_i$,  which depend on $M$ and $T$, but do not
depend on $n$ and $N$:
\begin{align} \label{estim1}
I_{n,N}^1 &\leq \big( {K}_0 + {K}_1 \bar{C}\big)^\frac{1}{2} 
\Big( \int_0^T \!\! \big( |u^0_n(s)-u^0_n(\bar{s}_N)|_{L^2}^2 + |u^0_\phi(s)-u^0_\phi (\bar{s}_N)|_{L^2}^2\big) ds
 \Big)^{\frac{1}{2}}
\nonumber\\
&\qquad \times  \Big( \int_0^T\!\! |\phi_n(s)-\phi(s)|_0^2 ds \Big)^{\frac{1}{2}}
\leq C_1 \; 2^{-\frac{N}{4}} \, ,\\
I_{n,N}^3 &\leq 2\sqrt{L_1 \bar{C}}  \Big(\int_0^T \!\! |u^0_\phi(s)-u^0_\phi(\bar{s}_N)|_{L^2}^2 ds \Big)^{\frac{1}{2}}
 \Big( \int_0^T \!\! |\phi_n(s)-\phi(s)|_0^2\, ds\Big)^{\frac{1}{2}} \leq C_3\;  2^{-\frac{N}{4}}\, ,
\label{estim2}\\
I_{n,N}^4 &\leq \big( {K}_0 + {K}_1 \bar{C}\big)^{\frac{1}{2}} 2\sqrt{\bar{C} }  \sup_{k=1, \cdots, 2^N} \Big( \int_{t_{k-1}}^{t_{k}} |\phi_n(s)
- \phi(s)|_0^2 ds \Big)^{\frac{1}{2}}
%\big( |u(t)|+|u_n(t)|\big)\Big]\,  2^{-\frac{N}{2}} \leq
\leq \; C_4\;  2^{-\frac{N}{2}}\,  .
\label{estim3}
\end{align}
Furthermore, condition {\bf (C')(ii)} implies that
\begin{align} \label{eqHolderdet}
I_{n,N}^2 &\leq C 2^{-N\gamma}\, \sup_{0\leq t\leq T}\big(|u^0_\phi(t)|_{L^2}+|u^0_n(t)|_{L^2}\big)
 \int_0^T (1+\|u^0_\phi(s)\|_{1,0})  (|\phi(s)|_0+
|\phi_n(s)|_0)\, ds \nonumber \\
& \leq C_2\, 2^{-N \gamma}.
\end{align}
For fixed $N$ and $k=1, \cdots, 2^N$,
 as $n\to \infty$, the weak convergence of $\phi_n$ to $\phi$ implies
that of $\int_{t_{k-1}}^{t_k} (\phi_n(s)-\phi(s))ds$ to 0 weakly in $H_0$.
 Since $\s(t_k, u^0_\phi(t_k))$ is a compact operator,
we deduce that for fixed $k$ the sequence $\s(t_k,u^0\phi(t_k)) \int_{t_{k-1}}^{t_k} (\phi_n(s)-\phi(s))ds$
converges to 0 strongly in $H$ as $n\to \infty$.
Since $\sup_{n,k}   |U_n(t_k)|\leq 2 \sqrt{\bar{C}}$, we have
 $\lim_n I_{n,N}^5=0$. Thus
 \eqref{errorbound}--% and \eqref{estim1}--
\eqref{eqHolderdet} yield  for every integer $N\geq 1$
%we have
\[
\limsup_{n \to \infty}\left\{  \sup_{ t\leq T} |U_n(t)|_{L^2}^2 +\int_0^T \|U_n(t)\|_{1,0}^2\, dt
\right\}\le C 2^{-N (\gamma \wedge \frac{1}{4})}. %,\quad N=1,2,\ldots
\]
Since $N$ is arbitrary, we deduce that $\|U_n\|_Y \to 0$ as $n\to \infty$.
This shows that every sequence in $K(M)$ has a convergent
subsequence. Hence $K(M)$ is  a sequentially relatively compact subset of $Y$.
Finally, let $\{u^0_n\}$ be a sequence of elements of $K(M)$ which converges to $v$ in $Y$.
 The above argument
shows that there exists a subsequence $\{u^0_{n_k}, k\geq 1\}$
which converges to some element $u_\phi \in K(M)$
for the same topology of $Y$.
 Hence $v=u^0_\phi$,
$K(M)$ is a closed subset of $Y$, and this completes the proof of  Proposition \ref{compact}. \hfill $\Box$
%\end{proof}

\vspace{.5cm}

%\newpage
\noindent {\bf Acknowledgements.} This   research was  started
in the fall 2016 while Annie Millet visited the University of Wyoming. 
She would like to thank this University  for the hospitality and the very pleasant  working conditions.
 Hakima Bessaih is partially supported by NSF grant DMS-1418838.\\
 The authors wish to thank anonymous referees for helpful comments and careful reading.

\end{document}